\documentclass[12pt,fleqn]{article}
\usepackage{amsfonts,amssymb,amsmath,amsthm,graphicx,color}
\usepackage{overpic,paralist,stmaryrd,float,hyperref,lineno}

\newcommand*{\medcup}{\mathbin{\scalebox{1.5}{\ensuremath{\cup}}}}%

\newtheorem{thm}{Theorem}
\newtheorem{lem}{Lemma}

\theoremstyle{definition}
\newtheorem{defn}{Definition}

\author{David Krieg and Elias Wegert}

\title{Rigidity of Circle Packings with Crosscuts}

\begin{document}

\begin{center}
{\LARGE Rigidity of Circle Packings with Crosscuts}
\vspace{2ex}\par
{\Large David Krieg%
\footnote{Supported by S{\"a}chsisches Landesgraduiertenstipendium}
and Elias Wegert%
\footnote{Supported by the Deutsche Foschungsgemeinschaft, grant We\,1704/8-2}}
\vspace{1ex}\par
\today
\end{center}
\vspace{2ex}\par

\begin{abstract}

\noindent
Circle packings with specified patterns of tangencies form a discrete
counterpart of analytic functions. In this paper we study univalent
packings (with a combinatorial closed disk as tangent graph)
which are embedded in (or fill) a bounded, simply connected domain.
We introduce the concept of crosscuts and investigate the rigidity of
circle packings with respect to maximal crosscuts. The main result
is a discrete version of an indentity theorem for analytic functions
(in the spirit of Schwarz' Lemma), which has implications to uniqueness
statements for discrete conformal mappings.

\end{abstract}


{\bf keywords:}
circle packing, crosscut, prime ends, conformal mapping, Schwarz's lemma,
Apollonian packing
\bigskip

\section{Introduction} \label{sec.Intro}

The study of circle packings, as they are understood in this paper,
was initiated by Paul Koebe as early as in 1936 in the context of
conformal mapping, but the real success of the topic begun with William
Thurston's talk at the celebration of the proof of the Bieberbach
conjecture in 1985. The publication of Ken Stephenson's book \cite{SteBook}
inspired further research and made the topic accessible to a wide
audience.
Since then many classical results in complex analysis found their
discrete counterpart in circle packing.

In this paper we consider circle packings embedded in a bounded, simply
connected domain. We introduce the concept of crosscuts for domain-filling
circle packings, and study the rigidity of packings with respect
to maximal crosscuts (for definitions see below). The main result is
a discrete version of an indentity theorem for analytic functions,
which has implications to uniqueness results for boundary value problems
for circle packings, and especially to discrete conformal mappings.

To be more specific, we recall that the tangency relations of a circle
packing are encoded in a \emph{2-dimensional simplicial complex}
$K$, referred to as the combinatorics of the packing. In this paper it is
assumed that $K$ is a finite triangulation of a topological disk.

Circle packings are a mixture of flexibilty and rigidity.
Counting the degrees of freedom for the centers and the radii, and comparing
this with the number of conditions caused by the tangency relations, we see
that the first number exceeds the latter by $m+3$, where $m$ is the number of
boundary circles.
In fact, the set of all circle packings for a fixed complex
$K$ forms a smooth manifold of dimension $m+3$ (Bauer et al.~\cite{BauSteWgt}).
\label{NrOfBndVert}

So the question arises which sort of conditions are appropriate to eliminate
the flexibility of a packing and make it rigid. Motivated by our work on
nonlinear Riemann-Hilbert problems, we are interested in \emph{boundary
value problems} for circle packings. These problems involve $m$ boundary
conditions (one for each boundary circle) and three additional conditions,
which can be imposed in different form on boundary circles and interior circles
as well.

A standard boundary value problem of this kind consists in finding circle
packings with (given combinatorics and) prescribed radii of its boundary
circles.
Somewhat surprisingly, this problem has always a \emph{locally univalent}
solution,
and the solution is unique up to a rigid motion of the complete packing
(see \cite{SteBook}, Section~11.4, for details).

The existence of solutions is also known for a related more general problem,
the \emph{discrete Beurling problem}, where the radii of the boundary circles
are prescribed as functions of their centers (see~\cite{WgtRotKra}), but the
question of uniqueness has not yet been answered satisfactorily.

Last but not least there are several approaches to \emph{discrete conformal
mapping} via circle packing which fall into this category
(see Stephenson~\cite{SteBook}, in particular Chap.~19 and 20, with many
interesting comments on the history of this topic, also summarizing \cite{He},
\cite{RodSul}, \cite{Thur}).

In our favorite setting of discrete conformal mapping, the domain packing
$\mathcal{P}$ is a so-called \emph{maximal packing}, which `fills' the
complex unit disk $\mathbb{D}$, while the range packing $\mathcal{P}'$ is
required to `fill' a bounded, simply connected domain $G$.
That a packing `fills a domain $G$' basically means that all its circles lie
in the closure $\overline{G}$ of that domain and all its boundary circles
intersect (touch) the boundary $\partial G$ of $G$. For domains which are
not Jordan this has to be complemented by a more subtle condition (see
Definition~\ref{def.Pack.Fill}).

\smallskip\par

In a series of papers, Oded Schramm proved several outstanding results about
packings which fill a Jordan domain. His very general existence theorems do
not only address packings of circles, but of much more general \emph{packable
sets} (for an explanation see~\cite{Schramm1}).

Surprisingly, much less is known about uniqueness. It is clear that uniqueness
of a domain-filling (circle) packing can only be expected if one imposes
additional conditions which eliminate the (three) remaining degrees of freedom.
Whether this works depends on the type of normalization conditions and on
the geometry of the domain. For example, in his uniqueness proofs,
Schramm needs that the Jordan domain is (as he says) \emph{decent}
(see~\cite{Schramm2}).

This paper is devoted to the question which \emph{additional conditions}
are appropriate to make a \emph{domain filling} circle packing \emph{unique}.
In analogy to the standard normalization of conformal mappings, it seems
reasonable to fix the center of a distinguished circle (the so-called
\emph{alpha-circle}) at some point in $G$ and to require that the center
of a neighboring circle lies on a given ray emerging from that point.
Keeping the first condition, we have chosen another setting for the second
one. This condition, involving crosscuts, is non-standard, more flexible and
allows one to address other uniqueness problems too.
\smallskip\par

In order to give the reader a flavor of the result, we first state an analogous
theorem for analytic functions.
Recall that a \emph{crosscut} of a domain $G$ in the complex plane $\mathbb{C}$
is an open Jordan arc $J$ in $G$ such that $\overline{J}=J \cup\{a,b\}$ with
$a,b\in \partial G$ (see Pommerenke~\cite{PomBook}). Slightly abusing terminology,
we shall also denote $\overline{J}$ as a crosscut in $G$.

\begin{thm}[Identity Theorem for Analytic Functions]
\label{thm.ContRigid}
Let $J$ be a crosscut of a simply connected domain $G$, with $G^-$ and $G^+$
denoting the (simply connected) components of $G \setminus J$.
If $f: G \rightarrow G$ is analytic, $f(z_0)=z_0$ for some $z_0\in G^+$,
and $f(G^-)\subset G^-$, then $f(z)=z$ for all $z\in G$.
\end{thm}

\begin{proof}
Let $g:G \rightarrow\mathbb{D}$ be a conformal mapping of $G$ onto the unit
disk $\mathbb{D}$ with $g(z_0)=0$. Then $g$ maps the crosscut $J$ of $G$ to
a crosscut of $\mathbb{D}$ (see~\cite{PomBook}, Prop.2.14) and the composition
$g \circ f \circ g^{-1}$ satisfies the assumptions of the lemma with
$G:=\mathbb{D}$ and $z_0:=0$. Hence it suffices to consider this special case.

Let $z_1$ be a point on $J$ with $|z_1|=\min_{z\in J} |z|$. Since $J$ is a
crosscut in $\mathbb{D}$, and $0=z_0\in G^+$, we have
\[
0<|z_1|\le \min\left\{|z|: z\in \overline{G^-}\right\} <1.
\]
By continuity, $f(G^-)\subset G^-$ and $z_1\in \overline{G^-}$ imply that
$f(z_1)\in \overline{G^-}$, and hence $|f(z_1)|\ge |z_1|$. Invoking Schwarz'
Lemma, we get $f(z)=cz$ in $\mathbb{D}$, where $c$ is a unimodular
constant. Finally, the only rotation of $\mathbb{D}$ which maps $G^-$
into itself is the identity.
\end{proof}

Although Schwarz' Lemma has already been investigated in the framework of
circle packing (see~\cite{Rodin}, or~\cite{PomBook} Chap.~13) the following
interpretation of Theorem~\ref{thm.ContRigid} is new.
Though precise definitions will be deferred to the next section, we hope that
Figure~\ref{fig.RigidPack} helps to get an intuitive understanding of the
setting.

\begin{figure}[H]
\begin{center}
\begin{overpic}{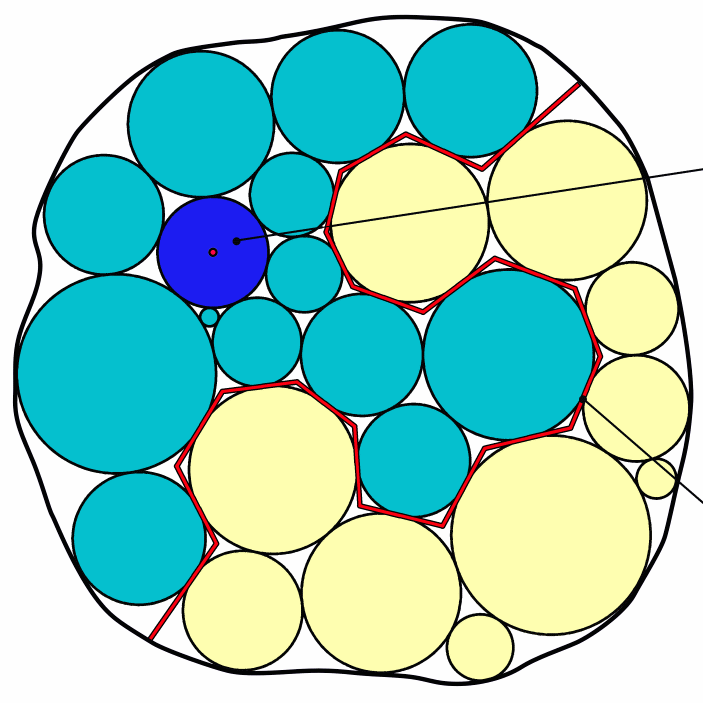}
\fontsize{12pt}{14pt}\selectfont
\put(109,25){\makebox(0,0)[cc]{$J$}}
\put(110,75){\makebox(0,0)[cc]{$D_\alpha$}}
\end{overpic}
\hspace{0.05\textwidth}
\includegraphics{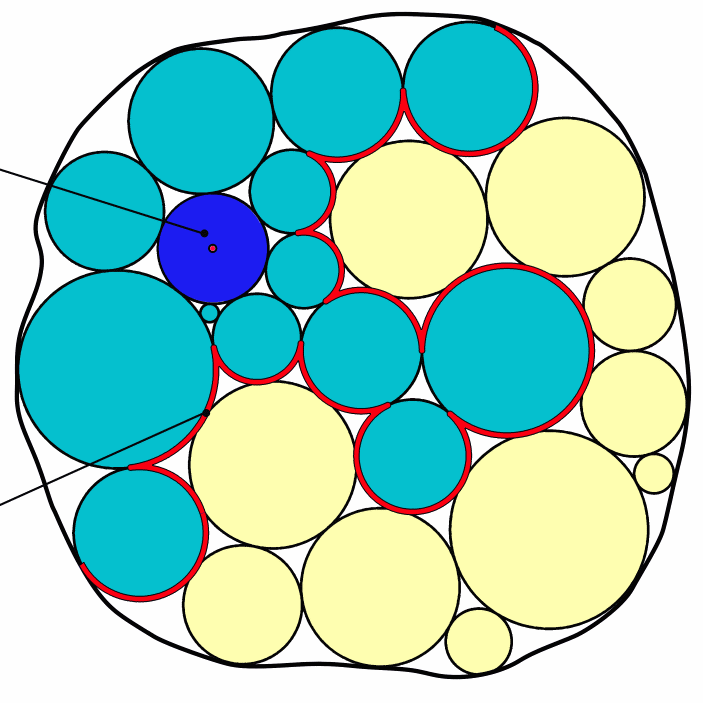}
\caption{A domain-filling packing $\mathcal{P}$ with a crosscut and
a maximal crosscut}
\label{fig.RigidPack}
\end{center}
\end{figure}

\noindent
\begin{thm}[Rigidity of Circle Packings with Crosscuts]
\label{thm.CircRigid}
Assume that a univalent circle packing $\mathcal{P}=\{D_v\}$ for a complex $K$
with vertex set $V$ fills a bounded, simply connected domain $G$.\label{def.G}
Let $J$ be a (maximal) crosscut of $\mathcal{P}$ in $G$, such that $G^-$ is a
simply connected component of $G\setminus J$, and denote by $V^-$ and $V^+$ the
sets of vertices of $K$ associated with circles in $G^-$ and $G^+:=G \setminus
\overline{G^-}$, respectively. Let $D_\alpha$ be an interior circle
of $\mathcal{P}$ which is contained in $G^+$.

Assume further that a second univalent packing $\mathcal{P}'=\{D_v'\}$ for $K$ is
contained in $G$, such that $D_\alpha$ and $D_\alpha'$ have the same center,
and $D'_v\subset G^-$ for all $v\in V^-$. Then $D_v'=D_v$ for all \emph{accessible}
vertices $v\in V$.
\end{thm}

\noindent
We point out that everything hinges on the assumption about the common center
of the two alpha-circles. Since we do not assume that $\mathcal{P}'$ fills $G$,
it is solely this condition which prevents that $\mathcal{P}'$ can be completely
moved into $G^-$.
\smallskip\par

The notion of accessible vertices will be explicated in Definition~\ref{def.Access}.
Here we only note that \emph{all vertices} $v\in V$ \emph{are accessible} if and
only if the complex $K$ is \emph{strongly connected}, which means $K$ satisfies the
following conditions
(i) and (ii):
\begin{itemize} \label{CondAccess}
\itemsep0mm
\item[{\rm(i)}]
Every boundary vertex has an interior neighbor.
\item[{\rm(ii)}]
The interior of $K$ is connected.
\end{itemize}
Note that some authors of the circle packing community make the general assumption
that the underlying complex $K$ is strongly connected (see~\cite{SteBook}).
For circle packings with this simpler combinatoric structure the theorem yields
\emph{complete rigidity} with respect to crosscuts, i.e., $D_v'=D_v$ for all $v\in V$.
\smallskip\par

Figure~\ref{fig.CounterEx} illustrates some effects which can be observed for
packings with general combinatorics.
The picture on the left shows an Apollonian packing $\mathcal{P}$ with four generations.
The highlighted line is a maximal crosscut, disks in the ``lower domain'' are the white
ones, disks in the ``upper domain'' are the colored ones.
The disk with the darkest color is the alpha-disk with fixed center.
The accessible disks are those which can be connected with the
alpha-disk by a chain of interior disks (see Definition~\ref{def.Access}).

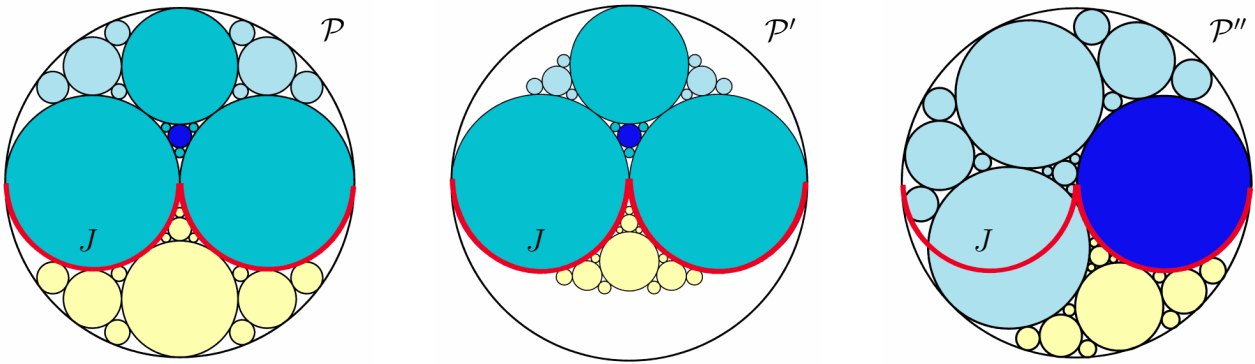
\begin{figure}[H]
\label{page.counterex}
\begin{center}
\begin{overpic}{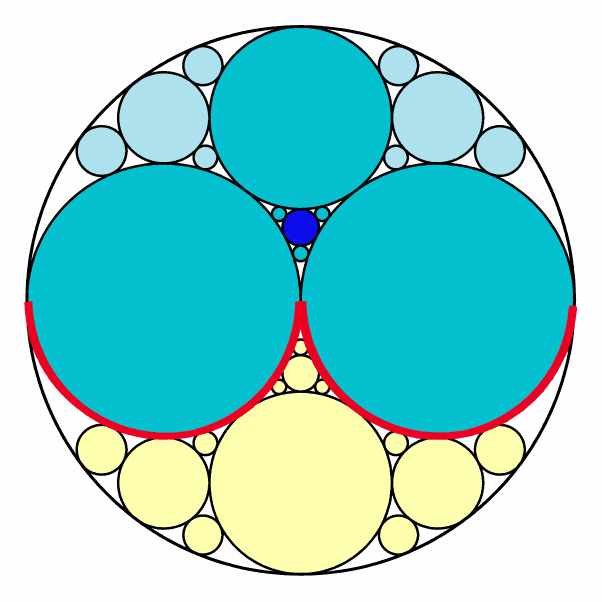}
\fontsize{12pt}{14pt}\selectfont
\put(26,35){\makebox(0,0)[cc]{$J$}}
\put(90,90){\makebox(0,0)[cc]{$\mathcal{P}$}}
\end{overpic}
\hfill
\begin{overpic}{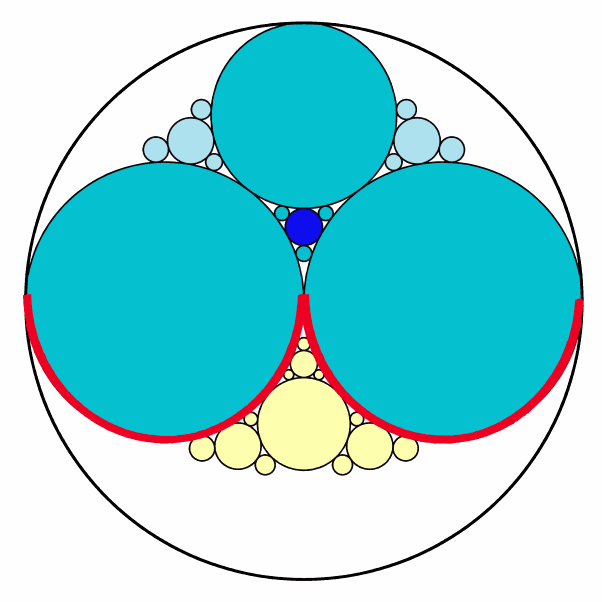}
\fontsize{12pt}{14pt}\selectfont
\put(26,35){\makebox(0,0)[cc]{$J$}}
\put(90,90){\makebox(0,0)[cc]{$\mathcal{P}'$}}
\end{overpic}
\hfill
\begin{overpic}{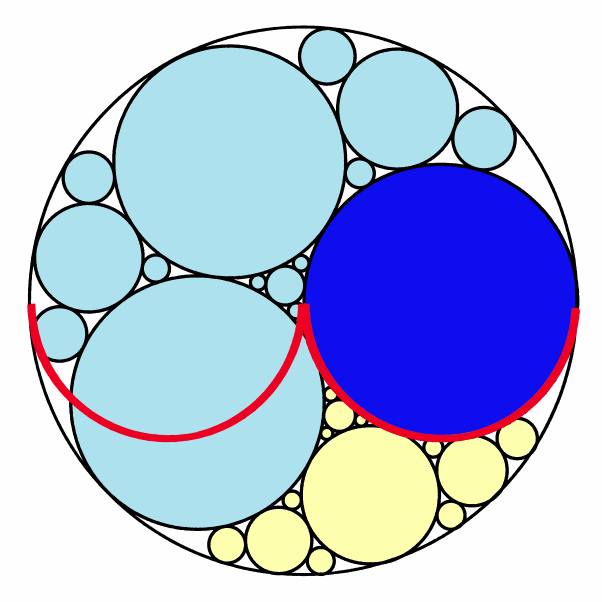}
\fontsize{12pt}{14pt}\selectfont
\put(26,35){\makebox(0,0)[cc]{$J$}}
\put(90,90){\makebox(0,0)[cc]{$\mathcal{P}''$}}
\end{overpic}
\caption{Some examples illustrating assumptions and assertions of
Theorem~\ref{thm.CircRigid}}
\label{fig.CounterEx}
\end{center}
\end{figure}

The packing $\mathcal{P}'$, depicted in the middle, satisfies the assumptions of the
theorem. In this example, only the accessible disks of $\mathcal{P}'$
(shown in darker colors) coincide with their partners in $\mathcal{P}$.
The non-accessible disks (shown in white and lighter colors) differ from the corresponding
disks in $\mathcal{P}$.

The example on the right illustrates that the result need not hold if the
alpha-disk is a boundary disk. The depicted packing $\mathcal{P}''$ satisfies
all other assumptions (for the same crosscut), but, apart from the alpha disk,
it is completely different from the packing $\mathcal{P}$ shown on the left-hand
side.
\smallskip\par

\noindent
The result has an intuitive interpretation when we think of circle packings as
dynamic structures: Suppose that $\mathcal{P}$ fills $G$, and allow its
circles to move (change position and size) in such a way, that they all remain in $G$,
the center of the alpha-circle is fixed in $G^+$, and the circles in $G^-$ are not
allowed to leave $G^-$. Then only those circles which are not accessible can be
moved, while the core part of the packing is \emph{rigid}.
\smallskip\par

In fact we shall even prove a stronger result, where the condition $D'_v\subset G^-$
need only be satisfied for $v$ in $U^-$, which stands for the set of those vertices
$v\in V^-$ associated with circles $D_v$ \emph{touching the crosscut} $J$.
\smallskip\par

\noindent
In order to illustrate the analogies with Theorem~\ref{thm.ContRigid}, we
interpret the result in the framework of discrete analytic functions:
The circle packing $\mathcal{P}$ filling $G$ is the domain packing, the packing
$\mathcal{P}'$ lies in $G$, so that $\mathcal{P}\rightarrow \mathcal{P}'$ defines
a discrete analytic function from $G$ into itself.
Fixing the centers of the alpha-circles of both packings at the same point $z_0$
corresponds to the normalization $f(z_0)=z_0$. Finally, the condition $D'_v\subset G^-$
for all $v\in V^-$ expresses the invariance of the subdomain $G^-$.

Since the packing $\mathcal{P}$ represents the identity function on $G$, it is
natural to suppose that $\mathcal{P}$ is univalent. Contrary to the continuous
setting of Theorem~\ref{thm.ContRigid}, also $\mathcal{P}'$ was assumed to be
univalent in Theorem~\ref{thm.CircRigid}. It is challenging to investigate what
happens when this condition is dropped.
\smallskip\par

\noindent
Terminological remark. For our purposes it would be better to work with
\emph{disk} packings instead of \emph{circle} packings. Though we stay with
the traditional notion, we shall often speak of the disks in a circle
packing. In order to avoid cumbersome formulations, we also say that
\emph{a circle $\partial D$ lies in a domain $G$} when this holds for the
open disk $D$ bounded by that circle. We already made use of this convention above.
\smallskip\par

\section{Circle Packings} \label{sec.CircPack}

\noindent
In order to make the paper self-contained we recall basic concepts and
notions of topology and circle packing (for details we refer to Henle~\cite{Hen}
and Stephenson~\cite{SteBook}).
\smallskip\par

\noindent
{\bf Some Geometry}.
If $A$ and $B$ are subsets of the
(complex) plane, we say that $A$ \emph{intersects} $B$ if $A\cap B\not=\emptyset$.
If $A$ is a disk, then the phrase $A$ \emph{touches} $B$ is in general used when
$\overline{A}\cap\overline{B}\not =\emptyset$ and $A\cap B=\emptyset$.
In this case we also say that the circle $\partial A$ touches $B$. As usual,
the symbol $\partial $ denotes the boundary operator.
\label{page.bndop}
\smallskip\par

By a \emph{curve} $\gamma$ we understand the image of a continuous mapping
$\varphi:[a,b]\rightarrow \mathbb{C}$. The points $\varphi(a)$ and $\varphi(b)$
are said to be the \emph{initial point} and the \emph{terminal point} of $\gamma$,
respectively; both are referred to as \emph{endpoints} of $\gamma$.
A \emph{Jordan arc} and a \emph{Jordan curve} are the homeomorphic images
of a segment and a circle, respectively. By an \emph{open Jordan arc} we mean a
Jordan arc without its endpoints.

Let $J$ be
an \emph{oriented} Jordan curve. For $p,q\in J$
with $p\not= q$ we denote by $J(p,q)$ the (oriented) open subarc of $J$ with
initial point $p$ and terminal point $q$.
If $p,q,r$ are three pairwise different points on $J$, we say that $q$
\emph{lies between $p$ and $r$ on $J$} if $q\in J(p,r)$.  Corresponding to
whether $q$ lies between $p$ and $r$, or $q$ lies between $r$ and $p$, the
\emph{orientation of the triplet $(p,q,r)$ with respect to $J$} is said to be
\emph{positive} or \emph{negative}, respectively.
\smallskip\par

Let $G$ be a bounded, simply connected domain in $\mathbb{C}$. A conformal mapping
$g: \mathbb{D}\rightarrow G$ of $\mathbb{D}$ onto $G$ has a continuous extension to
$\overline{\mathbb{D}}$ if and only if $\partial G$ is a \emph{closed curve}, i.e., a
continuous image of the unit circle $\mathbb{T}$ (see~\cite{PomBook} Theorem~2.1).
This extension (which we again denote by $g$) is a homeomorphism between
$\overline{\mathbb{D}}$ and $\overline{G}$ if (and only if) $G$ is a
\emph{Jordan domain}, i.e., $\partial G$ is a Jordan curve (see~\cite{PomBook},
Theorem~2.6).

In general, the conformal mapping $g$ induces a one-to-one correspondence between
the points on $\mathbb{T}$ and certain
equivalence classes of open Jordan arcs $\gamma$ in $G$ with terminal point $q$ on
$\partial G$, so called \emph{prime ends}. For the details we refer to
Pommerenke~\cite{PomBook}, Chap~2, and Golusin~\cite{GolBook}, Section~2.3.

If $G$ contains a disk $D$ which touches the boundary $\partial G$ at some point
$p\in\partial D\cap\partial G$, then every Jordan arc with starting point in $D$ and
terminal point $p$ is contained in the same equivalence class. Hence there is a well
defined \emph{prime end} of $G$ \emph{associated with $p$ by $D$}.
\label{page.asspe}
\smallskip\par

\noindent
{\bf Complexes}.
The skeleton of a circle packing is a \emph{simplicial 2-complex} $K$.
Throughout this paper it is assumed that $K$ is a {\em combinatorial closed disk},
i.e., it is finite, simply connected and has a nonempty boundary.
Simply speaking of a complex, we always mean a complex of this class.
Properties of complexes which are relevant in circle packing are summarized in
Lemma~3.2 of~\cite{SteBook}.
\label{def.complex}
\smallskip\par

We denote the sets of vertices, edges and faces of $K$ by $V,E,F$, respectively.
The edge adjacent to the vertices $u$ and $v$ is denoted by $e(u,v)$ or
$\langle u,v\rangle$, where the first version stands for the \emph{non-oriented
edge}, while the second means the oriented edge from $u$ to $v$.
Similarly, a face of $K$ with vertices $u,v,w$ is written as $f(u,v,w)$
(non-oriented) or $\langle u,v,w\rangle$ (oriented), respectively.
\label{def.VEF}

Two vertices $u$ and $v$ are said to be \emph{neighbors} if they are connected
by an edge $e(u,v)$ in $E$.

For any vertex $v\in V$ we denote by $E(v)$ the set of edges adjacent to $v$. This
set is endowed with a natural cyclic (counterclockwise) ordering, so that for
$e_1,e_2\in E(v)$ definitions like $\{e\in E(v): e_1<e\le e_2\}$ make sense.
\label{page.DefEdgeOrd}
\smallskip\par

Any edge $e$ of $K$ is adjacent to one or two faces. In the first case $e$
is a \emph{boundary edge}, otherwise it is an \emph{interior edge} of $K$.
\emph{Boundary vertices} are those vertices of $K$ which are adjacent to a
boundary edge. The sets of boundary edges and boundary vertices are
denoted $\partial E$ and $\partial V$, respectively, the vertices in
$V \setminus\partial V$ are called \emph{interior vertices}.\label{def.BndVE}
We point out that a boundary vertex can be adjacent to many other boundary
vertices, and that an edge which connects two boundary vertices
need not be a boundary edge (cf.~Figure~\ref{fig.Flower}, left).
\smallskip\par

We let $B(v)$ denote the smallest sub-complex of $K$ which contains a vertex $v$
and all its neighbors. If $v$ is an interior vertex $B(v)$ is said to be the
\emph{flower} of $v$, if $v$ is a boundary vertex we speak of an \emph{incomplete
flower}.\label{def.flower}
Note that $B(v)$ need \emph{not} contain all edges which connect neighbors
of $v$ (see Figure~\ref{fig.Flower}).
\smallskip\par

Since $K$ is a triangulation with non-void boundary, it must have at least
three boundary vertices. The natural cyclic ordering of boundary edges, corresponding
to the orientation of the boundary of the triangulated surface, induces a cyclic
ordering of the boundary vertices.
With respect to this ordering, any boundary vertex has a \emph{precursor}
and a \emph{successor} which are well-defined.
\smallskip\par

Speaking of a \emph{chain}, we mean a finite sequence $(c_1,\ldots,c_n)$
of vertices, edges or faces,
such that neighboring elements $c_j$ and $c_{j+1}$ are adjacent to a common edge
(if the $c_j$ are vertices or faces) or a common vertex (if the $c_j$ are edges),
respectively.

\begin{figure}[H]
\begin{center}
\includegraphics{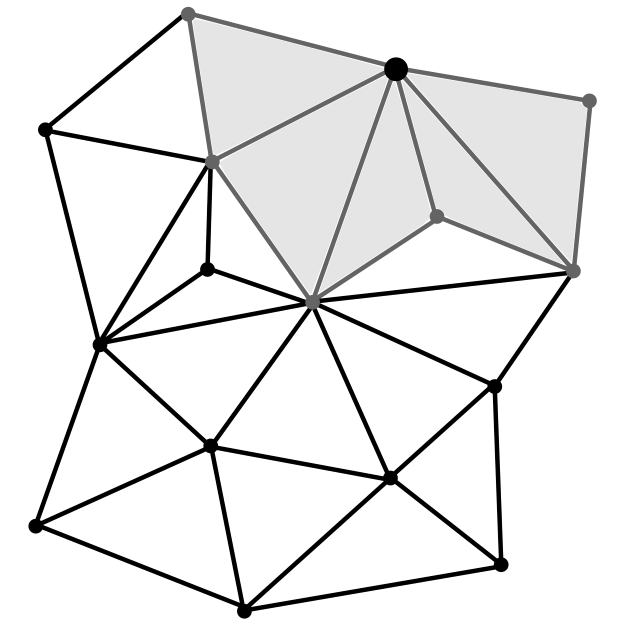}
\hfill
\includegraphics{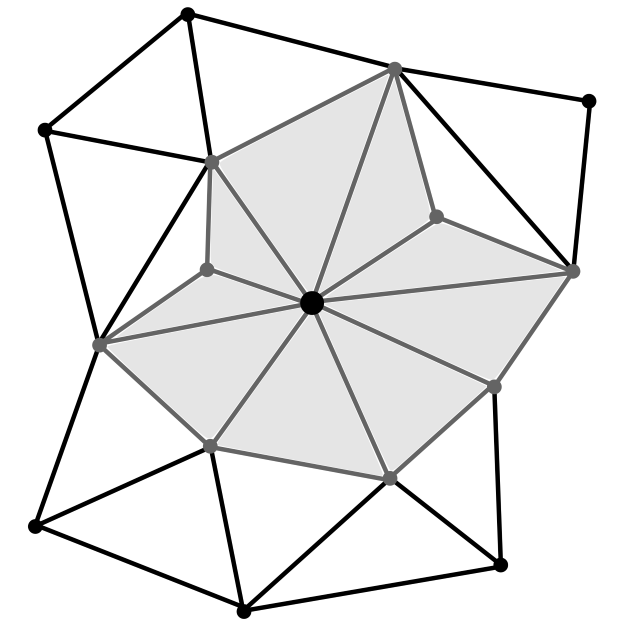}
\hfill
\includegraphics{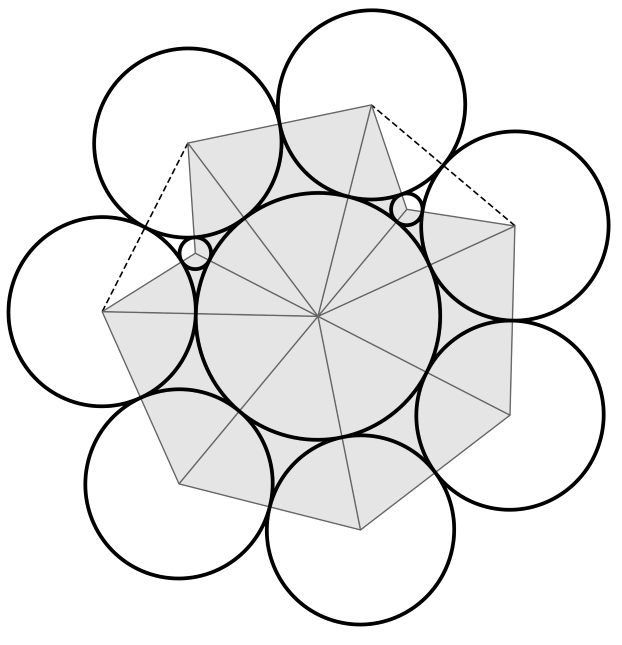}
\caption{The sub-complex of a (incomplete) flower and a corresponding packing}
\label{fig.Flower}
\end{center}
\end{figure}
\vspace{-4mm}
\noindent

\smallskip\par

On page~\pageref{page.counterex} we have illustrated some limitations of
Theorem~\ref{thm.CircRigid}. The reason for the observed effects is the relative
independence of some substructures from the rest of the packing. This is
described more precisely in the following definition.

\begin{defn}\label{def.Access}
Let $K$ be a complex with a distinguished interior vertex, the \emph{alpha-vertex}
$v_\alpha$. Then a vertex $v\in V$ is called \emph{accessible} (from $v_\alpha$)
if there is a chain of
vertices $(v,v_1,\ldots,v_n,v_\alpha)$ such that $v_1,\ldots,v_n$ are interior
vertices. The set of all accessible vertices of $K$ is denoted by $V^*$, the set of
all edges $e(u,v)\in E$ with $u,v\in V^*$ by $E^*$, and the set of all faces
$f(u,v,w)\in F$ with $u,v,w\in V^*$ by $F^*$. The \emph{kernel} $K^*$ of $K$ is
defined as the simplicial-2-complex arising from $V^*,E^*,F^*$, that is
$K^*(V^*,E^*,F^*)\subset K$.
\end{defn}

Recall that a complex $K$ is \emph{strongly connected}, if the interior of
$K$ is connected, and every boundary vertex has an interior neighbor.
The following lemma establishes a relation between this property and
accessible vertices, as already stated on page~\pageref{page.counterex}.

\begin{lem}\label{lem.solid1}
Let $K$ be a complex with a distinguished interior alpha-vertex $v_\alpha$.
All vertices of $K$ are accessible, i.e., $K=K^*$, if and only if $K$ is
strongly connected.
\end{lem}

\begin{proof}
Assume that all vertices of $K$ are accessible. Let $v\neq w$ be two interior
vertices of $K$.
Since $v$ and $w$ are both accessible, there are two chains of vertices
$(v,v_1,\ldots,v_i,v_\alpha)$ and $(w,w_1,\ldots,w_j,v_\alpha)$ such that
$v_1,\ldots,v_i$,$w_1,\ldots,w_j$ are interior vertices, hence the interior
of $K$ is connected. Let $u$ be a boundary vertex of $K$. Because $u$ is
accessible, there is a chain of vertices $(u,u_1,\ldots,u_n,v_\alpha)$ such
that $u_1,\ldots,u_n$ are interior vertices, hence every boundary vertex of
$K$ has an interior neighbor, and $K$ is strongly connected.

Conversely, assume that $K$ is strongly connected. Let $v$ be an \emph{interior
vertex} of $K$. Since the interior of $K$ is connected, we  find a chain
of vertices $(v,v_1,\ldots,v_n,v_\alpha)$ such that $v_1,\ldots,v_n$ are
interior vertices. If $w$ is a \emph{boundary vertex} of $K$, it has some
neighboring interior vertex $u$, and there exists a chain of interior
vertices $(u,u_1,\ldots,u_n,v_\alpha)$ from $u$ to $v_\alpha$.
Then the chain $(w,u,u_1,\ldots,u_n,v_\alpha)$ connects $w$ with $v_\alpha$.
Hence any vertex of $K$ is accessible.
\end{proof}

Now, all vertices of the kernel $K^*$ are accessible per definition, so if we
can show that $K^*$ is a complex, then it is also strongly connected. In order
to prove this we provide the following two lemmas.

\begin{lem}\label{lem.kernel1}
Let $K^*$ be the kernel of a complex $K$, and let $V^*,V$ be their vertex sets,
respectively. Then every vertex $v\in\partial V\cap V^*\neq\emptyset$ has
exactly two other vertices of $\partial V\cap V^*$ as neighbors in $K^*$.
Moreover all accessible neighbors of $v$ form an incomplete flower
$B^*(v)\subset K^*$ around $v$.
\end{lem}

\begin{proof}
Since all neighbors of accessible interior vertices of $K$ are accessible, we have
$\partial V\cap V^*\neq\emptyset$.
Let $v\in\partial V\cap V^*$ be such a (boundary) vertex. Because $v$ is accessible
there must be a neighbor $u$ of $v$ in $K$, which is an accessible interior vertex
in $K$, $u\in V^*\setminus\partial V$. Looking at the incomplete flower $B(v)$ of $v$
in $K$ it becomes clear that there must be an edge chain $C$ in $B(v)$, connecting two
different boundary vertices $w_1,w_2\in\partial V$, such that $C$ contains $u$ and
no other boundary vertices of $K$ except $w_1,w_2$. Hence $w_1$ and $w_2$ are accessible,
$w_1,w_2\in V^*$, and $v$ has at least two other boundary vertices of
$\partial V\cap V^*$ as neighbors.

Assume now that there is a third boundary vertex $w_3\in\partial V$, different
from $w_1$ and $w_2$, which is an accessible neighbor of $v$. Let $C_1,C_2\subset C$
be the chains of vertices connecting $u$ with $w_1,w_2$, respectively.
Since $v$ and $w_3$ are accessible, there are chains $(v,u,c_1,...,c_i,v_\alpha)$
and $(w_3,c_1',...,c_j',v_\alpha)$, such that $c_1,...,c_i,c_1',...,c_j'$ are
accessible interior vertices of $K$. The concatenation
$C_3:=(v,u,c_1,...,c_i,v_\alpha,c_j',...,c_1',w_3,v)$ must encircle either $w_1$
or $w_2$ (see Figure~\ref{fig.Connected} left), which is impossible because both
are boundary vertices. Hence, every boundary vertex $v\in\partial V\cap V^*$
has exactly two other boundary vertices of $\partial V\cap V^*$ as neighbors,
and the (incomplete) flower $B^*(v)$ around $v$ with respect to $K^*$ has the
structure depicted in Figure~\ref{fig.Connected} (middle), with
$\{v_1,...,v_n\}\in V^*\setminus\partial V$ and $n\geq 1$.
\end{proof}

\begin{lem}\label{lem.kernel2}
The kernel $K^*$ of a complex $K$ is a triangulation.
\end{lem}

\begin{proof}
Let $K^*(V^*,E^*,F^*)$ be the kernel of a complex $K$. Lemma~3.2 of~\cite{SteBook}
tells us that $K^*$ is a triangulation if and only if it has the following
properties (i)--(vi) --- what we will check on the run.

(i) \emph{$K^*$ must be connected.} As we already used above, every two vertices
$v,w\in V^*$ can be connected by a chain of accessible vertices via the
alpha-vertex.

(ii) \emph{Every edge of $E^*$ must belong to either one or two faces of $F^*$.}
Since $E^*\subset E$ and $F^*\subset F$, it is impossible that an edge of $E^*$ belongs
to more than two faces of $F^*$. So it remains to show that there is
no isolated edge, which does not belong to a face. Let $e\in E^*$ be an edge with
$e=e(v,w)$, that is $v,w\in V^*$. If $v$ is an interior vertex of $K$, then
all its neighbors are accessible, too, so $B(v)\subset K^*$. If $v$ is a boundary
vertex of $K$, than all its accessible neighbors form an incomplete flower
$B^*(v)\subset K^*$ around $v$ (Lemma~\ref{lem.kernel1}). In both cases $e$ is
contained in at least one face of $F^*$.

(iii) \emph{Every vertex $v$ of $K^*$ belongs to at most finitely many faces, and these
form an ordered chain in which each face shares an edge from $v$ with the next.}
The first assertion holds, because $K^*$ is a subset of $K$. The second part follows
easily by considering the flower $B(v)$ (if $v$ is an interior vertex) or
the imcomplete flower $B^*(v)$ (if $v$ is a boundary vertex).

(iv) \emph{Every vertex $v$ of $K^*$ belongs either to no boundary edge, or to exactly
two boundary edges.}
Using once more the flower around $v$ immediately shows this property.

(v),(vi) \emph{Any two faces are either disjoint, share a single vertex, or share a single
edge, and all of them are properly oriented.}
This follows directly from $K^*\subset K$.
\end{proof}

\begin{figure}[H]
\begin{center}
\begin{overpic}{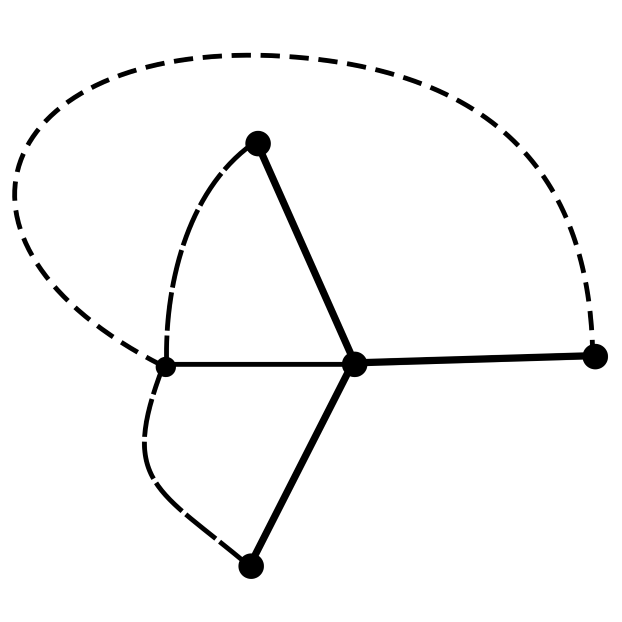}
\fontsize{12pt}{14pt}\selectfont
\put(60,47){\makebox(0,0)[cc]{$v$}}
\put(46,81){\makebox(0,0)[cc]{$w_1$}}
\put(47,6){\makebox(0,0)[cc]{$w_2$}}
\put(96,36){\makebox(0,0)[cc]{$w_3$}}
\put(30,37){\makebox(0,0)[cc]{$u$}}
\put(35,60){\makebox(0,0)[cc]{$C_1$}}
\put(30,25){\makebox(0,0)[cc]{$C_2$}}
\put(77,88){\makebox(0,0)[cc]{$C_3$}}
\end{overpic}
\hspace{0.02\textwidth}
\begin{overpic}{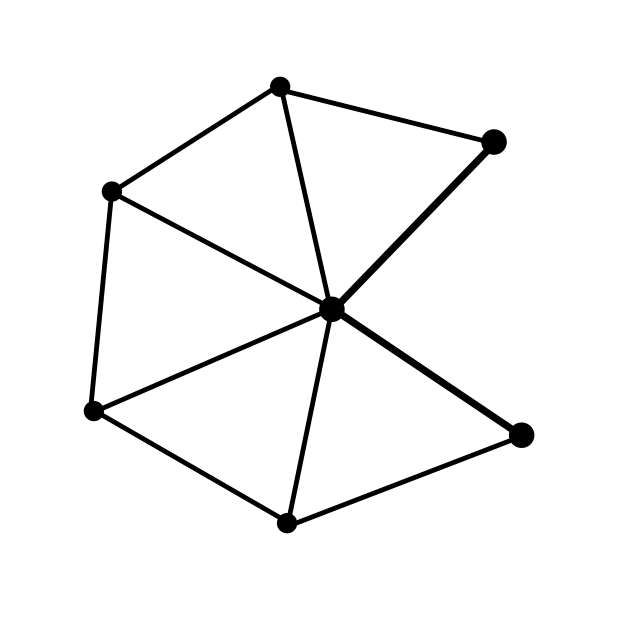}
\fontsize{12pt}{14pt}\selectfont
\put(61,50){\makebox(0,0)[cc]{$v$}}
\put(91,32){\makebox(0,0)[cc]{$w_1$}}
\put(87,80){\makebox(0,0)[cc]{$w_2$}}
\put(45,91){\makebox(0,0)[cd]{$v_1$}}
\put(15,70){\makebox(0,0)[rc]{$v_2$}}
\put(12,33){\makebox(0,0)[rc]{$v_3$}}
\put(46,10){\makebox(0,0)[cu]{$v_4$}}
\end{overpic}
\hspace{0.02\textwidth}
\begin{overpic}{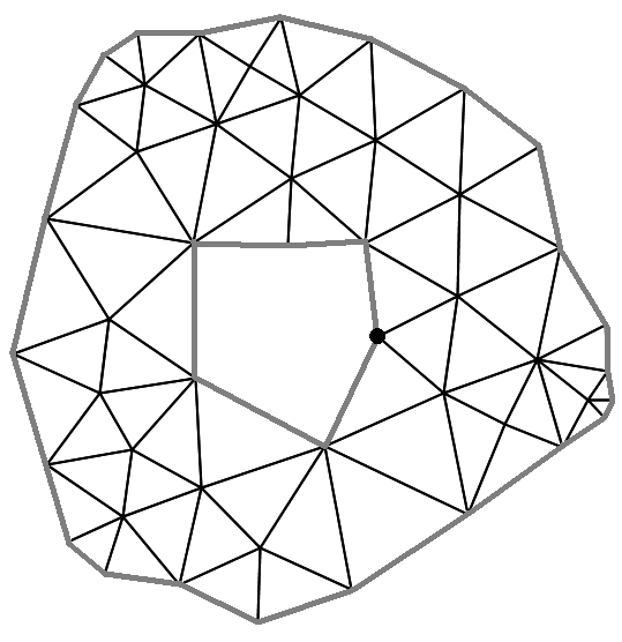}
\fontsize{12pt}{14pt}\selectfont
\put(54,49){\makebox(0,0)[cc]{$v$}}
\end{overpic}
\caption{Constructions for the proof of Lemma~\ref{lem.kernel1} and~\ref{lem.solid3}}
\label{fig.Connected}
\end{center}
\end{figure}
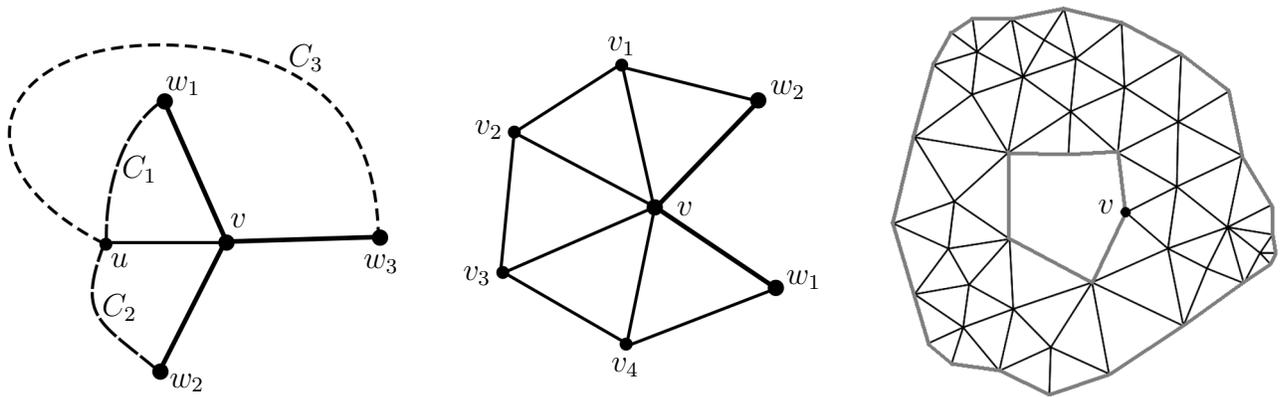
\noindent

The crucial properties of the kernel $K^*$ are summarized in the following lemma.

\begin{lem}\label{lem.solid3}
The kernel $K^*$ of a complex $K$ is a strongly connected complex with
$\partial V^*=\partial V\cap V^*$.
\end{lem}

\begin{proof}
In order to prove $\partial V^*=\partial V\cap V^*$ let $v$ be an
accessible interior vertex of $K$. By definition of accessible vertices the flower
$B(v)\subset K$ around $v$ must also lie in $K^*$, hence $v$ is an interior vertex
of $K^*$, which implies $\partial V^*=\partial V\cap V^*$.

Since $K^*$ is a finite triangulation with nonempty boundary (Lemma~\ref{lem.kernel2}),
it is a complex (in our sense) if it is simply connected. Because every boundary
vertex of $K^*$ has exactly two other boundary vertices of $K^*$ as neighbors
(Lemma~\ref{lem.kernel1}), $K^*$ is simply connected if and only if the boundary
of $K^*$ is connected.

Assume that the boundary of $K^*$ is not connected. This implies that there is a
boundary vertex $v\in\partial V^*$, which is enclosed by a closed chain of boundary
vertices different from $v$ (see Figure~\ref{fig.Connected}, right). Because $K^*$
is a subset of $K$, the vertex $v$ must be enclosed by the boundary chain of $K$.
Hence $v$ is an interior vertex of $K$, a contradiction to
$\partial V^*=\partial V\cap V^*$.

Since $K^*$ is a complex whose vertices are all accessible, Lemma~\ref{lem.solid1}
tells us that $K^*$ is strongly connected.
\end{proof}

\noindent
{\bf Circle Packings}.\label{def.Pack.Pack}
A collection $\mathcal{P}$ of disks $D_v$ is said to be a \emph{circle packing}
for the complex $K=(V,E,F)$ if it satisfies the following conditions (i)--(iii):
\begin{itemize}
\itemsep0mm
\item[\rm{(i)}]
Each vertex $v\in V$ has an associated disk $D_v\in \mathcal{P}$, such that
$\mathcal{P}=\{D_v: v\in V\}$.
\item[\rm{(ii)}]
If $\langle u,v\rangle\in E$ is an edge of $K$, then the disks $D_u$ and $D_v$
touch each other.
\item[\rm{(iii)}]
If $\langle u,v,w,\rangle\in F$ is a positively oriented face of $K$, then the
centers of the disks $D_u, D_v, D_w$ form a positively oriented triangle in the
plane.
\end{itemize}
A circle packing is called \emph{univalent}, if its disks are \emph{non-overlapping},
$D_u \cap D_v=\emptyset$ for all $u,v\in V$ with $u \not= v$. In this paper all
circle packings are assumed to be univalent.

\smallskip\par

Since the structure of the underlying complex $K$ carries over to the associated
packing $\mathcal{P}$, all related attributes can be applied to the disks $D_v$
 as well -- so we shall speak of boundary disks, interior
disks, neighboring disks, etc.

\smallskip\par

\noindent
The \emph{contact point} of two neighboring disks $D_u,D_v$ is defined by
$c(u,v):=\overline{D}_u\cap\overline{D}_v$.
The \emph{contact points of a packing} $\mathcal{P}$ for the complex $K=(V,E,F)$ are
the points $c(u,v)$ with $e(u,v)\in E$.
\label{def.contact}
\smallskip\par

We denote by $D$ the union of all disks in $\mathcal{P}$, $D:=\bigcup_{v\in V} D_v$.
If $\mathcal{P}$ is univalent and $p$ and $q$ are different points of $\partial D$,
there is at most one disk $D_v$ whose boundary $\partial D_v$ contains $p$ and $q$. If such a
disk exists, we define $\delta(p,q)$ as the positively oriented open subarc of
$\partial D_v$ from $p$ to $q$, and $\delta[p,q]:=\overline{\delta(p,q)}$.
In addition we set $\delta(p,p):=\emptyset$ and $\delta[p,p]:=\{p\}$.
Note that $\delta(p,q)$ and $\delta[q,p]$ are complementary subarcs of
$\partial D_v$, provided that $p \not=q$.\label{def.arcs}

If $\langle u,v,w,\rangle$ is a face of $K$, the \emph{interstice} $I(u,v,w)$ of
$\mathcal{P}$ is the Jordan domain bounded by the arcs
$\delta_u:=\delta\big(c(u,v),c(u,w)\big)$, $\delta_v:=\delta\big(c(v,w),c(v,u)\big)$
and $\delta_w:=\delta\big(c(w,u),c(w,v)\big)$ (see Figure~\ref{fig.PackInter}, left).\label{def.interstice}
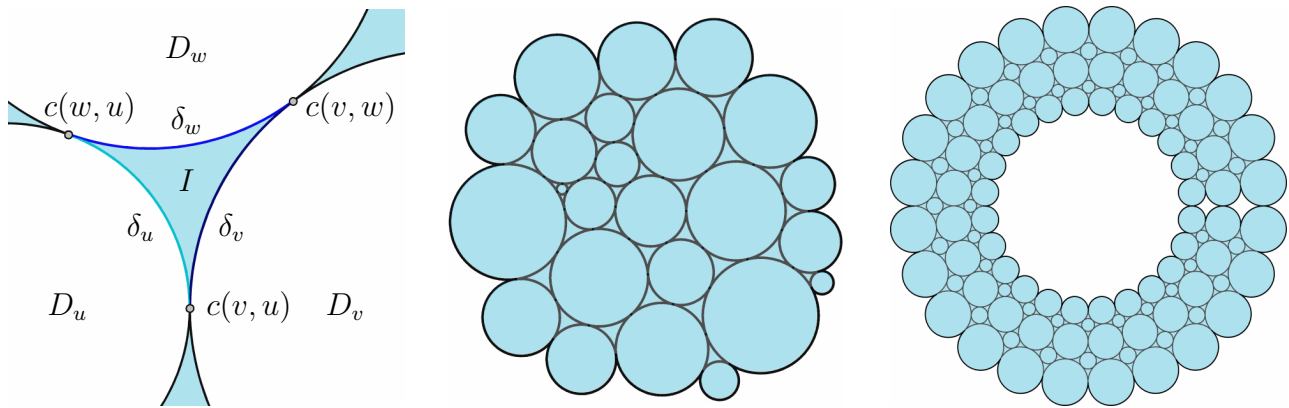
\begin{figure}[H]
\begin{center}
\begin{overpic}{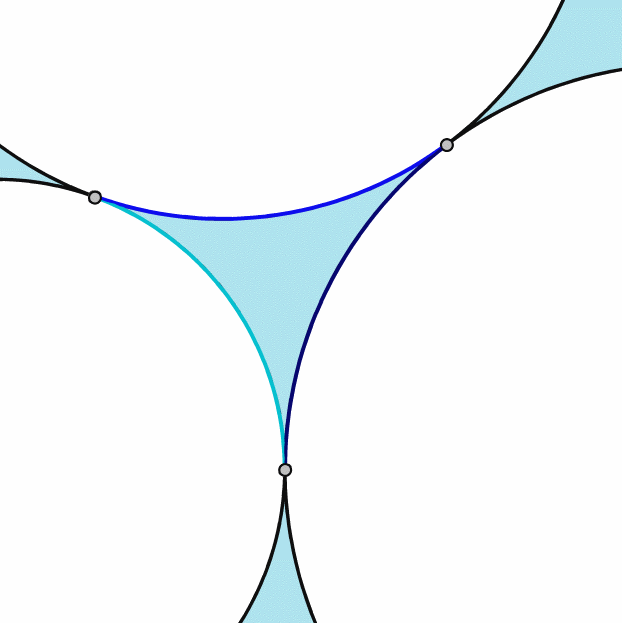}
\fontsize{12pt}{14pt}\selectfont
\put(15,25){\makebox(0,0)[cc]{$D_u$}}
\put(85,25){\makebox(0,0)[cc]{$D_v$}}
\put(45,90){\makebox(0,0)[cc]{$D_w$}}
\put(37,45){\makebox(0,0)[rc]{$\delta_u$}}
\put(53,45){\makebox(0,0)[lc]{$\delta_v$}}
\put(45,72){\makebox(0,0)[cc]{$\delta_w$}}
\put(50,25){\makebox(0,0)[lc]{$c(v,u)$}}
\put(9,75){\makebox(0,0)[lc]{$c(w,u)$}}
\put(75,75){\makebox(0,0)[lc]{$c(v,w)$}}
\put(45,57){\makebox(0,0)[cc]{$I$}}
\end{overpic}
\hfill
\includegraphics{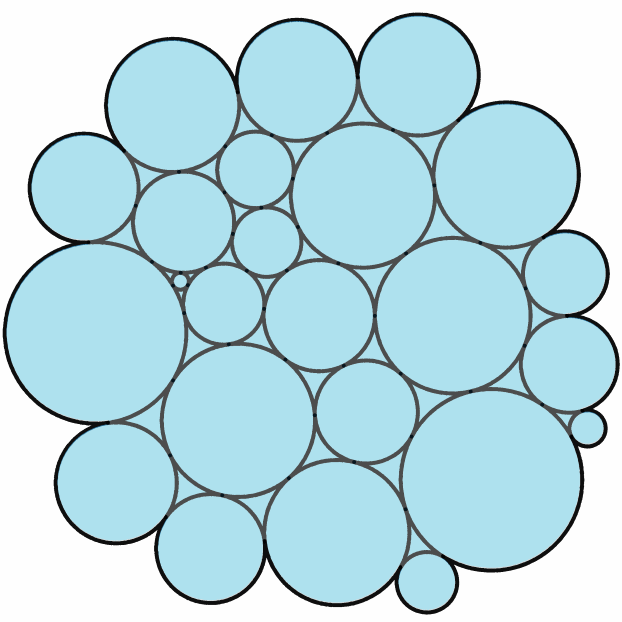}
\hfill
\includegraphics{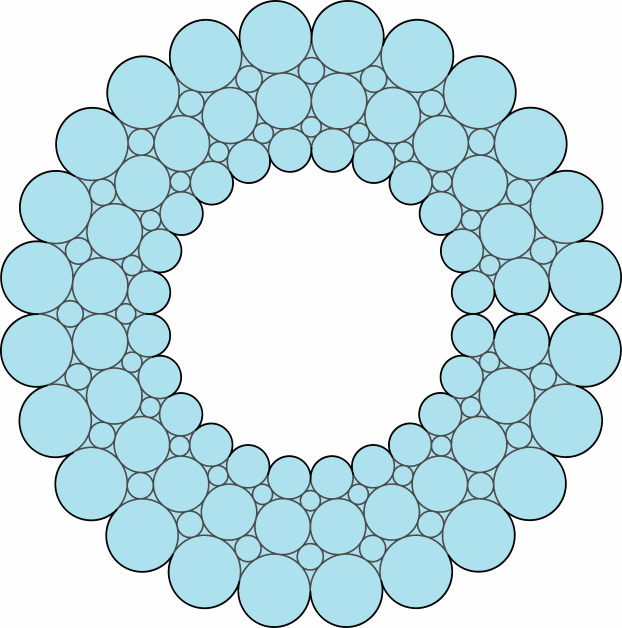}
\caption{Definition of the interstice $I:=I(u,v,w)$ and the carrier $D^*$ of two packings.}
\label{fig.PackInter}
\end{center}
\end{figure}

\noindent
Besides the union $D$ of all disks in a packing $\mathcal{P}$ we need the
\emph{carrier} of $\mathcal{P}$, which is the compact set
\[
D^*:=\overline{D} \cup \bigcup_{f(u,v,w)\in F} I(u,v,w)
\]
(see Figure~\ref{fig.PackInter}, middle and right).\label{def.carrier}
Note that this definition is somewhat different from Stephenson's
(cp.~\cite{SteBook} p.58). The carrier is essential in the next definition.

\begin{defn}\label{def.Pack.Fill}
Let $G$ be a bounded, simply connected domain.
We say that a (univalent) circle packing $\mathcal{P}$ is \emph{contained in} $G$
(or \emph{lies in} $G$) if the interior of $D^*$ is a subset of $G$.
A packing $\mathcal{P}$ contained in $G$ is said to \emph{fill} $G$ if every
boundary disk of $\mathcal{P}$ touches $\partial G$.
\end{defn}

If $G$ is a Jordan domain, $\mathcal{P}$ is contained in $G$ if and only if
any disk of $\mathcal{P}$ is a subset of $G$. For general domains the latter
condition alone would be too week, since then it could happen that ``spikes''
of $\partial G$ (think of $G$ as a slitted disk) penetrate into the packing,
sneaking through between two boundary disks at their contact point. This is prevented
by our definition; in particular it guarantees that $\partial G\cap I=\emptyset$
for every interstice $I$ of $\mathcal{P}$.

What happens when $\partial G$ meets a contact point of two boundary disks
is explored in the following lemma (an explanation of associated prime ends
is given on page~\pageref{page.asspe}).

\begin{lem}[]\label{lem.PrimeContact}
Let $G$ be a bounded, simply connected domain, and let $\mathcal{P}$ be a circle
packing contained in $G$. Then every contact point $c(u,v)\in\partial G$ is
associated with the same prime end by both $D_u$ and $D_v$.
\end{lem}

\begin{proof}
Let $c=c(u,v)$ be a contact point of $\mathcal{P}$ which lies on the boundary of $G$.
Then there exists a vertex $w\in V$ such that $f(u,v,w)$ is a face in the
complex of $\mathcal{P}$, and we denote by $I=I(u,v,w)$ the corresponding interstice.

For $\varepsilon>0$, let $B_\varepsilon$ be an open disk centered at $c$ with radius
$\varepsilon$ and define
\[
\widetilde{B}_\varepsilon := B_\varepsilon\cap\big(D_u\cup D_v\cup \overline{I}\big).
\]
If $\varepsilon$ is sufficiently small, $\widetilde{B_\varepsilon}\setminus \{c\}$
is a Jordan domain contained in $G$, and we have
$D_u \cap B_\varepsilon \subset \widetilde{B}_\varepsilon$,
$D_v \cap B_\varepsilon \subset \widetilde{B}_\varepsilon$
(see Figure~\ref{fig.BndContact}, left). As a Jordan domain
$\widetilde{B_\varepsilon}\setminus \{c\}$
has a unique prime end $c^*$ corresponding to its boundary point $c$, so the prime ends
of $G$ associated with $c$ by the disks $D_u$ and $D_v$, respectively, must
coincide.
\end{proof}

\begin{figure}[H]
\begin{center}
\begin{overpic}{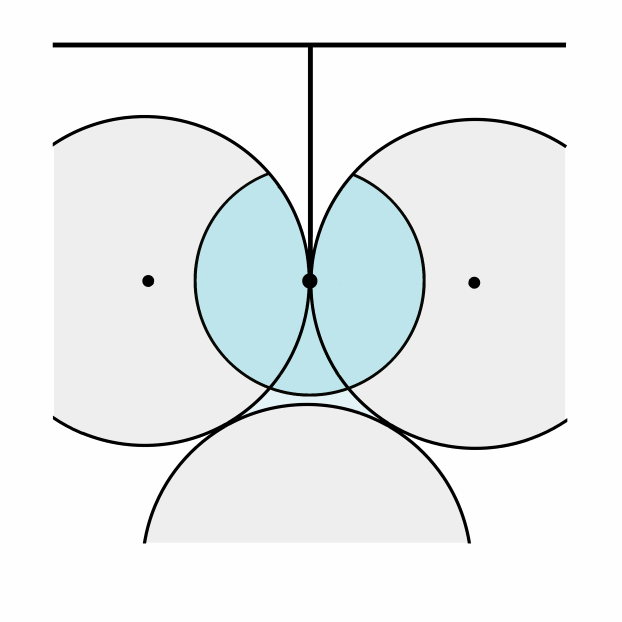}
\fontsize{12pt}{14pt}\selectfont
\put(46,86){\makebox(0,0)[rc]{$\partial G$}}
\put(54,59){\makebox(0,0)[cc]{$c$}}
\put(80,65){\makebox(0,0)[cc]{$D_v$}}
\put(15,65){\makebox(0,0)[cc]{$D_u$}}
\put(50,20){\makebox(0,0)[cc]{$D_w$}}
\put(40,50){\makebox(0,0)[cc]{$\widetilde{B}_\varepsilon$}}
\end{overpic}
\hfill
\begin{overpic}{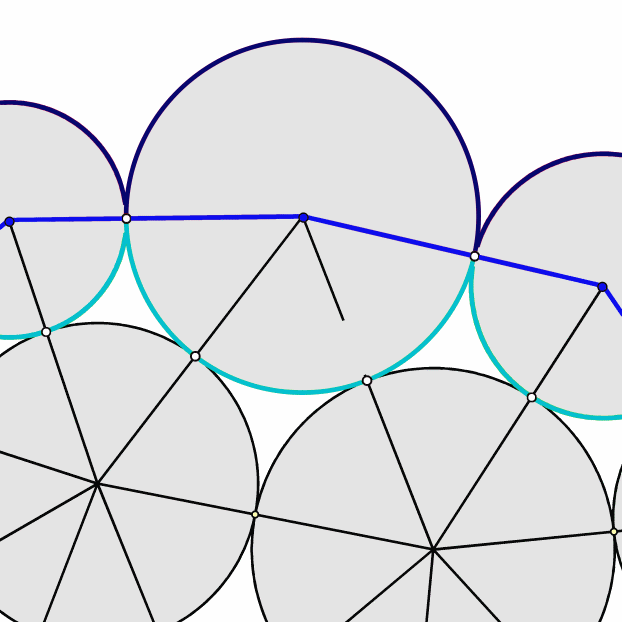}
\fontsize{12pt}{14pt}\selectfont
\put(93,63){\makebox(0,0)[cc]{$D_{k-1}$}}
\put(77,53){\makebox(0,0)[lc]{$c_{k-1}^+$}}
\put(73,55){\makebox(0,0)[rc]{$c_{k}^-$}}
\put(22,60){\makebox(0,0)[lc]{$c_{k}^+$}}
\put(50,72){\makebox(0,0)[cc]{$D_{k}$}}
\put(12,60){\makebox(0,0)[cc]{$c_{k+1}^-$}}
\put(9,72){\makebox(0,0)[cc]{$D_{k+1}$}}
\put(48,86){\makebox(0,0)[cc]{$\delta(c_k^-,c_k^+)$}}
\put(48,44){\makebox(0,0)[cc]{$\delta(c_k^+,c_k^-)$}}
\put(35,69){\makebox(0,0)[cc]{$\eta_{k}$}}
\end{overpic}
\hfill
\begin{overpic}{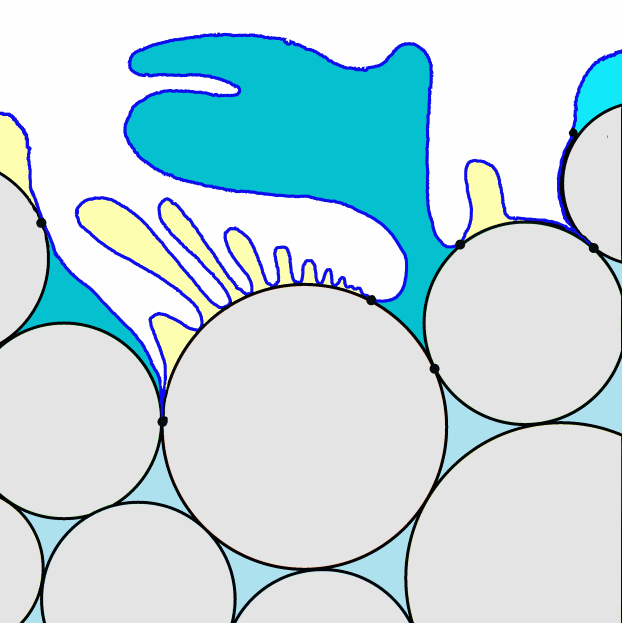}
\fontsize{12pt}{14pt}\selectfont
\put(48,25){\makebox(0,0)[cc]{$D_{k+1}$}}
\put(88,47){\makebox(0,0)[cc]{$D_k$}}
\put(50,80){\makebox(0,0)[cc]{$I_k$}}
\put(76,42){\makebox(0,0)[cc]{$c_k^+$}}
\put(78,58){\makebox(0,0)[cc]{$g_k^+$}}
\put(55,49){\makebox(0,0)[cc]{$g_{k+1}^-$}}
\end{overpic}
\caption{Definitions of $\widetilde{B}_\varepsilon$, boundary arcs and boundary interstices}
\label{fig.BndContact}
\end{center}
\end{figure}
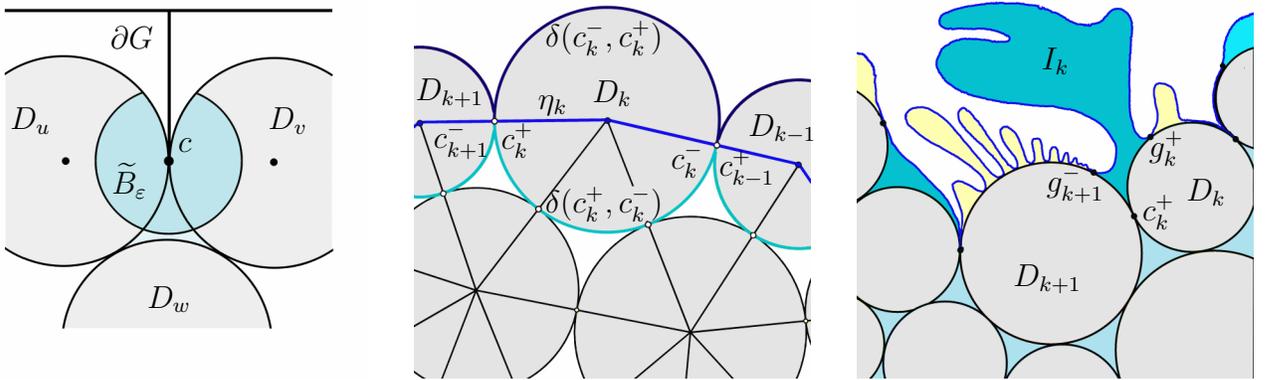
\noindent

A packing which fills the unit disk $\mathbb{D}$ is called \emph{maximal}.
A celebrated result, the Koebe-Andreev-Thurston-Theorem (which can be traced back
to Koebe's paper~\cite{Koebe}), tells us that any complex $K$ has an associated
maximal packing, which is unique up to conformal automorphisms of $\mathbb{D}$.
A far reaching generalization is the Uniformization Theorem of Beardon and
Stephenson (\cite{BeaSte}, see also Chapter~II in~\cite{SteBook}).
\smallskip\par

Recall that the boundary disks of a packing form a chain $D_1,\ldots,D_m$.
Since this is a cyclic structure, we label it modulo $m$, in particular
$D_0:=D_m$ and $D_{m+1}:=D_1$. For $k\in\{1,\ldots,m\}$, we denote by
$\eta_k$ the closed segment which connects the centers of $D_{k}$ and $D_{k+1}$.
These \emph{boundary segments} form a (polygonal) Jordan curve $\eta$.
\label{def.eta}

If $D_{k-1},D_k$ and $D_{k+1}$ are three consecutive boundary disks,
the contact points $c_k^-:=\overline{D}_{k-1}\cap\overline{D}_k$
and $c_k^+:=\overline{D}_{k}\cap\overline{D}_{k+1}$ split $\partial D_k$ into
two arcs. We call $\delta(c_k^-,c_k^+)$ the \emph{exterior boundary arc} and
$\delta(c_k^+,c_k^-)$ the \emph{interior boundary arc} of $D_k$, respectively
(see Figure~\ref{fig.BndContact}, middle).\label{def.arcs2}

\noindent

\begin{lem}[]\label{lem.Pack.BndArc}
Let $D_k$ be a boundary disk of a circle packing $\mathcal{P}$. Then the exterior
boundary arc of $D_k$ contains no contact points of disks in $\mathcal{P}$.
\end{lem}

\begin{proof}
The polygonal line $\eta$ which connects consecutive centers of the boundary disks
is a Jordan curve which separates the exterior boundary arcs from the interior boundary
arcs. The interior of $\eta$ contains the closures $\overline{D}_v$ of all interior
disks.
Any contact point $c$ of $\mathcal{P}$ is either a contact point of two boundary circles,
or it lies on the boundary of an interior disk. In both cases $c$ does not belong to
any exterior boundary arc.
\end{proof}

To provide some more notation, let $\mathcal{P}$ be a circle packing which fills
a bounded, simply connected domain $G$. By definition, every boundary disk $D_k$
touches $\partial G$ in a non-void (possibly uncountable) set $G_k$ of points,
and $G_k$ must be contained in the closure
$\delta[c_k^-,c_k^+]$ of the exterior boundary arc $\delta(c_k^-,c_k^+)$
of $D_k$. Let $\delta_k:=\delta[g_k^-,g_k^+]$ be the smallest subarc (we admit
that this `arc' degenerates to a point) of $\delta[c_k^-,c_k^+]$ which contains
$G_k$. Since $G_k$ is a closed set, we have $g_k^-,g_k^+\in G_k$.\label{def.delta}
\smallskip\par

In order to define the \emph{boundary interstice} \label{page.boundary.interstice}
$I_k$ between two consecutive boundary disks $D_k$ and $D_{k+1}$
(see Figure~\ref{fig.BndContact}, right) we distinguish two cases.
If $g_k^+=c_k^+$, we set $I_k:=\emptyset$. Otherwise we let $\delta$ be the union
of the arcs $\delta(g_k^+,c_k^+]$ (a subarc of $\partial D_k$) and
$\delta[c_k^+,g_{k+1}^-)$ (a subarc of $\partial D_{k+1}$). The open Jordan arc
$\delta$ is contained in $G$ with different endpoints on $\partial G$, hence it
is a crosscut. The set $G\setminus\delta$ consists of two simply connected components
$G_1$ and $G_2$. One of these components contains all disks of $\mathcal{P}$,
the other one is (by definition) the boundary interstice $I_k$.

\begin{lem}[]\label{lem.Interstice}
$I_k\cap\mathcal{D}=\emptyset$ for all $k=1,...,m$.
\end{lem}

\begin{proof}
Let $k\in\{1,...,m\}$ be fixed. If $I_k=\emptyset$ the assertion is trivially fulfilled.
Let $I_k\neq\emptyset$ and let $\delta$ be the crosscut defined above, so that
$G\setminus\delta$ consists of exactly two simply connected domains $G_1=I_k$ and
$G_2$.

Clearly every disk of $\mathcal{P}$ is contained either in $G_1$ or $G_2$. We
assume that there is a disk $D_u$ in $G_1$ (remember $D_k\subset G_2$). Because
$K$ is connected there is a chain $C$ of vertices $\{u,...,v\}$,
where $v$ is the vertex associated with $D_k$.
Because $D_u\subset G_1$ and $D_k\subset G_2$ there have to be two consecutive
vertices $w_1,w_2$ in $C$, so that $D_{w_1}$ is contained in $G_1$ and $D_{w_2}$ in
$G_2$. The contact point $c(w_1,w_2)$ must lie on $\partial G_1\setminus\delta$,
because there are no contact points of $\mathcal{P}$ on $\delta$ according to
Lemma~\ref{lem.Pack.BndArc}.

Let $w_3$ be a vertex, so that $f(w_1,w_2,w_3)$ is a face of $K$. The interstice
$I:=I(w_1,w_2,w_3)$ is contained either in $G_1$ or $G_2$, because it is disjoint
to $\partial G$. Moreover both arcs
$\partial D_{w_1}\cap\partial I$ and $\partial D_{w_2}\cap\partial I$ (up to their
endpoints) lie in the same domain as $I$, without being contained in the
boundary of $G$. This implies, that both disks $D_{w_1}$ and $D_{w_2}$ are contained
either in $G_1$ or $G_2$, a contradiction. Hence, $I_k\cap\mathcal{D}=\emptyset$ for
all $k=1,...,m$.
\end{proof}

Last but not least we state a result about glueing simply connected domains along
a common boundary arc. The proof is left as an exercise (see~\cite{PomBook}).

\begin{lem}[]\label{lem.Geo.Glue}
Let $G_1$ and $G_2$ be simply connected domains with locally connected
boundaries. If $G_1$ and $G_2$ touch each other along a Jordan arc $J$ with
endpoints $a,b$, i.e., $G_1 \cap G_2=\emptyset$ and
$\overline{G}_1\cap\overline{G}_2=J$, then $\big(G_1\cup J \cup G_2\big)
\setminus \{a,b\}$ is a simply connected domain and its boundary is locally
connected.
\end{lem}

\section{Crosscuts} \label{sec.Cuts}

Before we introduce
crosscuts of a (univalent) circle packing which fills a domain $G$, we define
crosscuts of its complex.

\begin{defn}\label{def.Cuts.Cross}
A (combinatoric) \emph{crosscut of a complex} $K$ is a sequence
$L=(e_0,e_1,\ldots,e_l)$ of edges in $K$ with the following properties (i)--(iv):
\begin{itemize}
\itemsep0mm
\item[{\rm(i)}]
The edges are pairwise different, if $0\le j < k\le l$ then $e_j\not=e_k$.
\item[{\rm(ii)}]
For $1\le j\le l$ the edges $e_{j-1}$ and $e_j$ are adjacent to a common face of $K$.
\item[{\rm(iii)}]
Three consecutive edges are not adjacent to the same face of $K$.
\item[{\rm(iv)}]
The edges $e_0$ and $e_l$ are boundary edges.
\end{itemize}
\end{defn}

\begin{figure}[H]
\begin{center}
\begin{overpic}{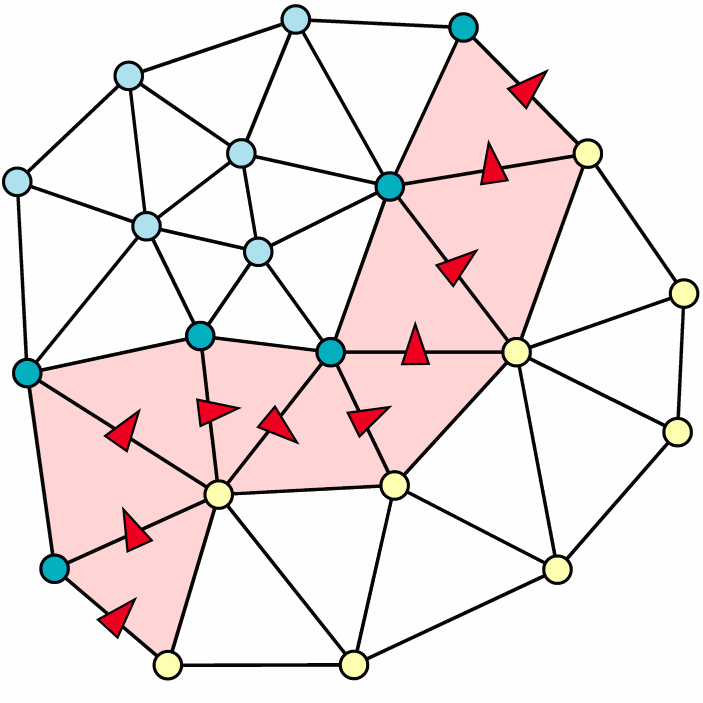}
\fontsize{12pt}{14pt}\selectfont
\put(10,5){\makebox(0,0)[cc]{$L$}}
\put(90,15){\makebox(0,0)[cc]{$K_L^-$}}
\put(5,90){\makebox(0,0)[cc]{$K_L^+$}}
\end{overpic}
\hspace{0.1\textwidth}
\includegraphics{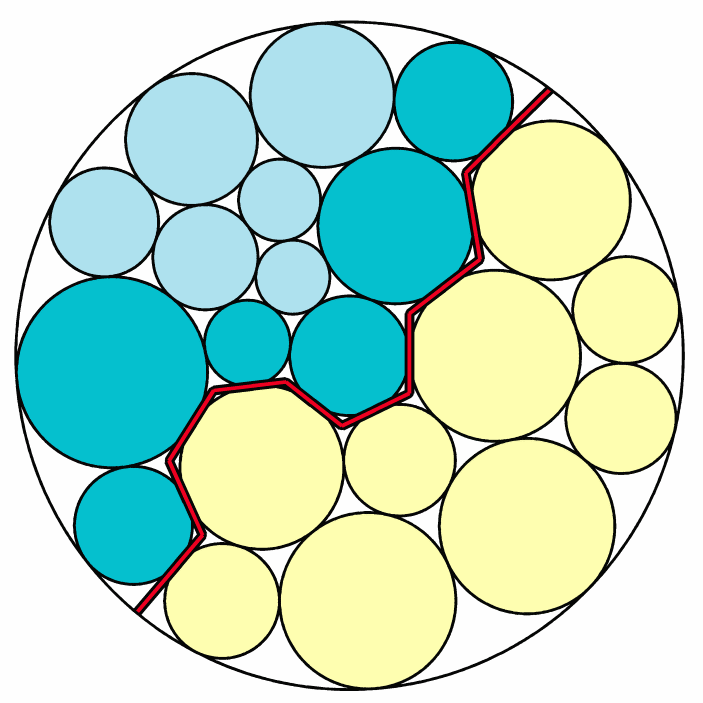}
\caption{A crosscut $L$ of $K$, the vertex sets $V_L^-$, $V_L^+$, $U_L^+$,
and a corresponding packing}
\label{fig.CrossCut}
\end{center}
\end{figure}
\noindent

It is easy to see that only the first and the last edge of a crosscut can be
boundary edges of $K$. Because $e_0\neq e_l$ we have $l\geq 1$.
When one edge of a face $f$ belongs to $L$, then $L$
must contain exactly two edges of $f$, and these are subsequent members
of $L$. So a crosscut can also be represented by a sequence $(f_1,\ldots,f_l)$
of faces,
where $e_{j-1}$ and $e_j$ are adjacent to $f_j$. Since the three edges of
a face are not allowed to be consecutive members of $L$, all faces $f_j$ must
be pairwise different.

After removing the edges of a crosscut $L$ from $K$, the remaining graph consists
of two edge-connected components $K_L^-$ and $K_L^+$.
We assume that $K_L^-$ `lies to the right' and $K_L^+$ `lies to the left',
respectively, when we move along the edges $e_0,e_1,\ldots,e_l$ in this order.
The vertex sets of $K_L^-$ and $K_L^+$ are denoted by $V_L^-$ and $V_L^+$,
respectively, and we call them the \emph{lower} and the \emph{upper vertices}
of $K$ with respect to $L$.\label{def.lower.upper}
The set $U_L^+$ is constituted by all vertices $v$ in $V_L^+$ which are
adjacent to an edge in $L$. These vertices and the corresponding disks
are said to be the \emph{upper neighbors} of $L$.
A corresponding definition is made for the set $U_L^-$ of \emph{lower neighbors}
of $L$ (see Figure~\ref{fig.CrossCut}).
\smallskip\par

Given a (combinatoric) crosscut $L$ of a complex $K$ and a circle packing
$\mathcal{P}$ for $K$ which fills a domain $G$, we define several related
(geometric) crosscuts $J$ of $\mathcal{P}$ in $G$.
To begin with, we associate with every edge $e_j=e(u,v)$ in $L$ the
contact point $x_j:= \overline{D}_u \cap \overline{D}_v$ of the disks
$D_u,D_v\in \mathcal{P}$. The common tangent to $D_u$ and $D_v$ at $x_j$ is
denoted $\tau_j$.
The set $X:=\{x_0,\ldots,x_l\}$ of all contact points associated with edges
of $L$ has a natural
ordering, induced by the ordering of edges in the crosscut. Since
the indexing of the elements fits with this ordering, we write $x_j<x_k$ if $j<k$.
\label{def.X}
\smallskip\par

The \emph{polygonal crosscut} $J_L^0$ is build from the common tangents $\tau_i$
of circles at their contact points $x_i$ as follows.\label{def.PolyCrosscut}
Let $i\in\{1,\ldots,l\}$
and assume that $x_{i-1}$ and $x_i$ are consecutive contact points of the
pairs $D_u,D_v$ and $D_v,D_w$, respectively. Then the three circles
$\partial D_u, \partial D_v, \partial D_w$ bound an interstice $I:=I(u,v,w)$.
The tangents $\tau_{i-1}$ and $\tau_{i}$ intersect each other at a point $s_i$ in
$I$, and the union of the closed segments $[s_{i},s_{i+1}]$ for $i=1,\ldots,l-1$
is a Jordan arc in $G$ (see Figure~\ref{fig.Pack.CrossPoly}).

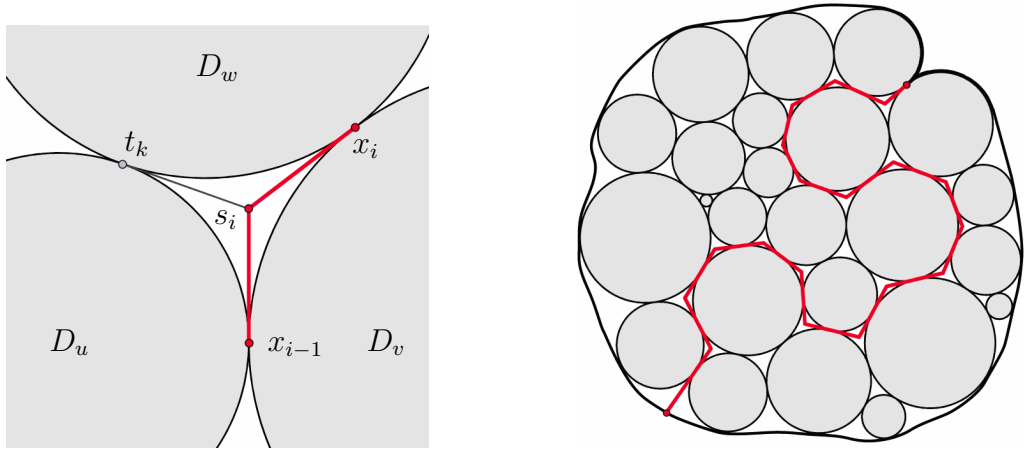
\begin{figure}[H]
\begin{center}
\begin{overpic}{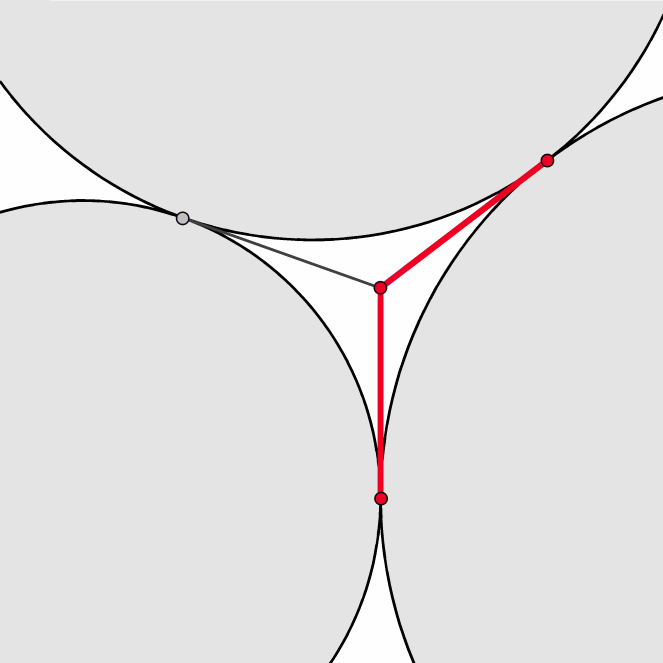}
\fontsize{12pt}{14pt}\selectfont
\put(15,25){\makebox(0,0)[cc]{$D_u$}}
\put(90,25){\makebox(0,0)[cc]{$D_v$}}
\put(50,90){\makebox(0,0)[cc]{$D_w$}}
\put(31,72){\makebox(0,0)[cc]{$t_k$}}
\put(62,26){\makebox(0,0)[lt]{$x_{i-1}$}}
\put(82,73){\makebox(0,0)[lt]{$x_{i}$}}
\put(52,54){\makebox(0,0)[cc]{$s_i$}}
\end{overpic}
\hspace{0.1\textwidth}
\includegraphics{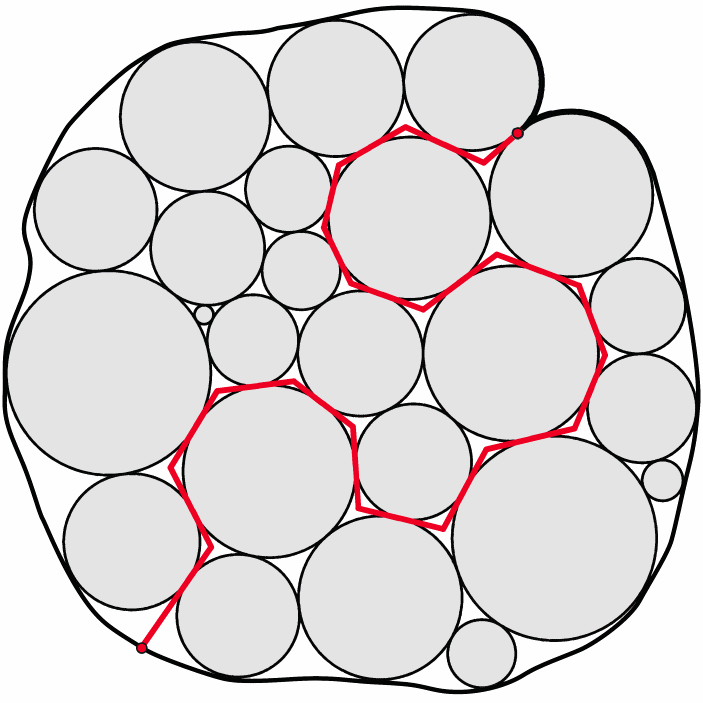}
\caption{Local construction and global view of a polygonal crosscut}
\label{fig.Pack.CrossPoly}
\end{center}
\end{figure}
\noindent

In order to complete this arc to a crosscut in $G$ we look at the boundary disks
$D_k$ and $D_{k+1}$ which touch each other at $x_0$. If $x_0$ is not a boundary point
of $G$ we define $s_0$ as the endpoint of the largest segment $(x_0,s_0)$ on the
tangent $\tau_0$ which is contained in $I_k$.
Since there is no disk of $\mathcal{P}$ intersecting $I_k$ (Lemma~\ref{lem.Interstice})
we see that $[x_0,s_0)\subset G$ is disjoint to $\mathcal{P}$ and $s_0\in\partial G$.
If $x_0$ is a boundary point of $G$ we set $s_0:=x_0$.

A similar construction is made for the point $s_{l+1}$ as (``the first'') intersection
point of the tangent $\tau_l$ with $\partial G$. Here $x_0\neq x_{l}$ ensures that
$[s_0,s_1]$ and $[s_l,s_{l+1}]$ live in two different boundary interstices.
Although this does not exclude $s_0=s_{l+1}$, it guarantees that $s_0$ and $s_{l+1}$
are endpoints of the segments $[s_1,s_0]$ and $[s_l,s_{l+1}]$, belonging to
\emph{different prime ends} $s_0^*$ and $s_{l+1}^*$, respectively.

Finally, the union of the closed segments $[s_{k},s_{k+1}]$ for $k=0,\ldots,l$
forms the desired polygonal crosscut $J_L^0:=\bigcup_{k=0}^l [s_{k},s_{k+1}]$ in $G$.
It can easily be verified that $J_L^0$ is a (topologically closed) Jordan arc which
meets $\overline{D}$ at the contact points $x_k$ -- more precisely we have
$X\subset J_L^0\cap\overline{D}\subset X\cup\{s_0,s_{l+1}\}$.
The open set $G\setminus J_L^0$ has two simply connected components $G^+_0$ and $G^-_0$,
containing the disks associated with $V^+_L$ and $V^-_L$, respectively.
\smallskip\par

It is clear that, for a fixed combinatorial crosscut $L$ of $K$, the statement of
Theorem~\ref{thm.CircRigid} depends on the choice of the geometric crosscut $J$:
the assertion becomes the stronger, the larger the domain $G_J^-$ is.
Unfortunately, there exists (in general) no crosscut $J$ which maximizes $G_J^-$,
since the boundary of the largest domain $G_J^-$ need not be a Jordan curve.
We therefore extend the concept of crosscuts somewhat, defining the
\emph{maximal crosscut} $J_L^+$ in $\mathcal{P}$ as follows.\label{def.MaxCrosscut}

\begin{figure}[H]
\begin{center}
\begin{overpic}{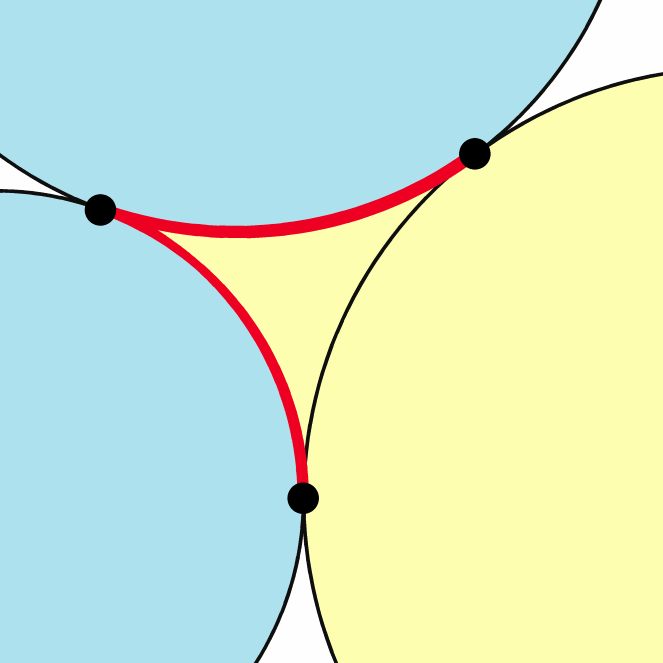}
\fontsize{12pt}{14pt}\selectfont
\put(15,25){\makebox(0,0)[cc]{$D_u$}}
\put(80,45){\makebox(0,0)[cc]{$D_v\subset G^-_L$}}
\put(40,90){\makebox(0,0)[cc]{$D_w$}}
\put(18,74){\makebox(0,0)[cc]{$t_i$}}
\put(50,26){\makebox(0,0)[lt]{$x_{k-1}$}}
\put(75,75){\makebox(0,0)[lt]{$x_{k}$}}
\put(42,58){\makebox(0,0)[cc]{$\alpha_k$}}
\end{overpic}
\hspace{0.1\textwidth}
\includegraphics{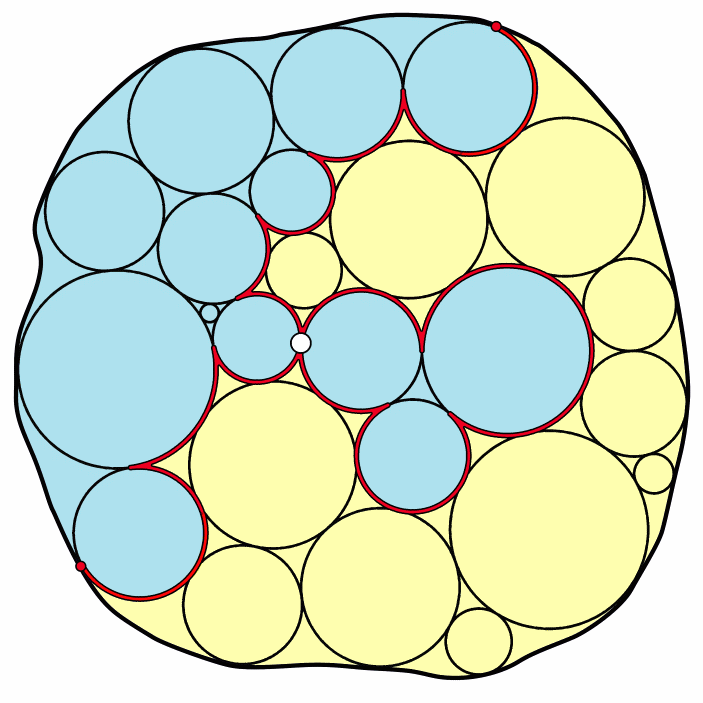}
\caption{Construction of a maximal crosscut (which is not a Jordan arc)}
\label{fig.MaxCut}
\end{center}
\end{figure}
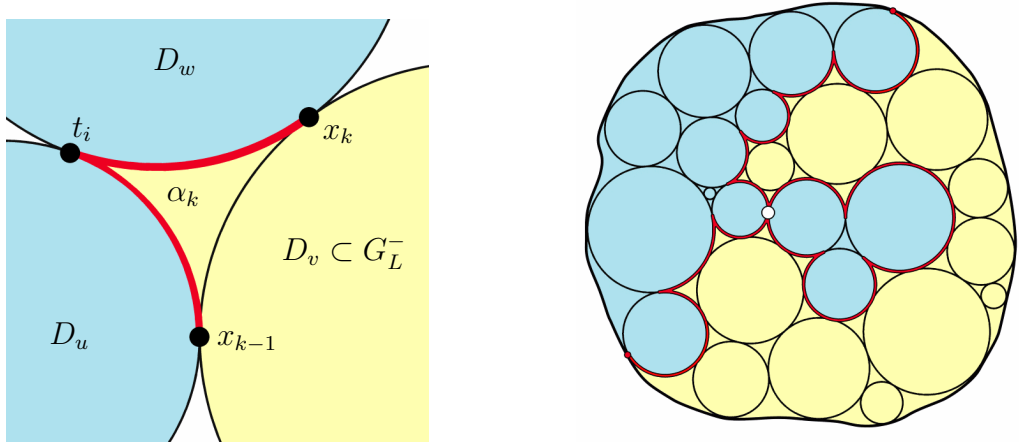
\noindent

Recall that $U_L^+$ is the vertex set of upper neighbors of $L$.
If $x_k$ and $x_{k+1}$ are contact points of the disks $D_u,D_v$ and
$D_v,D_w$, respectively, then either $v\in U_L^+$ or $u,w\in U_L^+$.
The interstice $I(u,v,w)$ is bounded by three (topologically closed) circular arcs
$\alpha_u$, $\alpha_v$ and $\alpha_w$, respectively. If $v\in U_L^+$ we connect $x_{k-1}$
with $x_{k}$ by the arc $a_k:=\alpha_v$, in the second case we connect these
points by the concatenation $a_k:=\alpha_u\cup\alpha_w$ (see Figure~\ref{fig.MaxCut}).
In addition we connect $x_0$ and $x_l$ with $\partial G$ by arcs
$a_0:=\delta(g_j^+,x_0)$ and $a_{l+1}:=\delta(x_l,g_k^-)$ of those circles
$\partial D_j$ and $\partial D_k$ which are upper neighbors of $L$ and contain $x_0$
and $x_l$, respectively. The union $J_L^+:=\bigcup_{k=0}^{l+1} a_k$ of these arcs is
a curve which we call the \emph{maximal crosscut} in $\mathcal{P}$ with respect to $L$.

\smallskip\par

The maximal crosscut $J_L^+$ is composed from a finite number of circular
(topologically closed) arcs $\omega_i$ which are linked at the \emph{turning points}
$t_i$ of $J_L^+$, and every contact point $x_k$ lies exactly on one arc $\omega_i$
(see Figure~\ref{fig.MaxCut}).
If $J_L^+$ is not a Jordan arc, $G \setminus J_L^+$ may consist of several connected
components (see Figure~\ref{fig.MaxCut}, right), one of them
containing all disks associated with vertices $v$ in $V_L^-$.
We call this component $G_L^-$ the \emph{maximal lower domain} for $L$ with respect
to $\mathcal{P}$, and we set $G_L^+:=G \setminus \overline{G_L^-}$.
For the sake of brevity we define $\omega:=J_L^+$ and $\Omega:=G_L^-$.\label{def.omega}

Since the curve $\omega$ can have multiple points (see Figure~\ref{fig.MaxCut}, right)
there is no natural ordering of the \emph{points} on $\omega$. However,
considering $\omega$ as part of the boundary of $\Omega$, we can introduce an
ordering of the \emph{terminal points} $q\in\omega$ of open Jordan arcs
$\gamma(p,q)$ in $\Omega$.
In order to describe this procedure we need the following result.

\begin{lem}[]\label{lem.PropGLM}
For any combinatorial crosscut $L$ the maximal lower domain $\Omega=G_L^-$ is simply
connected and has a locally connected boundary.
\end{lem}

\begin{proof}
Let $G_0^-$ be the lower domain with respect to the polygonal crosscut $J_0$ in
$\mathcal{P}$.
Then $G \setminus J_L^0$ consists of two simply connected domains $G_0^-$ and
$G_0^+$, respectively. The maximal lower domain $G_L^-$ is constructed by glueing
a finite number of simply connected domains along straight line segments to $G_0^-$. Hence
the assertion follows from Lemma~\ref{lem.Geo.Glue}.
\end{proof}

The assertion of Lemma~\ref{lem.PropGLM} guarantees that any (fixed) conformal mapping
$g: \mathbb{D}\rightarrow\Omega$ has a continuous extension to $\overline{\mathbb{D}}$,
which we again denote by $g$ (see~\cite{PomBook} Theorem~2.1).
With respect to this mapping, we let $\sigma_i\subset\mathbb{T}$ denote the preimage of
the circular arcs $\omega_i$ with $i=1,\ldots,n$. Then $\sigma:=\bigcup_{i=1}^n\sigma_i$
is the preimage of the maximal crosscut $\omega$.

By the Prime End Theorem, the mapping $g$ induces a bijection $g^*$ between $\mathbb{T}$
the set of prime ends of $\Omega$. We denote by $\omega^*:=g^*(\sigma)$ the
set of prime ends associated with $\Omega$, and, for $i=1,\ldots,n$, we let
$\omega_i^*:=g^*(\sigma_i)$ be the subsets of $\omega^*$
corresponding to the arcs $\sigma_i$.\label{def.PrimeEnds}

Note that the preimages $\sigma_i$ of the circular arcs $\omega_i$
are topologically closed subarcs of $\mathbb{T}$, and that the preimage
$\mathbb{T}\setminus\sigma$ of $\partial\Omega\setminus\omega$ is not empty.
Therefore $\sigma_i$ and $\sigma_j$, and thus $\omega_i^*$ and $\omega_j^*$, are
disjoint if $|i-j|>1$, while their intersection contains exactly one element
if $|i-j|=1$.

Further we see that the arcs $\sigma_1, \sigma_2, \ldots, \sigma_n$ (in this
order) are arranged in clockwise direction on $\mathbb{T}$.
It is therefore just natural to order the \emph{points} on the arc $\sigma$
(and hence on each subarc $\sigma_i$) also in \emph{clockwise} direction.
The mapping $g^*$ transplants this ordering from $\sigma$ to the set $\omega^*$
of prime ends.
If $\gamma_1^*=g^*(s_1)$ and $\gamma_2^*=g^*(s_2)$ are two prime ends of $\omega^*$,
the notion $\gamma_1^*\le\gamma_2^*$ refers to the ordering $s_1\le s_2$ of the
associated points on $\sigma$.

\medskip\par

\noindent
Remark. Every $\omega_i$ without its endpoints is an
open Jordan arc, so the interior points of $\omega_i$ and $\sigma_i$ corresponds
one-to-one. Let $\gamma$ in $\Omega$ be an open Jordan arc with terminal point
$q$ on $\omega$, then the associated unique prime end $\gamma^*$ in $\omega^*$
must lie in $\omega^*_i$, if $q$ is an interior point of $\omega_i$. Only if $q$
is an endpoint of $\omega_i$ there is a chance that the prime end $\gamma^*$ is
not contained in $\omega^*_i$, because now $\gamma^*$ depends on how $\gamma$
approaches $q$.

\section{Loners} \label{sec.Loner}

So far we have studied properties of a single circle packing $\mathcal{P}$ which
fills $G$. In the next step we consider pairs $(\mathcal{P},\mathcal{P}')$ of
packings which are subject to the assumptions of Theorem~\ref{thm.CircRigid}.

\begin{defn}\label{def.Chap}
A pair $(\mathcal{P},\mathcal{P}')$ of univalent circle packings for the
complex $K$ is said to be \emph{admissible} (for the crosscut $L$ of $K$ in $G$
with alpha-vertex $v_\alpha$) if it satisfies the following conditions:
\begin{itemize}
\itemsep0mm
\item[{\rm (i)}]
The packing $\mathcal{P}$ fills the bounded, simply connected domain $G$, and
the packing $\mathcal{P}'$ is contained in $G$ (see Definition~\ref{def.Pack.Fill}).
\item[{\rm (ii)}]
For all vertices $v\in U_L^-$ (the lower neighbors of $L$) the disks $D'_v$ are
contained in $G_L^-$ (the maximal lower domain of $G$ for $L$ with respect to
$\mathcal{P}$).
\item[{\rm (iii)}]
The centers of the alpha-disks of $\mathcal{P}$ and $\mathcal{P}'$ coincide
and lie in $G_L^+:=G \setminus G_L^-$.
\end{itemize}
\end{defn}

Though it would be more precise to speak of an admissible sixtuple
$(K,L,G,\mathcal{P},\mathcal{P}',v_\alpha)$, we shall use the term ``admissible''
generously, for instance saying that ``$L$ is an admissible crosscut
for $(\mathcal{P},\mathcal{P}')$''.
\smallskip\par

\noindent
Recall that $U_L^+$ denotes the vertex set of those disks in $\mathcal{P}$
which lie in $G_L^+$ and touch the crosscut (``upper neighbors of $L$'').
In the next step we are going to explore the interplay of the disks $D_v$ and
$D_w'$ for $v,w\in U_L^+$.

\begin{defn}\label{def.Lon.Lon}
Let $(\mathcal{P},\mathcal{P}')$ be an admissible pair of circle packings
for the complex $K$ with crosscut $L$. A vertex $v$ in $U_L^+$ is called a
\emph{loner}, if $D_v'\cap D_w= \emptyset$ for all $w\in U_L^+$ with $w \not=v$.
\end{defn}

The concept of loners was introduced by Schramm~\cite{Schramm2} in a similar
but somewhat different context. The main characteristic of a loner is the following.

\begin{lem}[]\label{lem.Loner.Property}
Let $v$ in $U_L^+$ be a loner of the admissible pair $(\mathcal{P},\mathcal{P}')$
with complex $K$ and crosscut $L$. Then $D'_v\cap(G^+_L\setminus D_v)=\emptyset$.
\end{lem}

\begin{proof}
Let $u\in U^-_L$ and $w\in U^+_L$ be neighbors of $v$, and let $p$ and $q$ be the
contact points of the disks $D'_v$ with $D'_u$ and $D_v$ with $D_w$, respectively.
Clearly $p\neq q$, otherwise $D'_u$ had to intersect $D_v$ or $D_w$, a
contradiction to condition~(ii) of the admissible pair $(\mathcal{P},\mathcal{P}')$.

Assume that $p$ is a boundary point of $D_v$. Then $\partial D_v$ and $\partial D'_v$
have a common tangent at $p$, otherwise $D'_u$ had to intersect $D_v$, a
contradiction to condition~(ii) of the admissible pair $(\mathcal{P},\mathcal{P}')$.
It follows that either $\overline{D'_v}\setminus\{p\}\subset D_v$ or $D'_v=D_v$ or
$\overline{D_v}\setminus\{p\}\subset D'_v$. The latter implies that $q\in D'_v$, hence
$D'_v\cap D_w\neq\emptyset$, which is impossible since $v$ is a loner. The other
two cases imply the statement we want to prove.

Assume that $p$ is not a boundary point of $D_v$. Suppose that the assertion of
Lemma~\ref{lem.Loner.Property} were false, i.e., there is some point $r$ in
$D'_v$ which is also contained in $G^+_L\setminus D_v$. Because $p$ lies in the
maximal lower domain $G^-_L$, and $r$ lies in the upper domain $G^+_L$, both
subarcs $\delta(p,r)$ and $\delta(r,p)$ of $D'_v$ must intersect the maximal
crosscut $J^+_L$ at points $r_1$ and $r_2$, respectively. Since the vertex
$v$ is a loner, we have $r_1,r_2\in\partial D_v$.
If $r_1=r_2$, the boundary of $D'_v$ is the union of $\delta[p,r_1]$ and
$\delta[r_2,p]$, hence $D'_v\cap G^+_L=\emptyset$, a contradiction to $r\in D'_v$.
If $r_1\neq r_2$, we have $\partial D'_v\cap D_v=\delta(r_2,r_1)$,
hence $r$ must be contained in $D_v$, a contradiction to $r\in G^+_L\setminus D_v$.
\end{proof}

In Section~\ref{sec.Proof} the property of loners described in
Lemma~\ref{lem.Loner.Property} will allow us to move the crosscut $L$
through the packing, reducing in every step the number of circles in $G^+_L$.
The next result is crucial for the applicability of this procedure.

\begin{lem}[Existence of Loners]\label{lem.Dia.Loner}
Every admissible pair $(\mathcal{P},\mathcal{P}')$ of circle packings with crosscut
$L$ has a loner.
\end{lem}

The proof is divided into several steps;
the first part uses the \emph{geometry} of disks,
then we employ some \emph{topology},
and finally everything is reduced to pure \emph{combinatorics}.
We start with some preparations.
\smallskip\par

\noindent
Recall the definition of the contact points $x_k$: If $L=(e_0,\ldots,e_l)$
and $e_k=\langle u,v\rangle$, for some $k\in \{0,\ldots,l\}$,  then
$x_k:= \overline{D_u}\cap\overline{D_v}$. Using the same notation, the
corresponding contact points of disks in $\mathcal{P}'$ are given by
$y_k:= \overline{D'_u} \cap \overline{D'_v}$, where $Y:=\{y_0,\ldots,y_l\}$
is the set of all such contact points.\label{def.Y}

The contact points $x_k$ form an ordered set on the maximal crosscut $\omega:=J_L^+$,
which is the upper boundary of the maximal lower domain $\Omega:=G_L^-$.
Since every $x_k$ lies on exactly one arc $\omega_i$, the set $X$ of contact
points splits into classes $X_i:=\{x_k\in X: x_k\in \omega_i\}$, $i=1,\ldots,n$.
The set $Y$ of the contact points of $\mathcal{P}'$ is divided accordingly,
$Y_i:=\{y_k\in Y: x_k\in \omega_i\}$ (the $x_k$ is no typo here). Like $X$, the
set $Y$ is endowed with a natural ordering, we write $y_j<y_k$ if $j<k$.
\label{def.XYi}
\smallskip\par

Our next aim is to construct a Jordan arc $\alpha$ which is contained
in $\overline{\Omega}$ and carries the contact points $y_k$ in their natural order.

\begin{figure}[H]
\begin{center}
\begin{overpic}{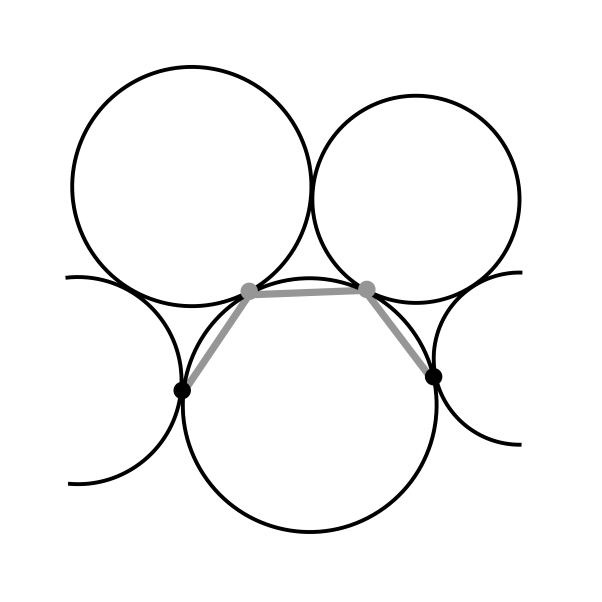}
\fontsize{12pt}{14pt}\selectfont
\put(51,30){\makebox(0,0)[cc]{$D'_v\subset\Omega$}}
\put(37,58){\makebox(0,0)[cc]{$y_{k-1}$}}
\put(65,57){\makebox(0,0)[cc]{$y_k$}}
\end{overpic}
\hfill
\begin{overpic}{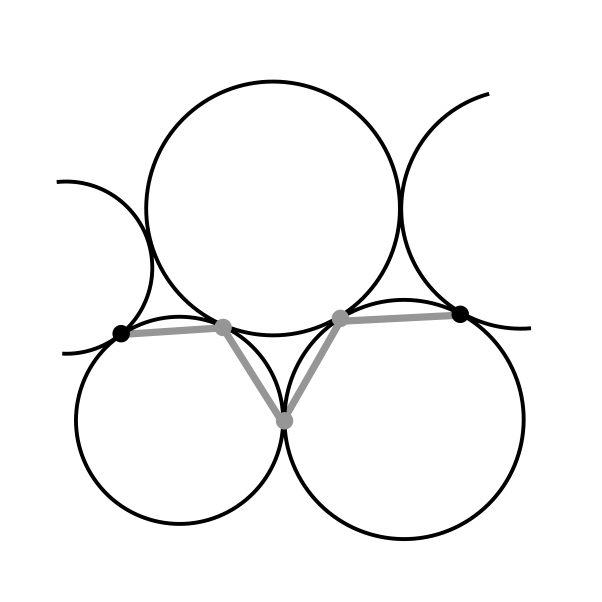}
\fontsize{12pt}{14pt}\selectfont
\put(28,28){\makebox(0,0)[cc]{$D'_w$}}
\put(70,28){\makebox(0,0)[cc]{$D'_u$}}
\put(45,70){\makebox(0,0)[cc]{$D'_v$}}
\put(42,50){\makebox(0,0)[cc]{$y_{k-1}$}}
\put(55,52){\makebox(0,0)[cc]{$y_k$}}
\put(52,27){\makebox(0,0)[cc]{$z_k$}}
\end{overpic}
\hfill
\begin{overpic}{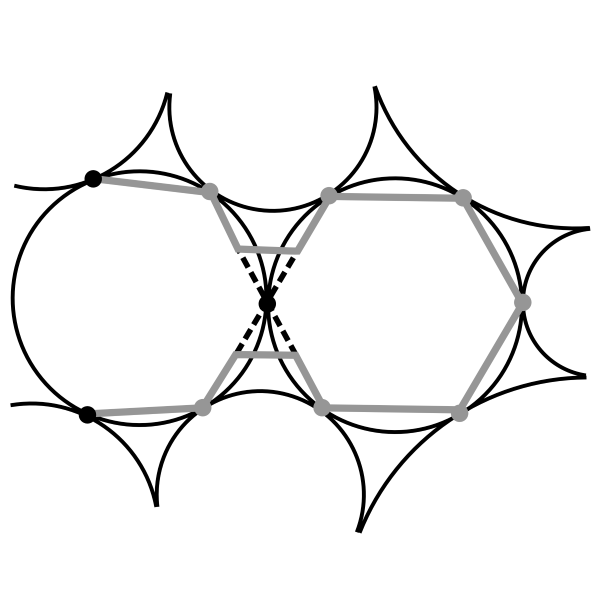}
\fontsize{12pt}{14pt}\selectfont
\put(20,50){\makebox(0,0)[cc]{$D'_u$}}
\put(67,50){\makebox(0,0)[cc]{$D'_w$}}
\put(49,27){\makebox(0,0)[cc]{$y_{k-1}$}}
\put(35,27){\makebox(0,0)[cc]{$y_k$}}
\put(39,74){\makebox(0,0)[cc]{$y_{i-1}$}}
\put(53,72){\makebox(0,0)[cc]{$y_i$}}
\put(50,49){\makebox(0,0)[cc]{$z$}}
\end{overpic}
\caption{Construction of the Jordan arc $\alpha$ in Case~1 (left) and Case~2 (middle,
right)}
\label{fig.Lone.Alpha}
\end{center}
\end{figure}
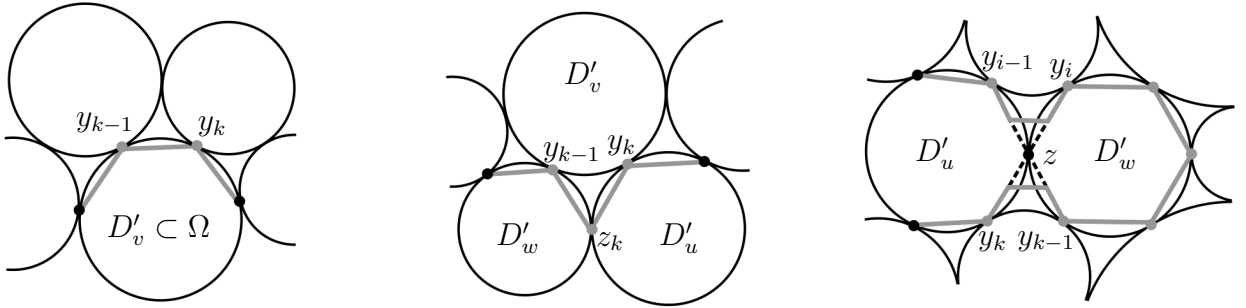
\noindent

\begin{lem}[]\label{lem.AlphaArc}
If $(\mathcal{P},\mathcal{P}')$ is an admissible pair, then there exist oriented
Jordan arcs $\alpha_k$ from $y_{k-1}$ to $y_k$ such $\alpha := \cup_{k=1,\ldots l}
\alpha_k$ is a Jordan arc in $\overline{\Omega}$ and $\alpha\cap\omega\subset Y$.
\end{lem}

\begin{proof}
Let $k\in\{1,\ldots,l\}$. In order to determine the arc $\alpha_k$ of $\alpha$ which
connects $y_{k-1}$ with $y_k$ we remark that both points lie on the boundary of one
and the same disk $D_v'\in \mathcal{P}'$. We distinguish two cases:
\smallskip\par

{\bf Case~1}. If $v\in V_L^-$, then the disk $D_v'$ is contained in $\Omega$, and we choose
the segment $\alpha_k:=[y_{k-1},y_k]$ (see Figure~\ref{fig.Lone.Alpha}, left).
\smallskip\par

{\bf Case~2}. If $v\in V_L^+$, then $e_{k-1},e_k$ and a third edge $\langle u,w\rangle$ of
$K$ form a face of $K$, and the (neighboring) disks $D_u'$ and $D_w'$ are both contained
in ${\Omega}$. So we let $z_k:=\overline{D_u'}\cap\overline{D_w'}$ and connect
$y_{k-1}$ with $y_k$ by $[y_{k-1},z_k]\cup[z_k,y_k]\subset\overline{\Omega}$
(see Figure~\ref{fig.Lone.Alpha}, middle).\label{def.Z}

\smallskip\par

It is clear that all \emph{open} segments $(y_{k-1},y_k)$, $(y_{k-1},z_k)$,
$(z_k,y_k)$ for $k=1,\ldots,l$ are pairwise disjoint, and that $y_k \not=z_j$.
However, it is possible that two endpoints $z_k$ and $z_j$ coincide for
$j\not=k$, in which case the concatenation of the arcs $\alpha_k$ is not a
Jordan arc.

If this happens, the point $z:=z_j=z_k$ is the contact point of two disks
$D_u'$ and $D_w'$ with $u,w\in V_L^-$. A little thought shows that then $z$
can neither lie on the boundary of $G$ nor on $\omega$, and hence it must be an
interior point of $\Omega$. This allows one to resolve the double point of $\alpha$
at $z$ without destroying its other properties (see Figure~\ref{fig.Lone.Alpha},
right.)
\end{proof}

In the next step we transform the existence of loners to a topological
problem. Technically this is much simpler when $\alpha$ and $\omega$ are disjoint.
We consider this `regular case' in Section~\ref{subsec.LoneReg}. The `critical
case', where intersections of $\alpha$ and $\omega$ are admitted, will be treated in
Section~\ref{subsec.LoneCrit}.

\subsection{The Regular Case} \label{subsec.LoneReg}

Here we assume that $\alpha\cap\omega=\emptyset$, which implies that all contact
points $y_k$ ($k=0,\ldots,l)$ lie in the lower domain $\Omega$.

We fix $i\in \{1,\ldots,n\}$ and denote by $y_i^-$ and $y_i^+$ the smallest and
the largest member of $Y_i$ with respect to the natural ordering of $Y$, respectively.
Both points (which may coincide), as well as all elements of $Y_i$, lie on the
same circle $\partial D_v'$, associated with a vertex $v=v(i)\in V$.\label{def.Yi}
\smallskip

Let $\delta_i'$ be the negatively oriented topologically closed subarc of
$\partial D_v'$ from $y_i^-$ to $y_i^+$. We consider the largest subarcs $\nu_i$ and $\pi_i$
of $\delta_i'$ which are contained in $\overline{\Omega}\setminus \omega$
and have initial points $y_i^-$ (for $\eta_i$) and $y_i^+$ (for $\pi_i$), respectively
(see Figure~\ref{fig.Reg.Intersect}).
\label{def.NuPi}

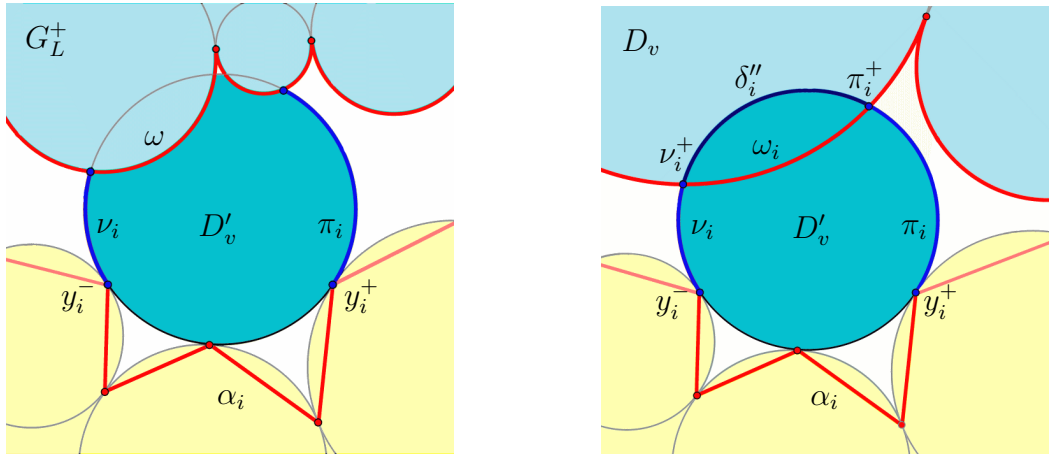
\begin{figure}[H]
\begin{center}
\begin{overpic}{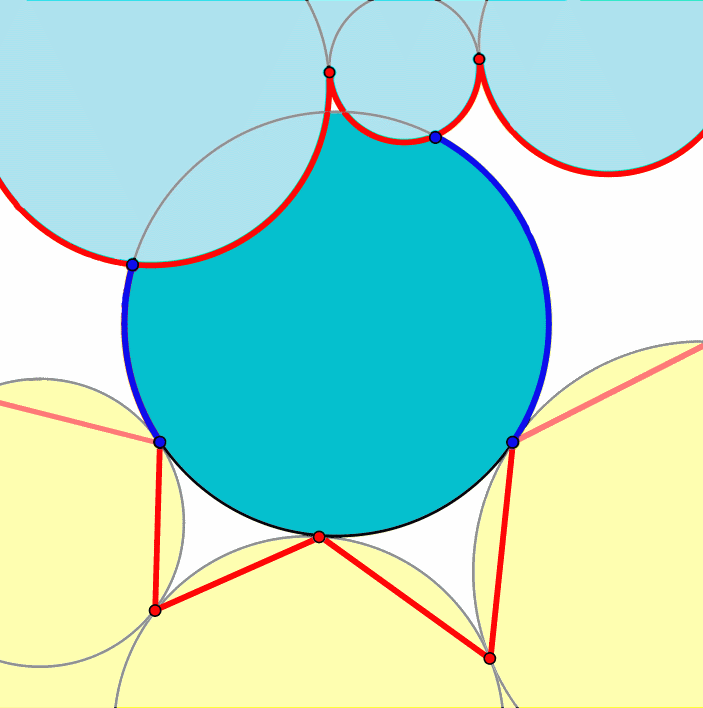}
\fontsize{12pt}{14pt}\selectfont
\put(9,92){\makebox(0,0)[cc]{$G_L^+$}}
\put(33,70,){\makebox(0,0)[cc]{$\omega$}}
\put(20,50){\makebox(0,0)[lc]{$\nu_i$}}
\put(75,50){\makebox(0,0)[rc]{$\pi_i$}}
\put(47,50){\makebox(0,0)[cc]{$D_v'$}}
\put(20,35){\makebox(0,0)[rc]{$y_i^-$}}
\put(75,35){\makebox(0,0)[lc]{$y_i^+$}}
\put(50,14){\makebox(0,0)[ct]{$\alpha_i$}}
\end{overpic}
\hspace{0.1\textwidth}
\begin{overpic}{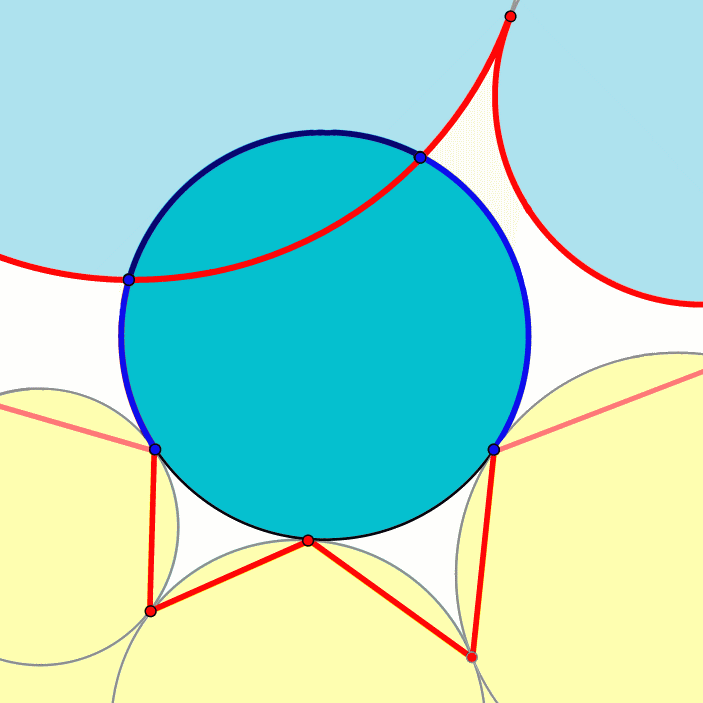}
\fontsize{12pt}{14pt}\selectfont
\put(9,92){\makebox(0,0)[cc]{$D_{v}$}}
\put(33,80){\makebox(0,0)[cb]{$\delta_i''$}}
\put(59,80){\makebox(0,0)[cb]{$\pi_i^+$}}
\put(37,68){\makebox(0,0)[cc]{$\omega_i$}}
\put(21,67){\makebox(0,0)[cr]{$\nu_i^+$}}
\put(50,14){\makebox(0,0)[ct]{$\alpha_i$}}
\put(20,50){\makebox(0,0)[lc]{$\nu_i$}}
\put(73,50){\makebox(0,0)[rc]{$\pi_i$}}
\put(47,50){\makebox(0,0)[cc]{$D_{v}'$}}
\put(20,34){\makebox(0,0)[rc]{$y_i^-$}}
\put(72,34){\makebox(0,0)[lc]{$y_i^+$}}
\put(,){\makebox(0,0)[cc]{$ $}}
\end{overpic}
\caption{The arcs $\nu_i$ and $\pi_i$ and their intersection with the boundary
of $G_L^+$}
\label{fig.Reg.Intersect}
\end{center}
\end{figure}

\begin{lem}[]\label{lem.Lone.NuPiToOmega}
If there exists no loner, then the terminal points $\nu_i^+$ and $\pi_i^+$ of
$\nu_i$ and $\pi_i$, respectively, lie on $\omega$ for $i=1,\ldots,n$.
\end{lem}

\begin{proof}
If one of the arcs $\nu_i$ or $\pi_i$ does not intersect $\omega$, then both coincide
with $\delta_i'$. In this case, the disk $D_{v(i)}'$ is separated from $G_L^+$ by the
union of the arcs $\alpha$ and $\delta_i'$, which implies that $D_{v(i)}'$ cannot
intersect any disk $D_w$ with $w\in U_L^+$, so that $v(i)$ is a loner.
\end{proof}

Since (with the exception of their endpoints) the circular arcs $\nu_i$
($i=2,\ldots, n$) and $\pi_i$ ($i=1,\ldots, n-1$) lie in $\Omega$ and have
terminal points $\nu_i^+$ and $\pi_i^+$ on $\omega$, they define prime
ends $\nu_i^*$ and $\pi_i^*$ in $\omega^*$.
\label{def.Prime.NuPi}
Because the arcs $\nu_1$ and $\pi_n$ need not lie in $\Omega$, a modified
definition is needed for the prime ends $\nu_1^*$ and $\pi_n^*$. To do so
we replace $\nu_1$ and $\pi_n$ by slightly perturbed circular arcs
$\nu_1^\varepsilon$ and $\pi_n^\varepsilon$, respectively, which have the same
endpoints as $\nu_1$ and $\pi_n$, respectively, and lie in $\Omega$
(with the exception of their endpoints).
Then $\nu_1^*$ and $\pi_n^*$ are defined as the prime ends associated
with the terminal points of $\nu_1^\varepsilon$ and $\pi_n^\varepsilon$,
respectively. Clearly such arcs $\nu_1^\varepsilon$ and $\pi_n^\varepsilon$
exist, and for all sufficiently small $\varepsilon$ they define the same
prime ends $\nu_1^*,\pi_n^*\in \omega^*$, respectively.
\smallskip\par

\noindent
Since the set of prime ends $\omega^*$ is endowed with a natural ordering, we can
compare the prime ends $\nu_i^*$ and $\pi_i^*$.

\begin{lem}[]\label{lem.Lone.Interlace}
If $(\mathcal{P},\mathcal{P}')$ has no loner, the prime ends $\nu_i^*$ and $\pi_i^*$
form an interlacing sequence with respect to the prime end ordering of $\omega^*$,
\[
\nu_1^* \le \pi_1^* \le \nu_2^* \le \pi_2^* \le \ldots \le \nu_n^* \le \pi_n^*.
\]
\end{lem}

\begin{proof}
Let $y_-:=y_0$ and $z_-$ be the initial and terminal points of $\nu_1$, while
$y_+:=y_l$ and $z_+$ are the initial and terminal points of $\pi_n$,
respectively. We have $z_-,z_+\in\omega$ due to Lemma~\ref{lem.Lone.NuPiToOmega}.
\label{def.yz-}

Further, let $\omega^*_0$ be the set of all prime ends $\gamma^*$ of $\omega^*$
with $\nu_1^*\leq\gamma^*\leq\pi_n^*$, and denote the set of all corresponding
points on $\omega$ by $\omega_0$.
The set $\omega_0$ is a curve or a single point. Together with the Jordan arcs
$\nu_1$, $\alpha$
and $\pi_n$ it forms the boundary of a simply connected domain $\Omega_0\subset\Omega$
with locally connected boundary. Let $\Omega_0^*$ be the set of all prime ends
associated with points on $\partial\Omega_0$. Because $\Omega_0\setminus\omega_0$
is an open Jordan arc, the points $y_-,y_+$ are associated with uniquely determined
prime ends $y_-^*,y_+^*$ of $\Omega_0$.

Contrary to this, the points $z_-,z_+$ may be associated with several prime ends
of $\Omega_0$. In order to explain which one we choose, let again
$\nu_1^\varepsilon,\pi_n^\varepsilon$ be small perturbations (as explained above)
of $\nu_1,\pi_n$, respectively, so that both arcs are crosscuts in $\Omega_0$.
We define $z_-^*$ and $z_+^*$ as the prime ends in $\omega^*$ associated with
the terminal points $z_-$ and $z_+$ of $\nu_1^\varepsilon,\pi_n^\varepsilon$,
respectively.

We have $n>1$, because otherwise a loner would exist. It follows that
$y_-\neq y_+$, so $y_-^*\neq y_+^*$. From $\alpha\cap\omega=\emptyset$ we
get $z_-,z_+\notin\{y_-,y_+\}$, hence $z_-^*,z_+^*\notin\{y_-^*,y_+^*\}$.

If $z_-^*=z_+^*=:z^*$, we directly get $\omega^*\cap\Omega_0^*=z^*$.
This implies $\nu_1^*=\pi_1^*=\nu_2^*=\ldots=\pi_n^*=z^*$, so the lemma holds.
(We consider this case here, though Lemma~\ref{lem.Lone.Crit} shows, that
it cannot occur.)

If $z_-^*\neq z_+^*$, the prime ends $y_-^*$, $y_+^*$, $z_-^*$ and $z_+^*$ are
pairwise distinct and with respect to the (cyclic) ordering of $\Omega_0$ we have
$y_-^*<y_+^*<z_-^*<z_+^*<y_-^*$. Therefore $\Omega_0$ can be mapped conformally
onto a rectangle $Q$ (with appropriately chosen aspect ratio) such that
$y_-^*,y_+^*,z_-^*$ and $z_+^*$ correspond to the four corners of $Q$
(see~\cite{PomBook}), what is depicted in Figure~\ref{fig.Lone.Interlace}.

Any of the arcs $\nu_i$ ($i=2,\ldots,n)$ and $\pi_i$ ($i=1,\ldots,n-1$) is mapped onto
a crosscut of $Q$ which connects two opposite sides of this rectangle. Since these
Jordan arcs cannot cross each other in the interior of $Q$, the ordering
of their initial points on one side of $Q$ is transplanted to the ordering of their
terminal points on the opposite side of $Q$. Translated back to $\Omega_0$, this
implies that the ordering of the prime ends $\nu_i^*$ and $\pi_i^*$ is the same
as the ordering of the initial points $y_i^-$ and $y_i^+$ of $\nu_i$ and $\pi_i$,
respectively, along the Jordan curve $\alpha$. By construction, the latter points form
an interlacing sequence.
\end{proof}

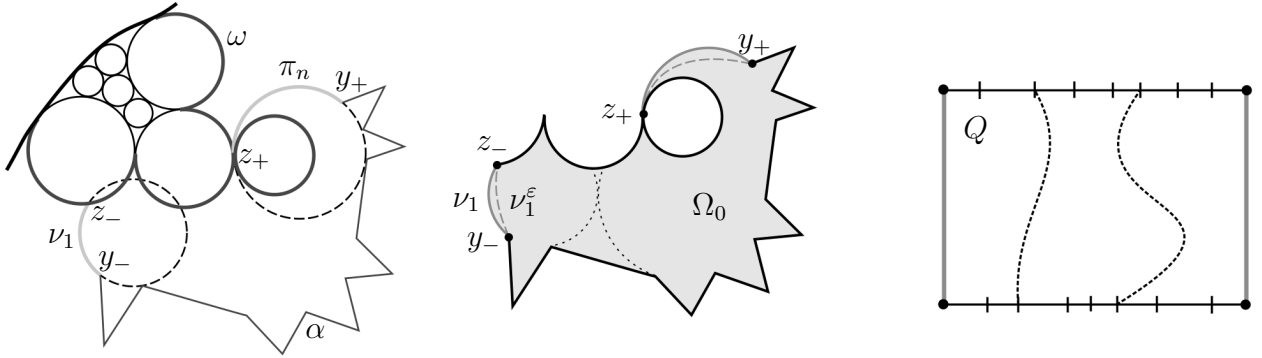
\begin{figure}[H]
\begin{center}
\begin{overpic}{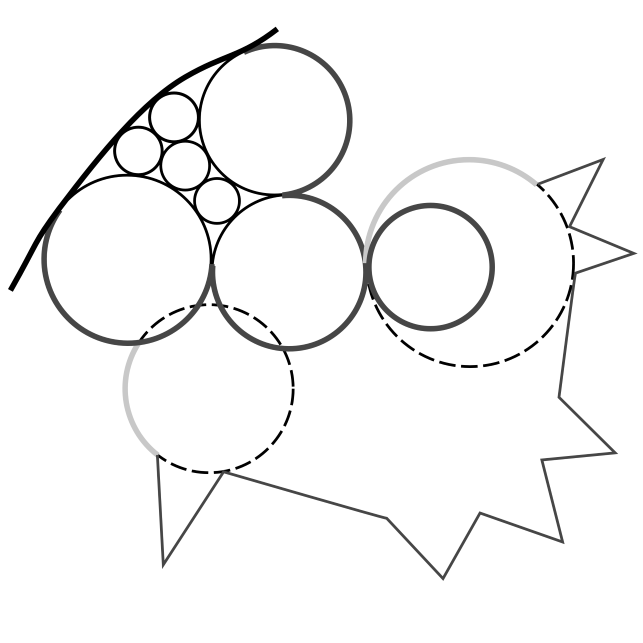}
\fontsize{12pt}{14pt}\selectfont
\put(58,87){\makebox(0,0)[cc]{$\omega$}}
\put(77,15.5){\makebox(0,0)[cc]{$\alpha$}}
\put(18.5,38){\makebox(0,0)[rc]{$\nu_1$}}
\put(72,77){\makebox(0,0)[cb]{$\pi_n$}}
\put(24,30){\makebox(0,0)[lb]{$y_-$}}
\put(22,46){\makebox(0,0)[lt]{$z_-$}}
\put(86,74){\makebox(0,0)[cb]{$y_+$}}
\put(62,55){\makebox(0,0)[cb]{$z_+$}}
\end{overpic}
\hfill
\begin{overpic}{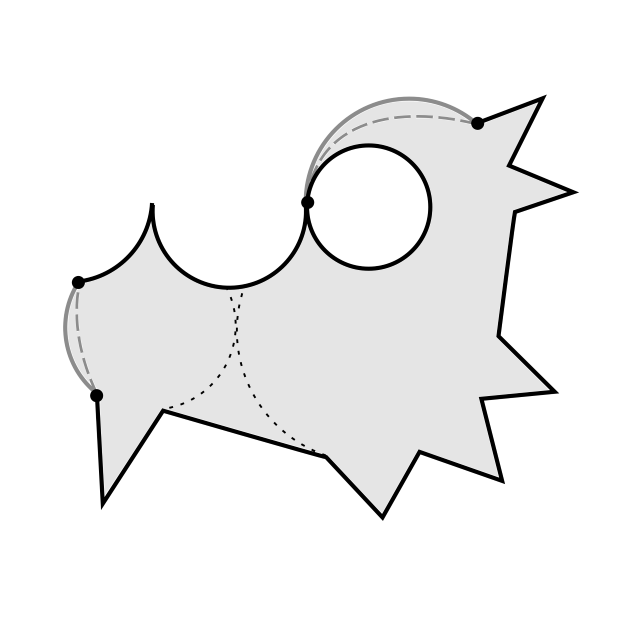}
\fontsize{12pt}{14pt}\selectfont
\put(11,58){\makebox(0,0)[cb]{$z_-$}}
\put(45.5,68.5){\makebox(0,0)[rc]{$z_+$}}
\put(13,37){\makebox(0,0)[ru]{$y_-$}}
\put(75,83){\makebox(0,0)[cb]{$y_+$}}
\put(64,46){\makebox(0,0)[cc]{$\Omega_0$}}
\put(8,47){\makebox(0,0)[rc]{$\nu_1$}}
\put(15,47){\makebox(0,0)[lc]{$\nu_1^\varepsilon$}}
\end{overpic}
\hfill
\begin{overpic}{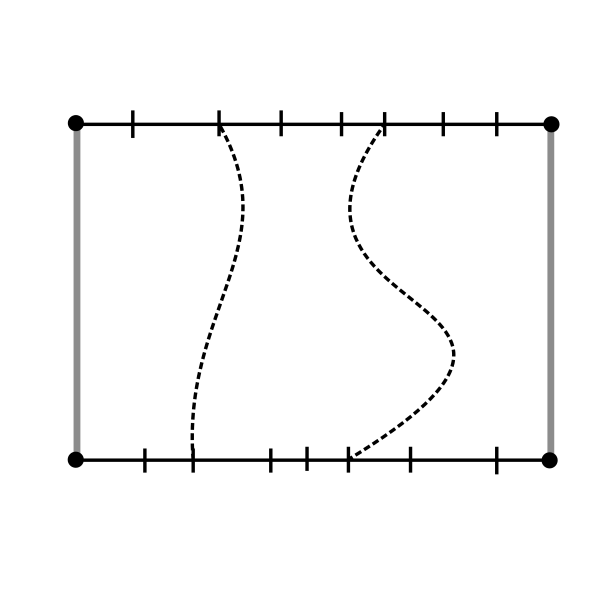}
\fontsize{12pt}{14pt}\selectfont
\put(21,69){\makebox(0,0)[cc]{$Q$}}
\end{overpic}
\caption{Construction of $\Omega_0$ and $Q$ from $\omega,\alpha$ and $\nu_1,\pi_n$}
\label{fig.Lone.Interlace}
\end{center}
\end{figure}

\begin{lem}[]\label{lem.Lone.Crit}
If both prime ends $\nu_i^*$ and $\pi_i^*$
belong to $\omega_i^*$, then the corresponding vertex $v(i)$ is a loner.
\end{lem}

\begin{proof}
Let $v:=v(i)$. It follows from $\nu_i^*,\pi_i^*\in\omega_i^*$ that
$\nu_i^+,\pi_i^+\in \omega_i \subset\partial D_v$. If $\pi_i^+\neq\nu_i^+$, the
positively oriented open subarc $\delta_i''$ of $P'_v$ from $\pi_i^+$ to
$\nu_i^+$ lies in $D_v$. If $\pi_i^+=\nu_i^+$, we set $\delta_i'':=\emptyset$.
In both cases the union of $\alpha_i,\pi_i,\delta_i''$ and $\nu_i$ is a Jordan
curve which does not intersect the disks $D_u$ with $u\in U_L^+$ and $u \not=v$.
So either $D_v'$ is disjoint to all such disks $D_u$, or one of the disks $D_u$ is
contained in $D_v'$. In the latter case the prime ends $\nu_i^*$ and $\pi_i^*$
cannot both belong to the same set $\omega_i^*$.
\end{proof}

\noindent
\begin{proof}[Proof of Lemma~\ref{lem.Dia.Loner}.]
After these preparations we are ready to harvest the fruits: Assume that
$(\mathcal{P},\mathcal{P}')$ has no loner.
Then, by Lemma~\ref{lem.Lone.NuPiToOmega}, the endpoint $\nu_i^+$ of the arc
$\nu_i$ must lie on $\omega$ and hence $\nu_i$ is associated with a prime end
$\nu_i^*\in\omega^*$.
If $\nu_i^*\in \omega_k^*$, we choose the smallest such $k$ and set $l(i):=k$.
Similarly, we denote by $r(i)$ the smallest number $k$ for which $\pi_i^*\in\omega_k^*$.
\label{def.number}

Lemma~\ref{lem.Lone.Interlace} tells us that $r(i)\ge l(i)$ and $l(i+1)\ge r(i)$.
In conjunction with Lemma~\ref{lem.Lone.Crit} we conclude that the first condition
implies $r(i)\ge l(i)+1$. Starting with $l(1)\ge 1$, we get inductively that
$r(i)\ge i+1$ for $i=1,\ldots,n$, ending up with the contradiction $r(n)\ge n+1$.
This proves Lemma~\ref{lem.Dia.Loner} in the regular case.
\end{proof}

\subsection{The Critical Case} \label{subsec.LoneCrit}

The second case, where we admit that $\alpha\cap\omega\not=\emptyset$, will be
reduced to the regular case by an appropriate deformation of the Jordan arc
$\alpha$.

\begin{defn}\label{def.Lone.Reg}
A contact point $y\in Y$ is called \emph{regular} if $y\notin\omega$, otherwise
it is said to be \emph{critical}.
\end{defn}

If $y\in Y$ is a critical contact point, then $y\in \alpha\cap\omega\not=\emptyset$,
and hence $y\in \omega_j$ for some $j$. Since $y=\partial{D_u'} \cap\partial {D_v'}$
with some $u\in U_L^-$ and $v=v(i)\in U_L^+$, we see that $y$ cannot be an endpoint
of $\omega_j$ (turning point of $\omega$) -- otherwise $D_u'$ would not be
contained in $\Omega$. Moreover, the circles $\partial D_u'$, $\partial D_v'$, and
$\omega_j$ must be mutually tangent at $y$.
The arc $\omega_j$ is a subset of the circle $\partial D_w$ (with $w=v(j)\in
U_L^+$). Hence either $D'_v\subset D_w$ (with $D'_v=D_w$ admitted) or $D_w$ is
a proper subset of $D'_v$.
\smallskip\par

In the next step we modify the Jordan arc $\alpha$ in a neighborhood of $y$
and redefine the arcs $\nu_i$ and $\pi_i$ (connecting $y$ with $\omega$) introduced
in the regular case.

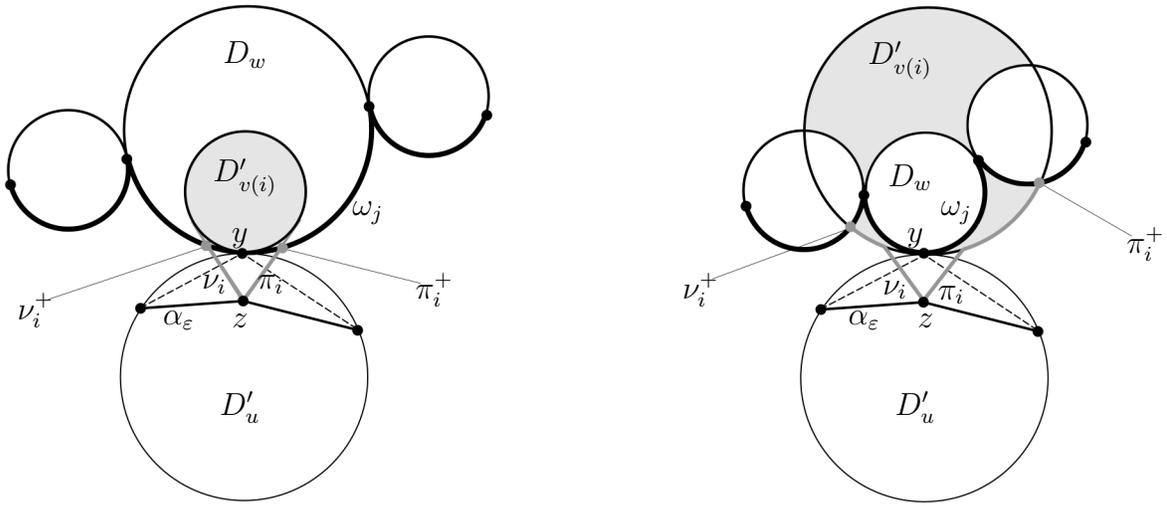
\begin{figure}[H]
\begin{center}
\begin{overpic}{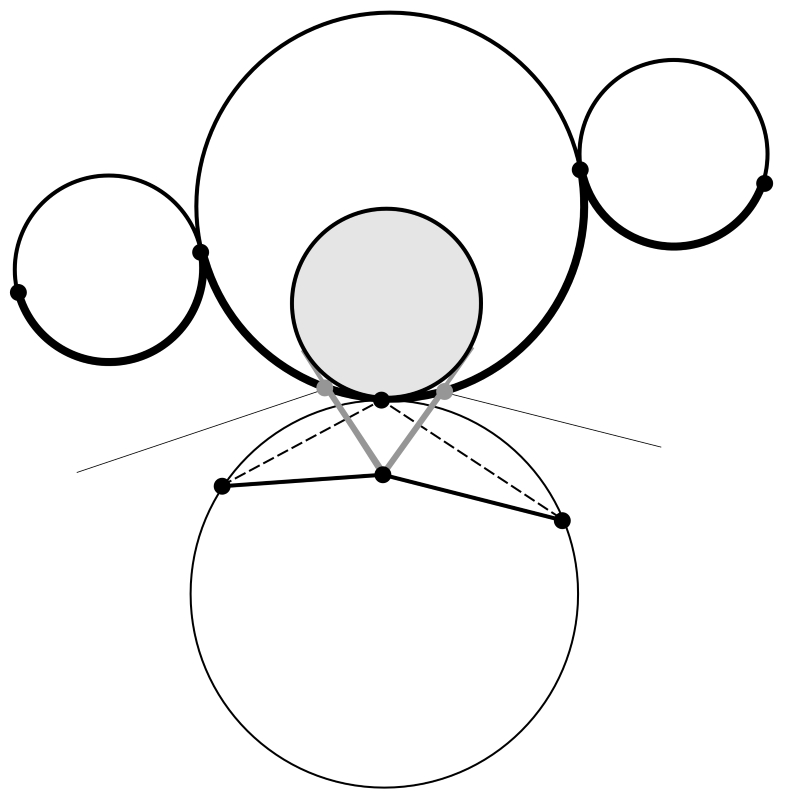}
\fontsize{12pt}{14pt}\selectfont
\put(48,89){\makebox(0,0)[cc]{$D_w$}}
\put(48,65){\makebox(0,0)[cc]{$D'_{v(i)}$}}
\put(47,20){\makebox(0,0)[cc]{$D'_u$}}
\put(47,53){\makebox(0,0)[cc]{$y$}}
\put(72,58){\makebox(0,0)[cc]{$\omega_j$}}
\put(42,44){\makebox(0,0)[cc]{$\nu_i$}}
\put(53,44){\makebox(0,0)[cc]{$\pi_i$}}
\put(7,39){\makebox(0,0)[cc]{$\nu_i^+$}}
\put(85,43){\makebox(0,0)[cc]{$\pi_i^+$}}
\put(47,37){\makebox(0,0)[cc]{$z$}}
\put(35,37){\makebox(0,0)[cc]{$\alpha_{\varepsilon}$}}
\end{overpic}
\hspace{0.1\textwidth}
\begin{overpic}{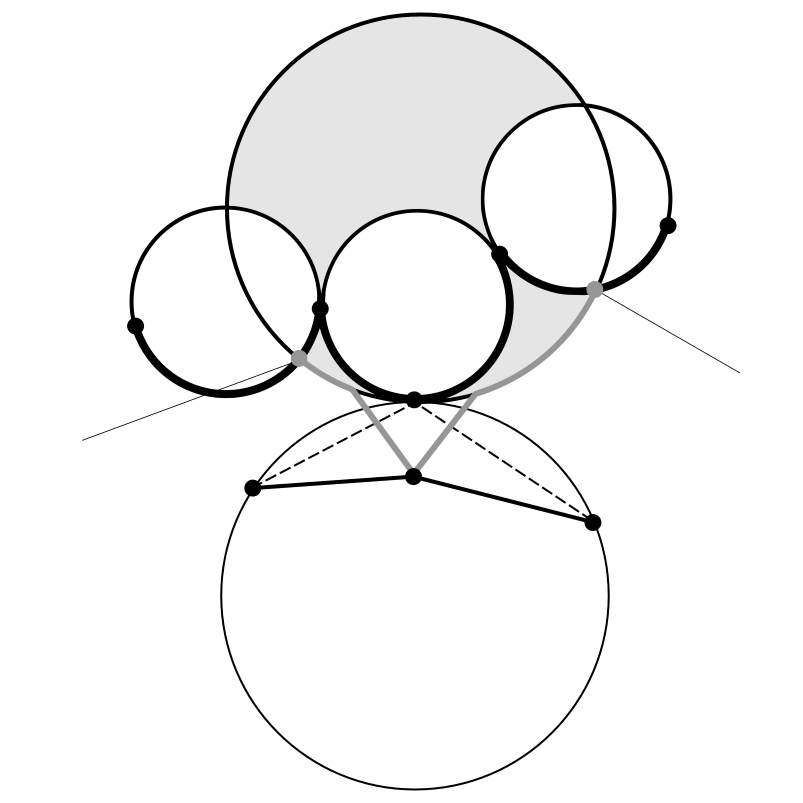}
\fontsize{12pt}{14pt}\selectfont
\put(47,88){\makebox(0,0)[cc]{$D'_{v(i)}$}}
\put(49,65){\makebox(0,0)[cc]{$D_w$}}
\put(50,20){\makebox(0,0)[cc]{$D'_u$}}
\put(50,53){\makebox(0,0)[cc]{$y$}}
\put(58,58){\makebox(0,0)[cc]{$\omega_j$}}
\put(46,43){\makebox(0,0)[cc]{$\nu_i$}}
\put(57,42){\makebox(0,0)[cc]{$\pi_i$}}
\put(8,43){\makebox(0,0)[cc]{$\nu_i^+$}}
\put(95,52){\makebox(0,0)[cc]{$\pi_i^+$}}
\put(52,37){\makebox(0,0)[cc]{$z$}}
\put(40,37){\makebox(0,0)[cc]{$\alpha_{\varepsilon}$}}
\end{overpic}
\caption{Modification of $\alpha$ and definition of the arcs $\nu_i$ and $\pi_i$
for critical contact points $y$}
\label{fig.Lone.Mod}
\end{center}
\end{figure}
\vspace{-4mm}
\noindent

Let $\varepsilon$ be a sufficiently small positive number. Denote by $z$ the
$\varepsilon$-shift of $y$ in the direction of the center of $D_u'$.
Append to $D_v'$ an equilateral open triangular domain $T$ with one vertex at
$z$, two vertices on $\partial D_v'$, and symmetry axis through $y$ and $z$
(see Figure~\ref{fig.Lone.Mod}).

For $y\notin\{y_0,y_l\}$ let $\nu_i$ (and $\pi_i$)
be the largest positively (negatively) oriented subarc of
$\partial (D_v'\cup T)$ which has initial point $z$ and is contained in
$\Omega$. For $y\in\{y_0,y_l\}$ (and only then) it can happen that $y$ is a
boundary point of $G$. Therefore we define $\nu_i:=[z,y]$ in the case $y=y_0$,
and $\pi_i:=[y,z]$ in the case $y=y_l$.
The case $y_0=y_l$ can never occur, because $l\geq 1$.

Denote by $\nu_i^+$ and $\pi_i^+$ the terminal points of $\nu_i$ and $\pi_i$.
Clearly, $\nu_i^+,\pi_i^+\in\omega$, so let $\nu_i^*, \pi_i^*\in \omega^*$ be
their associated prime ends.
\smallskip\par

We see, that the statement of Lemma~\ref{lem.Lone.NuPiToOmega} holds in the
critical case, too. Moreover, for the critical case, Lemma~\ref{lem.Lone.Interlace}
can be proved in exactly the same way as for the regular case, we
just have to apply the adapted definitions of $\nu_i^*$ and $\pi_i^*$.
All what is missing is the following ``critical'' version of Lemma~\ref{lem.Lone.Crit}.

\begin{lem}[]\label{lem.Lone.Crit2}
Assume that $\partial D_v'$ with $v=v(i)\in U_L^+$ contains a critical contact point
$y\in Y\cap \omega$. Then $v$ is a loner if and only if $\nu_i^*$ and $\pi_i^*$ belong
to $\omega_i^*$.
\end{lem}

\begin{proof}
We use the notations introduced above, with $\varepsilon>0$ fixed and sufficiently
small. We distinguish two cases.
\smallskip\par

\noindent
Case 1. Let $D_v'\subset D_w$ (see Figure~\ref{fig.Lone.Mod}, left). Then $v$ is a
loner if and only if $w=v$, and this holds, if and only if $j=i$ and
$\nu_i^*,\pi_i^*\in\omega_i^*$.
\smallskip\par

\noindent
Case 2. Let $D_w\subset D_v'$ and $D_w \not= D_v'$ (see Figure~\ref{fig.Lone.Mod},
right). Then $D_v'$ intersects at least two ``upper'' disks (namely $D_w$ and one
of its neighbors), so that $v$ is not a loner. According to our construction, we
have $\nu_i^*\leq y^*\leq\pi_i^*$ (where $y^*\in\omega_j^*$ is the
prime end corresponding to $y$ and $w=v(j)$), but both equalities are never
fulfilled at the same time, and $\nu_i^*,\pi_i^*\notin\omega_j^*$ for $w=v(j)$.
Therefore $\nu_i^*\in\omega_m^*$ and $\pi_i^*\in\omega_n^*$ with $m\leq j\leq n$,
but $m<n$, so the prime ends $\nu_i^*$ and $\pi_i^*$ cannot both belong to the same
class $\omega_i^*$.

\end{proof}

Remark. If $D_v'$ has several critical contact points $y\in Y\cap\omega_j$ with
the same arc $\omega_j$, then $D_v'$ must be tangent to $D_w$ with $w=v(j)$ at
two different points. This implies that $D_v'=D_w$, which explains why the
criterion is independent of the choice of $y$.
\smallskip\par

\noindent
After replacing all critical contact points $y_k$ by the shifted points $z_k$,
and modifying the construction of the curve $\alpha$ accordingly,
Lemma~\ref{lem.Dia.Loner} can be proved completely the same way as in the regular
case.
\smallskip\par

In Section~\ref{sec.Struct} we need the following generalization of
Lemma~\ref{lem.Dia.Loner}. We point out that $v(i)=v(j)$ is allowed in assertion (i).

\noindent
\begin{lem}[]\label{lem.Lone.MoreLone}
Let $D_{v(i)}=D_{v(i)}'$ and $D_{v(j)}=D_{v(j)}'$ with $1\le i\le j\le n$.
Then, in each of the following cases {\rm (i)-(iii)}, there exists a loner $v(k)$ which
is different from $v(i)$ and $v(j)$, such that $k$ satisfies the corresponding
conditions:
\begin{itemize}
\itemsep0mm
\item[{\rm(i)}]
If $1\le i<j-1\le n-1$, then $i<k<j$,
\item[{\rm(ii)}]
If $i>1$, then $1\le k<i$,
\item[{\rm(iii)}]
If $j<n$, then $j<k\le n$.
\end{itemize}
\end{lem}

\begin{proof}
The proof differs only slightly from the proof of Lemma~\ref{lem.Dia.Loner}.
For example, in order to prove (i) we need only replace the first inequality
$l(1)\ge 1$ by $l(i+1)\ge i+1$ (which follows from $D_{v(i)}=D_{v(i)}'$) and,
assuming that no loner $v(k)$ with $i<k<j$ exists, proceed inductively for
$k=i+1,\ldots,j$ until we arrive at $r(j)\ge j+1$. The last condition
contradicts $D_{v(j)}=D_{v(j)}'$.

If $v(k)=v(i)$ or $v(k)=v(j)$, we repeat the procedure, replacing $i$ (in the first
case) or $j$ (in the second case) by $k$, respectively. Iterating this a number of
times, if necessary, we eventually find a loner $v(k)$ which is different from
$v(i)$ and $v(j)$, because for all $m=2,3,...,n-1$ we have $v(m-1)\neq v(m)$ and
$v(m)\neq v(m+1)$.
\end{proof}

\section{Structure of Upper Neighbors} \label{sec.Struct}

In this section we analyze the structure of the set of upper neighbors $U_L^+$
and its subset of loners in more detail.
\smallskip\par

Two consecutive (non-oriented) edges $e_{j-1}$ and $e_j$ of $L=(e_0,\ldots,e_l)$ can
be represented as $e_{j-1}=e(u,v)$ and $e_j=e(v,w)$. The third edge of
the face $f(u,v,w)$ is considered as oriented from $u$ to $w$, and we set
$e_j^0:=\langle u,w\rangle$. The set of edges $e_j^0$ splits into
two classes. We define $E_L^-$ as the set of those $e_j^0$ where the face
$\langle u,v,w\rangle$ is oriented counter-clockwise, whereas $E_L^+$ consists
of those edges with clockwise orientation of $\langle u,v,w\rangle$, respectively.
After renumbering the elements of $E_L^-$ and $E_L^+$, without changing their order,
we get two sequences of oriented edges $E_L^-=\{e_1^-,\ldots,e_p^-\}$ and
$E_L^+=\{e_1^+,\ldots,e_q^+\}$ (with $p+q=l$), which are called the \emph{sequences
of lower} and \emph{upper accompanying edges} of the crosscut $L$, respectively.
\label{def.Accompany.Edge}

Here are some basic properties of $E^-_L,E^+_L$, which follow quite easy from the
definition of $L$ (proofs are left as exercises).
The \emph{oriented} edges in $E_L^-\cup E_L^+$ are pairwise disjoint;
the corresponding non-oriented edges can appear at most twice, and either both in
$E^-_L$ or both in $E^+_L$. Two consecutive edges $e_{j-1}^\pm$ and $e_j^\pm$ are
linked at a common vertex. The vertex set of all edges in $E_L^+$ is precisely the
set $U_L^+$ of upper neighbors of $L$.

Figure~\ref{fig.Accomp} shows two examples. The involved crosscut on the right
models the fourth generation of the Hilbert curve. With the exception of boundary
edges, all edges in $E_L^-$ (lighter color) and in $E_L^+$ (darker color) appear
with both orientations (not shown in the picture).

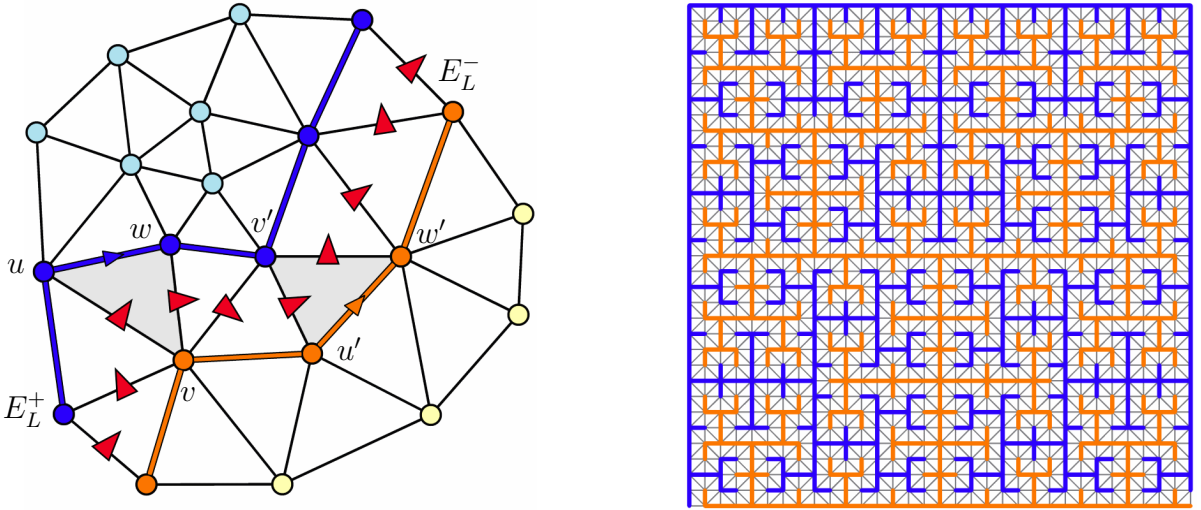
\begin{figure}[H]
\begin{center}
\begin{overpic}{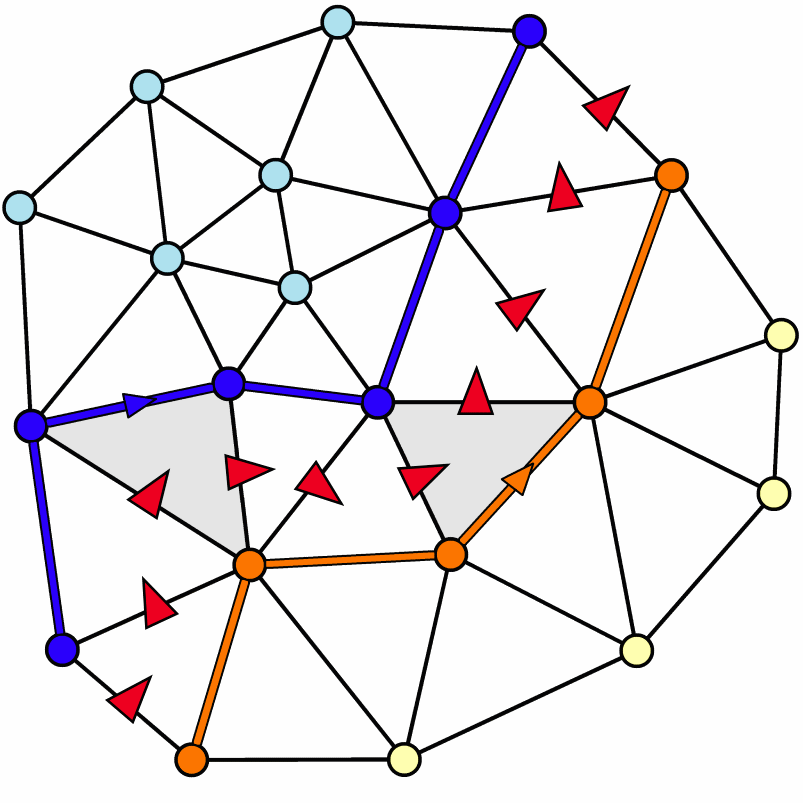}
\fontsize{12pt}{14pt}\selectfont
\put(0,48){\makebox(0,0)[rc]{$u$}}
\put(25,55){\makebox(0,0)[rc]{$w$}}
\put(32,23){\makebox(0,0)[cu]{$v$}}
\put(61,32){\makebox(0,0)[lc]{$u'$}}
\put(79.5,55){\makebox(0,0)[cc]{$w'$}}
\put(46.5,57){\makebox(0,0)[cc]{$v'$}}
\put(85,86){\makebox(0,0)[cc]{$E^-_L$}}
\put(0,20){\makebox(0,0)[cc]{$E^+_L$}}
\end{overpic}
\hspace{0.1\textwidth}
\includegraphics{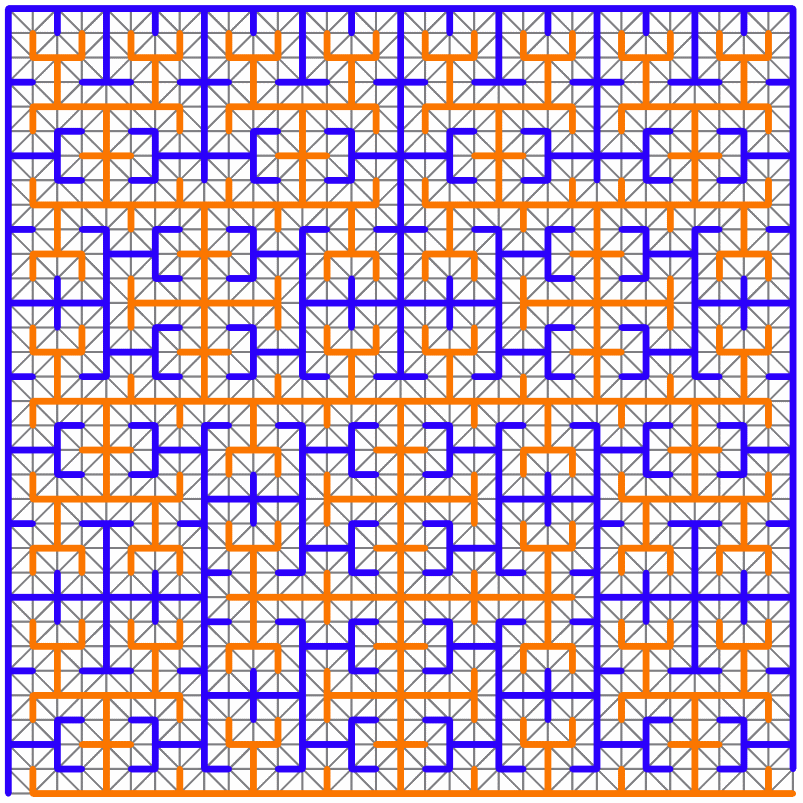}
\caption{The upper and the lower accompanying edges of a crosscut}
\label{fig.Accomp}
\end{center}
\end{figure}
\vspace{-4mm}
\noindent

When we arrange the elements of $U_L^+$ in the order they are met along the edge path
$E_L^+$ we get the \emph{sequence $S_L^+$ of upper accompanying vertices}.
A similar definition is made for the \emph{sequence $S_L^-$ of lower accompanying
vertices}.\label{def.Accompany.Vertex}
The geometry of circle packings causes some combinatorial obstructions for these
sequences.

\begin{lem}[]\label{lem.ForbiddenPatt}
The sequence $S_L^+$ of upper accompanying vertices cannot contain the pattern
$(\ldots,u,\ldots,v,\ldots,u,\ldots, v, \ldots)$ with $u \not=v$.
\end{lem}

\begin{proof}
If the sequence $S_L^+$ contains the pattern $(\ldots,u,\ldots,v,\ldots,u,\ldots)$,
the oriented curve $\omega$ has three subarcs $\omega_i,\omega_j,\omega_k$ with
$i<j<k$ such that $\omega_i, \omega_k \subset \partial D_u$ and
$\omega_j \subset \partial D_v$. But then $\omega$ cannot contain a subarc of
$\partial D_v \setminus \omega_j$ (see Figure~\ref{fig.LoopLone}, left), which
would be necessary to append another $v$ to the sequence.
\end{proof}

\begin{figure}[H]
\begin{center}
\begin{overpic}{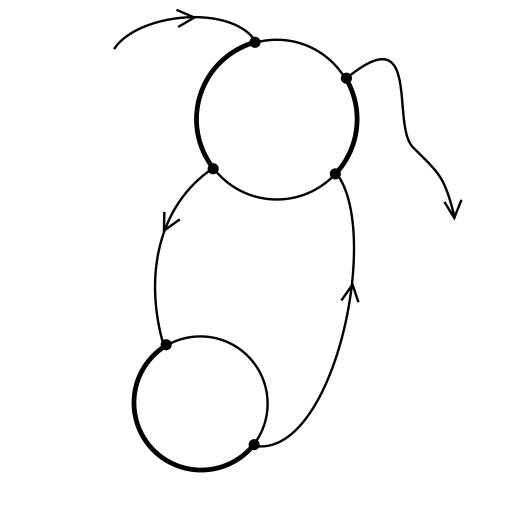}
\fontsize{12pt}{14pt}\selectfont
\put(20,88){\makebox(0,0)[ru]{$\omega$}}
\put(52,79){\makebox(0,0)[cc]{$D_u$}}
\put(39,22){\makebox(0,0)[cc]{$D_v$}}
\end{overpic}
\hspace{0.15\textwidth}
\begin{overpic}{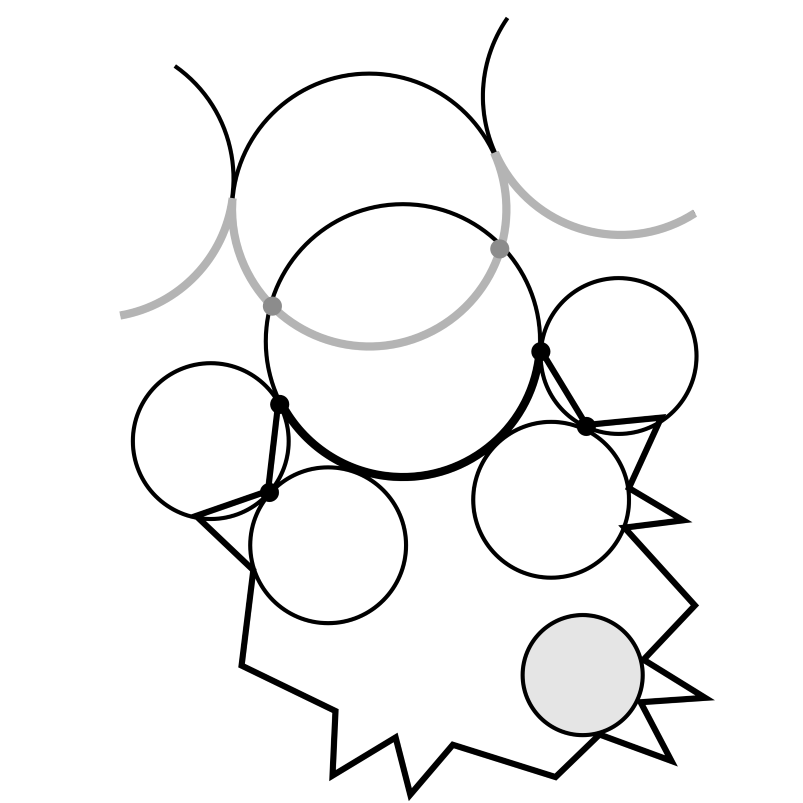}
\fontsize{12pt}{14pt}\selectfont
\put(11,58){\makebox(0,0)[cb]{$\omega$}}
\put(45,82){\makebox(0,0)[cc]{$D_v$}}
\put(49,66){\makebox(0,0)[cc]{$D'_v$}}
\put(52,44){\makebox(0,0)[cu]{$\delta'$}}
\put(33,49){\makebox(0,0)[rc]{$y_{i-1}$}}
\put(35,36.5){\makebox(0,0)[lu]{$y_{i}$}}
\put(69,43){\makebox(0,0)[lu]{$y_{j}$}}
\put(70,57){\makebox(0,0)[lc]{$y_{j+1}$}}
\put(36,10){\makebox(0,0)[cc]{$\alpha'$}}
\put(72,16){\makebox(0,0)[cc]{$D'_u$}}
\end{overpic}
\caption{Illustrations to Lemma~\ref{lem.ForbiddenPatt} and
Lemma~\ref{lem.Struct.LoopLone}}
\label{fig.LoopLone}
\end{center}
\end{figure}
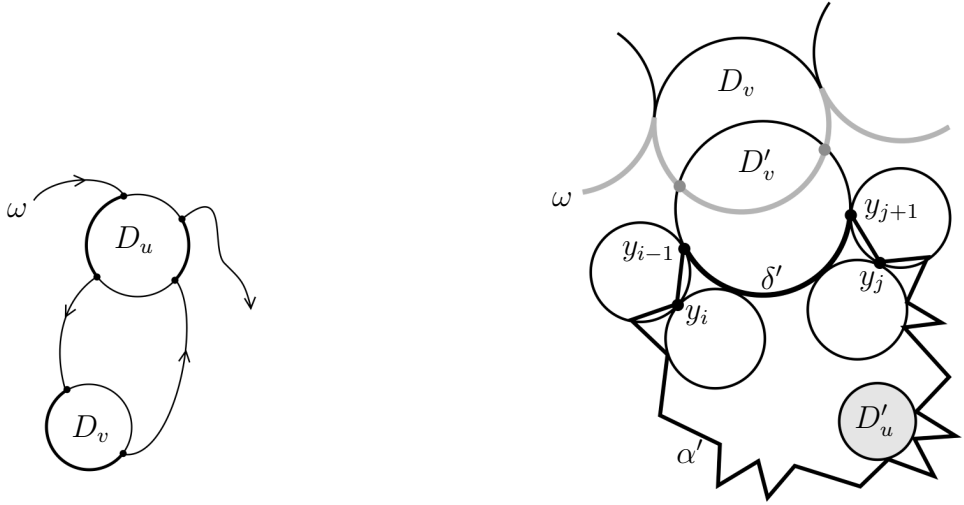

\begin{defn}\label{def.Struc.MultVert}
A vertex $v\in U_L^+$ which appears only once in the sequence $S_L^+$ is called
\emph{simple}, the other elements in $U_L^+$ are said to be \emph{multiple} vertices.
\end{defn}

If $v$ is a multiple vertex in $U_L^+$, there are sequences
$M:=\{e_i^+,e_{i+1}^+,\ldots,e_j^+\}\subset E_L^+$ of accompanying edges
such that $v$ is the initial vertex of $e_i^+$, as well as the terminal vertex of
$e_j ^+$ with $i<j$. Any such sequence is called a \emph{loop} for $v$.
We say that a loop $M$ \emph{meets a vertex} $u$, if $u$ is adjacent to an edge in
$M$ and $u\neq v$. The \emph{set of vertices} met by $M$ is denoted by $V_M$.
A loop $M$ also generates a \emph{sequence of vertices} $U_M=\{v,v_1,\ldots,v_m,v\}$
when we arrange the elements of $V_M$ in the order they are met along the edge path $M$.
\label{def.Loop}

\begin{lem}[]\label{lem.Struct.LoopVert}
Every loop $M$ of a multiple vertex $v$ meets a simple vertex $u$.
\end{lem}

\begin{proof}
We consider the sequence $U_M=\{v,v_1,\ldots,v_m,v\}$ of vertices in $V_M$, arranged
in the order as they are met by the edge path $M$. Let $w$ denote the element of
this sequence with the earliest second appearance (this does \emph{not} mean the
first element which appears twice). Since $w$ cannot appear twice in direct
succession, there exists a vertex $u$ in between the first two symbols $w$.

In order to show that $u$ is a simple vertex, we remark that $U_M$ is a subsequence
of the sequence $S_L^+$ of upper accompanying vertices. By definition of $w$, there
cannot be a second $u$ in $S_L^+$ between the two symbols $w$ next to $u$, and by
Lemma~\ref{lem.ForbiddenPatt}, the sequence $S_L^+$ cannot contain a second $u$ outside
these two $w$\,s.
\end{proof}

Since loners are vertices in $U_L^+$, it makes sense to speak of simple and multiple
loners.

\begin{lem}[]\label{lem.Struct.LoopLone}
Let $v$ be a multiple loner with $D_v'\not=D_v$. If $u\neq v$ is a vertex which is met
by a loop of $v$, then $u$ is a loner and $D_u'\cap D_u=\emptyset$.
\end{lem}

\begin{proof}
Let $M$ be a loop of $v$ with $U_M=\{v,v_1,...,v_m,v\}$. Let $i$ be the smallest
index, so that $y_i$ is a contact point of $v_1$, and let $j$ be the largest
index, so that $y_j$ is a contact point of $v_m$. According to the ordering
of $Y$ and $U_M$ (as subsequences of $S^+_L$), $y_{i-1}$ and $y_{j+1}$ are
contact points of $D'_v$. Let $u\in\{v_1,...,v_m\}$ with $u\neq v$.

The disk $D'_u$ is enclosed by the union of the subarc $\delta':=\delta[y_{i-1},y_{j+1}]$
of $D'_v$ and the subarc $\alpha'\subset\alpha$ which connects the points
$y_{i-1}$ and $y_{j+1}$ on $\alpha$ (see Figure~\ref{fig.LoopLone}). Since $v$ is a
loner with $D'_v\neq D_v$, it is clear that
$y_{i-1},y_{j+1}\notin D_v$, and hence either $D_v'\cap D_v=\emptyset$ or
$\partial D_v' \cap\partial D_v$ consists of one or two points.
In both cases $\delta'$ does not intersect $D_v$.
Therefore the union $\alpha'\cup\delta'$ is contained in $\overline{\Omega}$,
hence $u$ is a loner. In particular $D_u'\cap D_u=\emptyset$, which proves the
last assertion.
\end{proof}

Combining Lemma~\ref{lem.Dia.Loner}, Lemma~\ref{lem.Lone.MoreLone} (applied recursively),
Lemma~\ref{lem.Struct.LoopVert} and Lemma~\ref{lem.Struct.LoopLone} (applied recursively),
the essence of this section can be summarized in the following lemma.

\begin{lem}[]\label{lem.Struct.SimpLone}
Let $(\mathcal{P},\mathcal{P}')$ be an admissible pair of circle packings with crosscut
$L$.
\begin{itemize}
\itemsep0mm
\item[{\rm(i)}]
The pair $(\mathcal{P},\mathcal{P}')$ contains a simple loner $v\in U_L^+$.
\item[{\rm(ii)}]
Every loop of a multiple loner $v$ meets a simple loner $u$, and if $D_v'\not=D_v$ then
$D_u'\not=D_u$.
\end{itemize}
\end{lem}

\section{Proof of the Main Theorem} \label{sec.Proof}

After all these preparations we are eventually in a position to prove
Theorem~\ref{thm.CircRigid}.
To begin with, we use the concept of loners and combinatorial surgery
to modify the crosscut $L$. In every step of this procedure the number
of vertices in $V_L^+$ is reduced. At the end we get a special combinatorial
structure which is called a slit. Roughly speaking, this is a chain of vertices
connecting the alpha-vertex with a boundary vertex. We shall prove that
the disks of both packings coincide along a slit.

Then a subdivision procedure generates a sequence of slits, such that any
accessible boundary vertex appears among their end points.
So we get $D_v'=D_v$ for all accessible $v\in \partial V$, and finally a
well-known theorem tells us that $D_v'=D_v$ for all accessible $v\in V$.

\subsection{Combinatoric Reduction} \label{sec.CombiRed}

Let $L$ be a combinatoric crosscut of the complex $K$. In this section we describe
how a simple vertex $v\in U_L^+$ can be ``shifted'' from $V_L^+$ to $V_L^-$ such
that we get a new crosscut $L'$ with $\big|V_{L'}^+\big|<\big|V_L^+\big|$.
Depending on the properties of $v$ we distinguish three cases.
\medskip\par

\noindent
\label{page.Cases}%
{\bf Case~1.} Let $v\in U_L^+$ be a simple interior vertex.
\smallskip\par

\noindent
{\bf Case~2.} Let $v\in U_L^+$ be a simple boundary vertex, and assume that neither
the initial nor the terminal edge of $L$ are adjacent to $v$.
\smallskip\par

\noindent
{\bf Case~3.} Let $v\in U_L^+$ be a simple boundary vertex, and assume that
either the initial or the terminal edge of $L$ are adjacent to $v$.
\medskip\par

\noindent
Remark. The case where the initial \emph{and} the terminal edge of $L$ are adjacent to
$v$ cannot appear. Indeed, otherwise either $v$ is a multiple vertex (which is
not considered) or all edges adjacent to $v$ must belong to $L$. The latter implies
that $v$ is the only vertex in $V_L^+$, which is not allowed.
\medskip\par

\noindent
{\bf Reduction of Type~1.} In order to modify the crosscut $L=(e_0,e_1,\ldots,e_l)$
in Case~1, we consider the flower $B=B(v)$ of $v$.
Since $v$ is simple, the set of edges adjacent to $v$ consists
of a subsequence $S=(e_i,\ldots,e_j)$ (with $0\le i\le j\le l$) of $L$ and a
complementary sequence, which we denote by $S'=(e_1',\ldots,e_k')$ (with $k\ge 1$).
Replacing in $L$ the sequence $S$ by $S'$, we get a new edge sequence
\[
L' = (e_0, \ldots, e_{i-1},e_1',\ldots,e_k',e_{j+1},\ldots,e_l).
\]
The reader can easily convince herself (see Figure~\ref{fig.CombRed1}, left),
that the sequence $L'$ is a crosscut for $K$ with $\big|V_{L'}^+\big|<\big|V_L^+\big|$.

\begin{figure}[H]
\begin{center}
\begin{overpic}{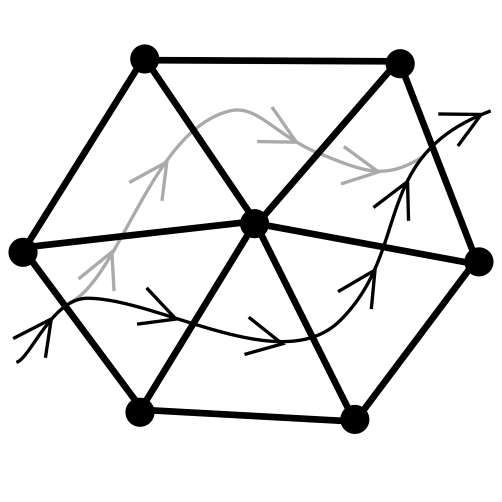}
\fontsize{12pt}{14pt}\selectfont
\put(52,62.5){\makebox(0,0)[cc]{$v$}}
\put(50,24){\makebox(0,0)[cc]{$L$}}
\put(28,72){\makebox(0,0)[cc]{$L'$}}
\end{overpic}
\hfill
\begin{overpic}{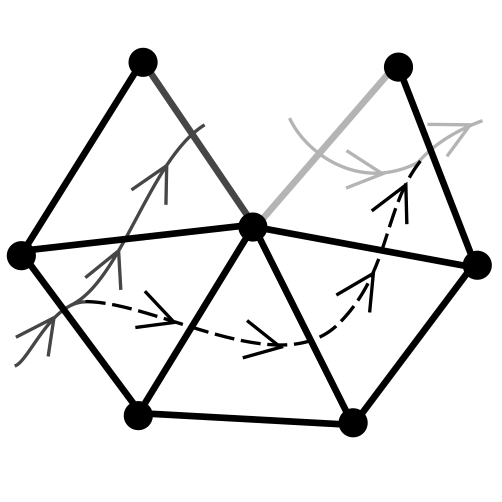}
\fontsize{12pt}{14pt}\selectfont
\put(52,62.5){\makebox(0,0)[cc]{$v$}}
\put(50,24){\makebox(0,0)[cc]{$L$}}
\put(28,72){\makebox(0,0)[cc]{$L'$}}
\put(78,72){\makebox(0,0)[cc]{$L''$}}
\end{overpic}
\hfill
\begin{overpic}{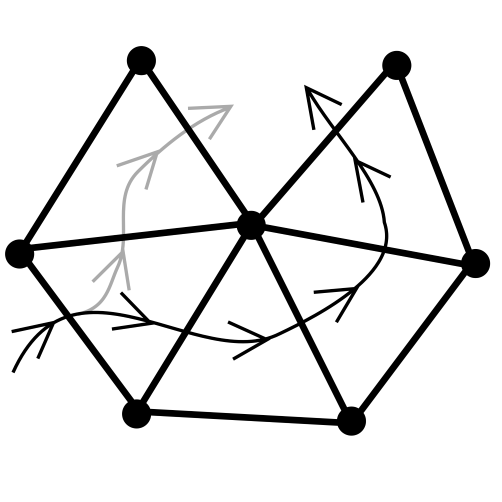}
\fontsize{12pt}{14pt}\selectfont
\put(51.5,62.5){\makebox(0,0)[cc]{$v$}}
\put(50,24){\makebox(0,0)[cc]{$L$}}
\put(36,61){\makebox(0,0)[cc]{$L'$}}
\end{overpic}
\caption{Modification of the crosscut $L$ in Case~1 (left), Case~2 (middle) and
Case~3(right)}
\label{fig.CombRed1}
\end{center}
\end{figure}
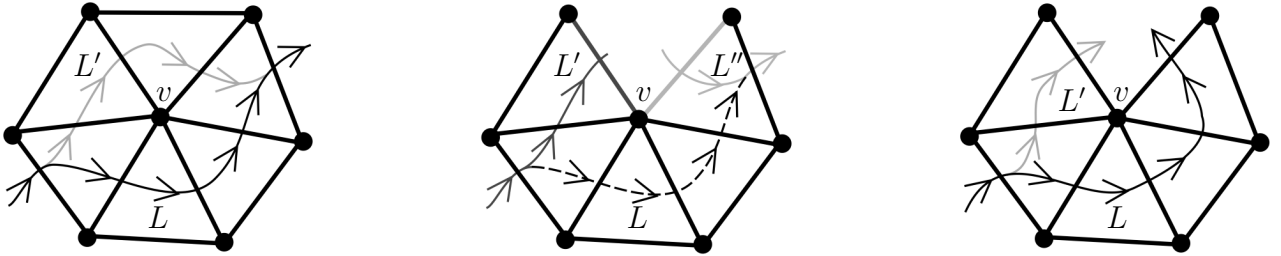
\vspace{-2mm}

\noindent
{\bf Reduction of Type~2.}
In Case~2 the flower of $v$ is incomplete. Nevertheless, the edges in $L$ which are adjacent
to $v$ form again a sequence of consecutive edges in this incomplete flower, because
$v$ is simple. However, the local modification of $L$ in a neighborhood of $v$ described
above does not result in a crosscut $L'$, since the complementary sequence $S'=S'_1\cup S'_2$
consists of exactly two connected components $S'_1=(e_1',\ldots,e_k')$ and
$S'_2=(e_1'',\ldots,e_m'')$ (see Figure~\ref{fig.CombRed1}, middle).
Replacing in $L$ the sequence $S$ by $S'_1$ or $S'_2$, we get a new edge sequence $L'$ or
$L''$, respectively, with
\[
L' = (e_0, \ldots, e_{i-1},e_1',\ldots,e_k').\\
L'' = (e_1'',\ldots,e_m'',e_{j+1},\ldots,e_l).
\]
Both $L'$ and $L''$ are new crosscuts of $K$, but only one ($L'$, say) contains $v_\alpha$
among its upper vertices, so we choose this one as the new crosscut. Clearly
$\big|V_{L'}^+\big|<\big|V_L^+\big|$.

\smallskip\par

\noindent
{\bf Reduction of Type~3.}
If either the initial or the terminal edge of $L$ are adjacent to $v$,
then the Type~1 reduction applied to the incomplete flower of $v$ results in an
admissible crosscut $L'$, which has one vertex (namely $v$) less in $V_{L'}^+$ than
in $V_{L}^+$ (see Figure~\ref{fig.CombRed1}, right).
\smallskip\par

\label{page.remark}
Remark. No matter which type of reduction we used, the sets $U^-_L$ and
$U^-_{L'}$ of lower neighbors before and after the reduction,
respectively, always fulfill $U^-_{L'}\setminus U^-_L=\{v\}$.
\smallskip\par

\noindent
In order to not lose the normalization, we will only reduce vertices different
from $v_\alpha$. This leads to a situation where none of the above reductions
can be applied, namely when $v_\alpha$ is the only simple vertex in $U_L^+$.
This special case will be explored in Section~\ref{sec.Slits}.

\subsection{Slits} \label{sec.Slits}

The next definition and the following lemma describe the situation
when all but exactly one vertex of $V$ are multiple.

\begin{defn}\label{def.Slit}
A combinatoric \emph{slit} of the complex $K=(V,E,F)$ is a sequence
$S=(v_1,v_2,\ldots,v_{s})$ of vertices in $V$ which satisfies the
following conditions (i)--(iv):
\begin{itemize}
\itemsep0mm
\item[{\rm(i)}]
The vertices of $S$ are pairwise different, $v_j\not=v_k$ if $1\le j<k\le s$.
\item[{\rm(ii)}]
For $j=1,\ldots,s-1$, the edges $e_j:=e(v_{j},v_{j+1})$ belong to $E$.
\item[{\rm(iii)}]
For $j=1,\ldots,s$, the vertices $v_{j-1}$ and $v_{j+1}$ are the only neighbors of
$v_j$ in $K$ which belong to $S$ (where $v_{0}:=\emptyset$ and $v_{s+1}:=\emptyset$).
\item[{\rm(iv)}]
The vertex $v_1$ is a boundary vertex, and $v_j$ are interior vertices for
$j=2,\ldots,s$.
\end{itemize}
The vertices $v_1$ and $v_s$ are referred to as the \emph{initial vertex} and the
\emph{terminal vertex} of $S$, respectively.
The sequence $E_S:=(e_1,\ldots,e_{s-1})$ (see (ii)) is said to be the
\emph{edge sequence} of $S$. Note that all $e_j$ are interior edges.
\end{defn}

\begin{lem}[]\label{lem.CutToSlit}
Assume that the interior vertex $v$ is the only simple vertex in $U_L^+$.
Then the sequence of upper accompanying vertices $S_L^+$ has the symmetric form
$(v_1,\ldots,v_{s-1},v,v_{s-1},\ldots,v_1)$ and $S=(v_1,\ldots,v_{s-1},v)$ is a slit.
\end{lem}

\begin{proof}
By definition of a multiple vertex, any vertex in $U_L^+$ except $v$ must appear
at least twice in the sequence $S_L^+$.
If there are vertices which show up twice \emph{at a position left of} $v$,
we choose one, say $u$, whose appearances have minimal distance in the sequence
$S_L^+ = (\ldots, u,\ldots,u,\ldots,v,\ldots)$.
Since neighboring vertices of $S_L^+$ must be different, there exists $w\not=u$ such
that $S_L^+ = (\ldots, u,\ldots,w,\ldots,u,\ldots,v,\ldots)$. Because $v$ is assumed
to be simple and $w$ is a multiple vertex, we have $w\not=v$ and $w$ must appear
again at another place in $S_L^+$.
By Lemma~\ref{lem.ForbiddenPatt} this can only happen in between the two occurrences
of $u$, which is in conflict with the minimal distance property of $u$.

Similarly, the assumption that there exists a vertex which appears in $S_L^+$
twice at a position right of $v$ leads to a contradiction. Hence, with the only
exception of $v$, any vertex of $U_L$ appears in $S_L^+$ exactly once on either
side of $v$. Applying Lemma~\ref{lem.ForbiddenPatt} again, we see that the ordering
of the vertices left of $v$ must be reverse to the ordering on the right of $v$,
so that $S_L^+$ has the symmetric form claimed in the lemma.

\smallskip\par

Moreover we have shown that $v_1,\ldots,v_{s-1},v$ are pairwise different, which is
condition (i) of Definition~\ref{def.Slit}. The second condition (ii) is trivial.

In order to verify condition~(iv), it remains to show that $v_j$ is an interior
vertex for $j=2,\ldots,{s-1}$, because $v_1$ is obviously a boundary vertex, while
$v_s:=v$ is an interior vertex, by assumption.
Assume $v_j$ is a boundary vertex. The flower of $v_j$ is incomplete and it is clear
that $v_{j-1}$ and $v_{j+1}$ are neighbors of $v_j$. On the one hand, the subsequence
$(v_{j-1},v_j,v_{j+1})$ of $S^+_L$ forces the crosscut $L$ to look locally like
shown in Figure~\ref{fig.Slit} left. On the other hand, the subsequence
$(v_{j+1},v_j,v_{j-1})$ of $S^+_L$ forces $L$ to look locally like shown in
the middle of Figure~\ref{fig.Slit}, a contradiction. Hence $v_j$ must be an interior
vertex and its flower must look qualitatively like shown in Figure~\ref{fig.Slit} right.

\begin{figure}[H]
\begin{center}
\begin{overpic}{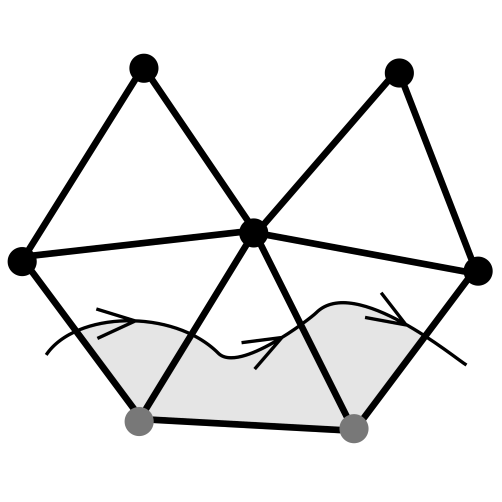}
\fontsize{12pt}{14pt}\selectfont
\put(52,64){\makebox(0,0)[cc]{$v_j$}}
\put(-1,41){\makebox(0,0)[cc]{$v_{j-1}$}}
\put(104,51){\makebox(0,0)[cc]{$v_{j+1}$}}
\put(52,88){\makebox(0,0)[cc]{$V^+_L$}}
\put(52,4){\makebox(0,0)[cc]{$V^-_L$}}
\put(98,28){\makebox(0,0)[cc]{$L$}}
\end{overpic}
\hspace{0.1\textwidth}
\begin{overpic}{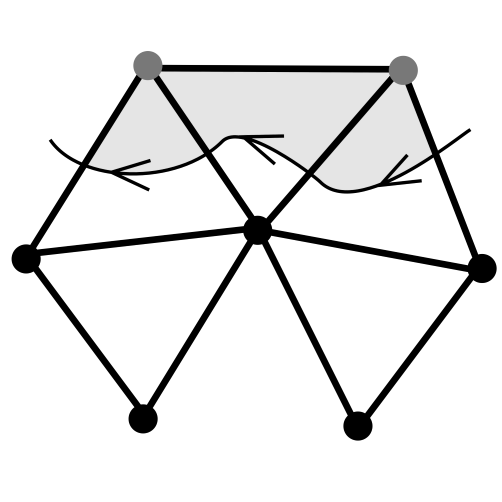}
\fontsize{12pt}{14pt}\selectfont
\put(52,62){\makebox(0,0)[cc]{$v_j$}}
\put(-1,42){\makebox(0,0)[cc]{$v_{j-1}$}}
\put(105,51){\makebox(0,0)[cc]{$v_{j+1}$}}
\put(52,15){\makebox(0,0)[cc]{$V^+_L$}}
\put(52,93.5){\makebox(0,0)[cc]{$V^-_L$}}
\put(5,72){\makebox(0,0)[cc]{$L$}}
\end{overpic}
\hspace{0.1\textwidth}
\begin{overpic}{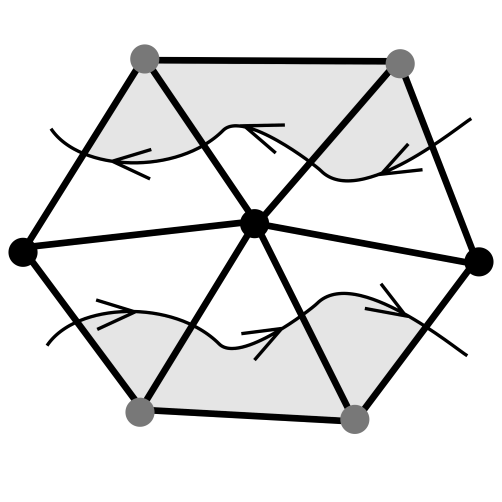}
\fontsize{12pt}{14pt}\selectfont
\put(52,64){\makebox(0,0)[cc]{$v_j$}}
\put(-1,43.5){\makebox(0,0)[cc]{$v_{j-1}$}}
\put(105,52){\makebox(0,0)[cc]{$v_{j+1}$}}
\put(52,96){\makebox(0,0)[cc]{$V^-_L$}}
\put(52,4){\makebox(0,0)[cc]{$V^-_L$}}
\put(98,28){\makebox(0,0)[cc]{$L$}}
\put(5,72){\makebox(0,0)[cc]{$L$}}
\end{overpic}
\caption{A sequence $S_L^+$ with only one simple interior vertex generates a slit}
\label{fig.Slit}
\end{center}
\end{figure}
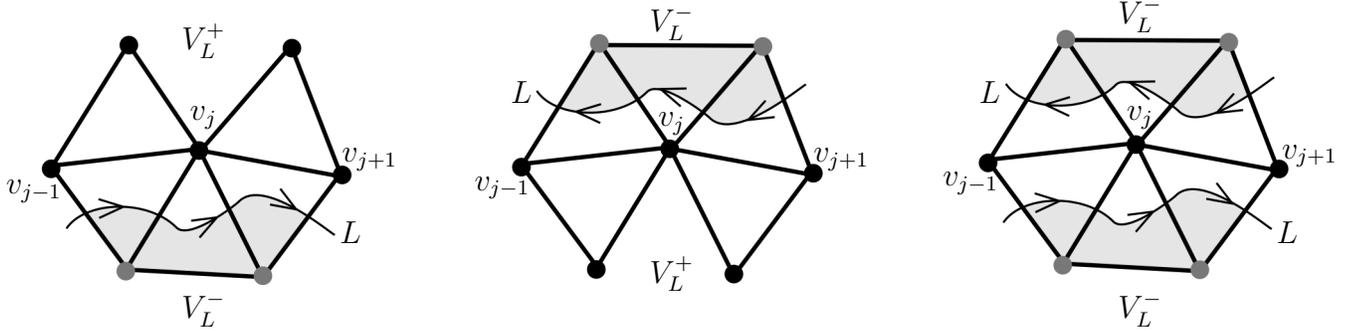
\vspace{-4mm}
\noindent

To verify condition~(iii) let $j\in\{2,...,s-1\}$ be fixed. Looking at the behavior
of the crosscut $L$ in the flower of $v_j$, it becomes clear that any edge
$e(v_{j-1},v_{j+1})$ (with the convention $v_s:=v$) belonging to $E$ must be contained
in $L$ twice, a contradiction. Furthermore, all other neighbors of $v_j$ belong
to $V_L^-$ and hence not to $V_L^+ \supset S^+_L$. A similar result can be derived by
looking at the local behavior of $L$ in the flower of $v$ and the incomplete flower
of $v_1$, now using the subsequences $(v_{s-1},v_s,v_{s-1})$ and $(v_1,v_2,...,v_2,v_1)$
of $S^+_L$, respectively.
\end{proof}

The following lemma explains why we are interested in slits.

\begin{lem}[]\label{lem.Upper.Goal}
Let $(\mathcal{P},\mathcal{P}')$ be an admissible pair of circle packings
for the complex $K$ with crosscut $L$ and alpha-vertex $v_\alpha$. Then there exists
a slit $S=(v_1,\ldots,v_{s},v_\alpha)\subset V_L^+$ with terminal vertex $v_\alpha$
such that $D_v'=D_v$ for all $v\in S$.
\end{lem}

\begin{proof}
To begin with, we invoke Lemma~\ref{lem.Struct.SimpLone}, which tells us
that the pair $(\mathcal{P},\mathcal{P}')$ has a simple loner $v_\lambda$.
The idea is to use the reduction procedures of the last section to shift
$v_\lambda$ from $V_L^+$ to $V_L^-$ which results in a new crosscut $L'$.

As we remarked earlier (on page~\pageref{page.remark}), the one and only lower
neighbor of $L'$ which has not already been a lower neighbor of $L$ is the
simple loner $v_\lambda$.
Therefore Lemma~\ref{lem.Loner.Property} guarantees that $L'$ is admissible
for $(\mathcal{P},\mathcal{P}')$.
In order to find the appropriate type of reduction we distinguish the following
cases:
\medskip\par

\noindent
{\bf Case~1.} There exists a simple interior loner $v_\lambda$ which is different
from the alpha-vertex $v_\alpha$.
\smallskip\par

\noindent
{\bf Case~2.} There exists a simple boundary loner $v_\lambda$.
\smallskip\par

\noindent
{\bf Case~3.} The only simple loner $v_\lambda$ is the alpha-vertex $v_\alpha$.
\medskip\par

In Case~1 we apply the reduction of Type~1, while in Case~2 either the reduction
of Type~2 or Type~3 can be applied, respectively, depending on whether $v_\lambda$
is adjacent to the initial or the terminal edge of $L$, or not.
In any case we get a new combinatoric crosscut $L'$ of $K$.
Applying the reduction in Case~1 and Case~2 recursively as long as possible, the
number of vertices in $V_L^+$ decays in every step at least by one, so that we
eventually arrive at Case~3.

The alpha-vertex $v_\alpha$ is a loner if and only if $D_\alpha'=D_\alpha$. This
implies, by Lemma~\ref{lem.Lone.MoreLone}, that there exists another loner $v_\mu$.
Since $v_\alpha$ is the only simple loner, $v_\mu$ must be a multiple loner.
If $D_\mu'\not=D_\mu$, then according to Lemma~\ref{lem.Struct.SimpLone} (i),
the vertex set $V_M$ of any loop $M$ of $v_\mu$ contains a simple loner, i.e.,
$M$ meets $v_\alpha$. Because $D_\alpha'=D_\alpha$, assertion (ii) of this lemma
tells us that $D_\mu'=D_\mu$.

Applying Lemma~\ref{lem.Lone.MoreLone} and Lemma~\ref{lem.Struct.SimpLone}
repeatedly in this manner, we see that all vertices in
$U_L^+\setminus \{v_\alpha\}$ must be multiple loners and hence that
$D_v'=D_v$ for all $v\in U_L^+$. Furthermore $v_\alpha$ is the only simple
vertex in $U^+_L$, so, by Lemma~\ref{lem.CutToSlit}, we just constructed a
slit $S\subset V_L^+$ with terminal vertex $v_\alpha$.
\end{proof}

In the next step we are going to construct crosscuts from slits. To begin with,
we introduce some more notations.

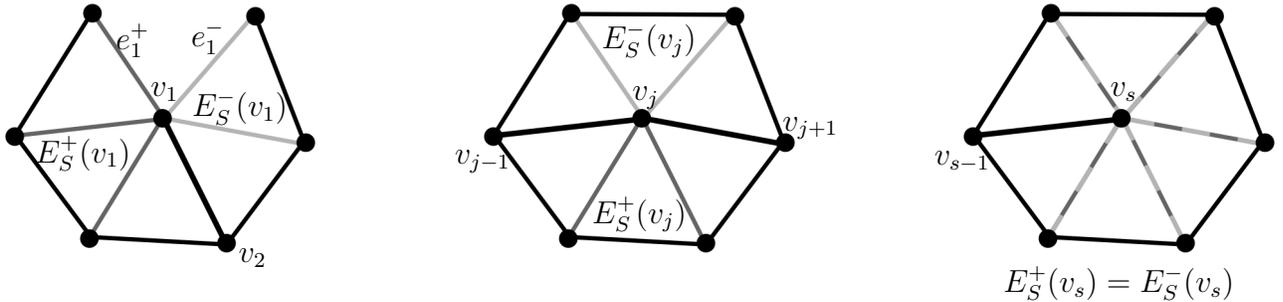
\begin{figure}[H]
\begin{center}
\begin{overpic}{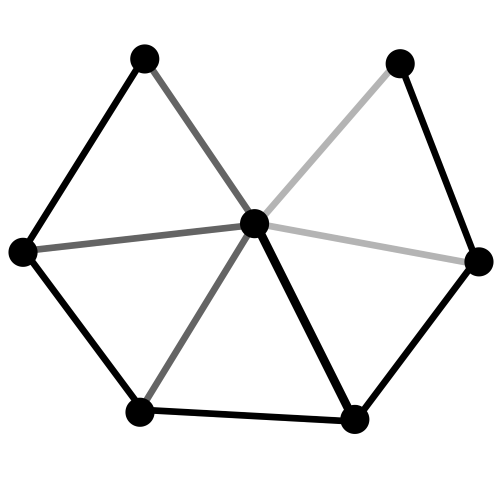}
\fontsize{12pt}{14pt}\selectfont
\put(51.5,64){\makebox(0,0)[cc]{$v_1$}}
\put(79,12){\makebox(0,0)[cc]{$v_2$}}
\put(75,59){\makebox(0,0)[cc]{$E^-_S(v_1)$}}
\put(26,44.5){\makebox(0,0)[cc]{$E^+_S(v_1)$}}
\put(41.5,80){\makebox(0,0)[cc]{$e^+_1$}}
\put(65,81){\makebox(0,0)[cc]{$e^-_1$}}
\end{overpic}
\hfill
\begin{overpic}{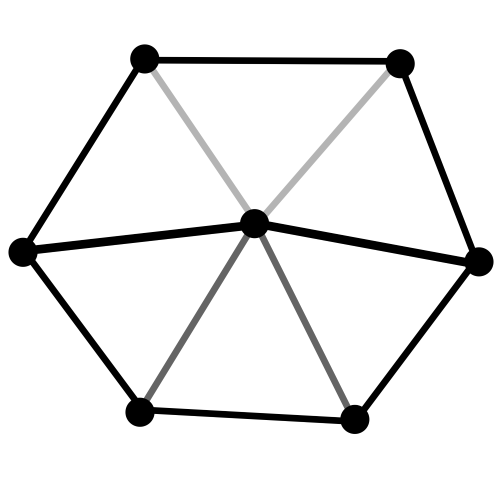}
\fontsize{12pt}{14pt}\selectfont
\put(52,62){\makebox(0,0)[cc]{$v_j$}}
\put(103.5,52.5){\makebox(0,0)[cc]{$v_{j+1}$}}
\put(1,42){\makebox(0,0)[cc]{$v_{j-1}$}}
\put(53,80){\makebox(0,0)[cc]{$E^-_S(v_j)$}}
\put(50,25.5){\makebox(0,0)[cc]{$E^+_S(v_j)$}}
\end{overpic}
\hfill
\begin{overpic}{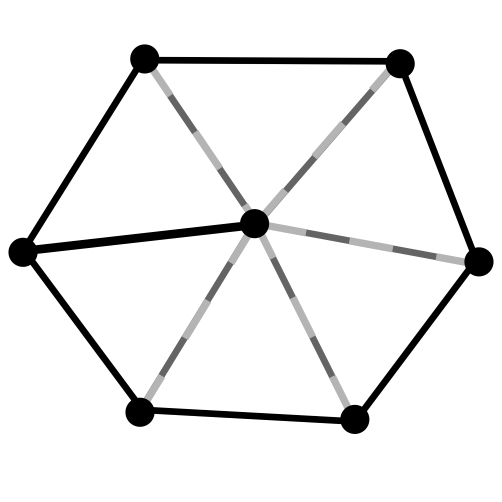}
\fontsize{12pt}{14pt}\selectfont
\put(51.5,64){\makebox(0,0)[cc]{$v_s$}}
\put(1,42){\makebox(0,0)[cc]{$v_{s-1}$}}
\put(50,4){\makebox(0,0)[cc]{$E^+_S(v_s)=E^-_S(v_s)$}}
\end{overpic}
\caption{The left and right neighboring edges of $v$ in a slit $S$}
\label{fig.CutFromSlit}
\end{center}
\end{figure}
\vspace{-4mm}
\noindent

Let $S=(v_1,\ldots,v_s)$ be a slit. For any vertex $v$ in $S$ we define the subsets
$E_S^-(v)$ and $E_S^+(v)$ of $E(v)$ as follows.\label{def.ESv+-}
For $v=v_1$, the (boundary) vertex $v_1$ has two adjacent boundary edges $e_1^-$ and
$e_1^+$ in $E(v_1)$, such that $e_1^-$ is the predecessor of $e_1^+$ in the chain of
boundary edges. We set (the meaning of the inequalities is explained on
page~\pageref{page.DefEdgeOrd})
\begin{align*}
E_S^-(v_1)&:=\{e\in E(v_1): e(v_1,v_2)< e\le e_1^-\},\\
E_S^+(v_1)&:=\{e\in E(v_1): e_1^+\le e < e(v_1,v_2)\}.
\end{align*}
If $v=v_j$, with $j=2,\ldots s-1$, we define
\begin{align*}
E_S^-(v_j)&:=\{e\in E(v_j): e(v_{j},v_{j+1}) < e < e(v_{j-1},v_j)\},\\
E_S^+(v_j)&:=\{e\in E(v_j): e(v_{j-1},v_j) < e < e(v_j,v_{j+1})\},
\end{align*}
and for the terminal vertex $v_s$ of $S$ we let
\[
E_S^-(v_s)=E_S^+(v_s):=\{e\in E(v_s): e(v_{s-1},v_{s}) < e < e(v_{s-1},v_s)\}.
\]
The edges in
\[
E_S^-:=\medcup_{j=1}^{s-1} E_S^-(v_j) \text{\ and\ }
E_S^+:=\medcup_{j=1}^{s-1} E_S^+(v_j)
\]
are called the
\emph{left} and the \emph{right neighbors of} $S$,
respectively.\label{def.ES+-} Note that condition (iii) in
Definition~\ref{def.Slit} guarantees that every edge $e$ which is
a neighbor of a slit $S$ has \emph{exactly one} adjacent vertex
in $S$.

\smallskip\par

\begin{lem}[]\label{lem.SlitToCut}
If $S=(v_1,\ldots,v_s,v)$ is a slit in $K$, then there exists a combinatoric
crosscut $L$ such that $v\in S_L^+$, and
$S_L^-=(v_1,\ldots,v_{s-1},v_s,v_{s-1},\ldots,v_1)$
is the sequence of \emph{lower accompanying vertices} of $L$.
\end{lem}

\begin{proof}
Walking along the slit $S$ from $v_1$ to $v_{s}$ and back to $v_1$, we build
the crosscut $L$ from the concatenation of the edge sequences
\[
E_{S}^-(v_{1}),\ \ldots,\ E_{S}^-(v_{s}),\ e(v_s,v),\ E_{S}^+(v_{s}),\ \ldots,\
E_{S}^+(v_{1}).
\]
It is easy to see that all edges in $L$ are pairwise different,
so that $L$ satisfies condition (i) of Definition~\ref{def.Cuts.Cross}.
Condition (ii) can easily be verified and (iv) is obvious. In
order to prove (iii) we assume that three edges of $L$ would
form a face of $K$. Since these edges are neighbors of $S$,
exactly one vertex of every edge must belong to $S$, which is
impossible.

The construction also guarantees that the sequence $S_L^-$ of lower accompanying
edges of $L$ has the desired form and that $v$ belongs to $S_L^+$ (see, for
example, Figure~\ref{fig.CutFromSlits}, left).
\end{proof}

\begin{figure}[H]
\begin{center}
\begin{overpic}{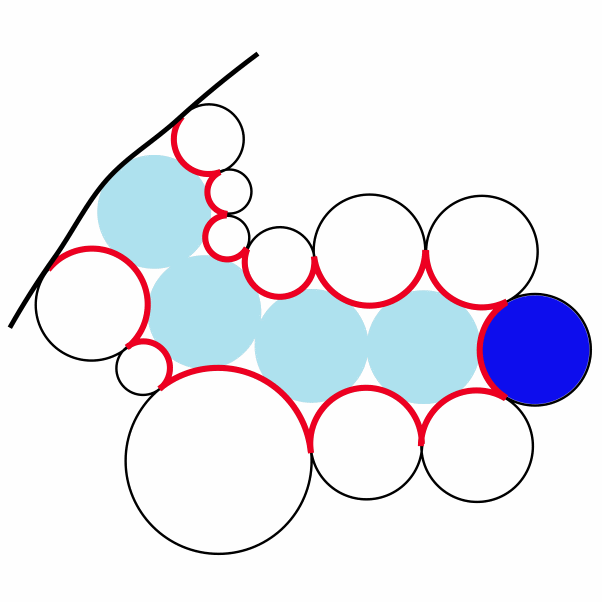}
\fontsize{12pt}{14pt}\selectfont
\put(51.5,42.5){\makebox(0,0)[cc]{$G^-$}}
\put(72,8){\makebox(0,0)[cc]{$G^+$}}
\put(55,80){\makebox(0,0)[ru]{$G^+$}}
\end{overpic}
\hspace{0.03\textwidth}
\begin{overpic}{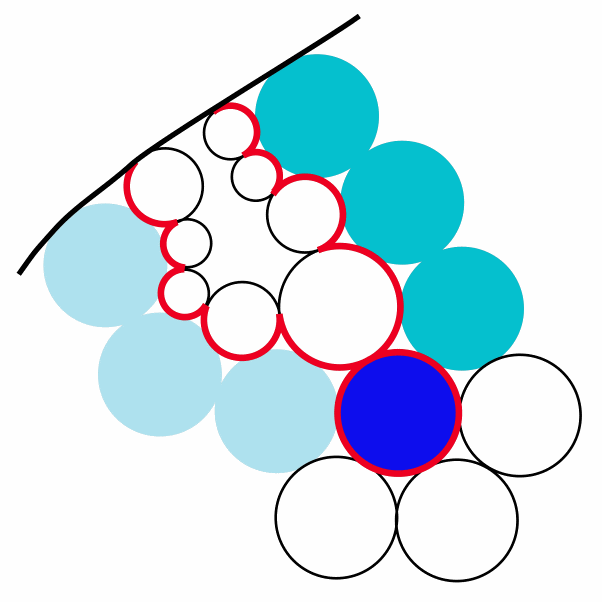}
\fontsize{12pt}{14pt}\selectfont
\put(46,92){\makebox(0,0)[rb]{$S_1$}}
\put(9,65){\makebox(0,0)[rb]{$S_2$}}
\put(22,60){\makebox(0,0)[cc]{$_{_-}$}}
\put(11.5,51.5){\makebox(0,0)[cc]{$_{_+}$}}
\put(48,78){\makebox(0,0)[cc]{$_{_+}$}}
\put(56,87){\makebox(0,0)[cc]{$_{_-}$}}
\put(30,42){\makebox(0,0)[cc]{$_{_-}$}}
\put(50,37){\makebox(0,0)[cc]{$_{_-}$}}
\put(21,31){\makebox(0,0)[cc]{$_{_+}$}}
\put(40,26){\makebox(0,0)[cc]{$_{_+}$}}
\put(60,62){\makebox(0,0)[cc]{$_{_+}$}}
\put(70,50){\makebox(0,0)[cc]{$_{_+}$}}
\put(71,71){\makebox(0,0)[cc]{$_{_-}$}}
\put(82,50){\makebox(0,0)[cc]{$_{_-}$}}
\put(30,10){\makebox(0,0)[cc]{$G^-$}}
\put(88,78){\makebox(0,0)[cc]{$G^-$}}
\end{overpic}
\hspace{0.03\textwidth}
\begin{overpic}{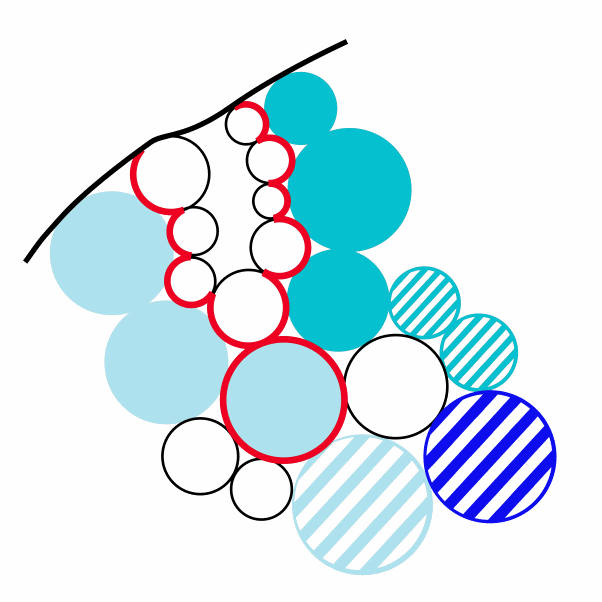}
\fontsize{12pt}{14pt}\selectfont
\put(47,34){\makebox(0,0)[cc]{$D_u$}}
\end{overpic}
\caption{Constructing crosscuts from one slit (left) and two slits (middle, right)}
\label{fig.CutFromSlits}
\end{center}
\end{figure}
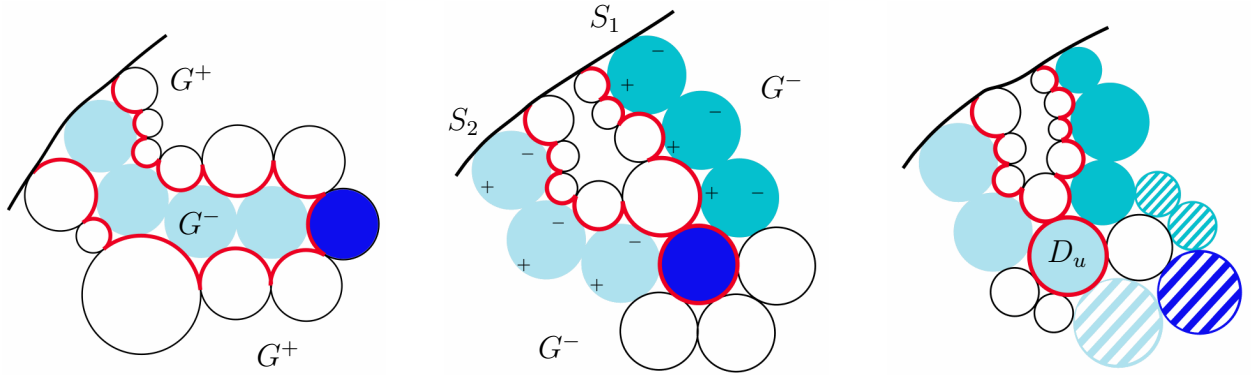
\vspace{-4mm}
\noindent

A crosscut $L$ can also be constructed from glueing two slits $S_1$ and $S_2$ with a
common terminal vertex $v$. This procedure is somewhat more complicated, in particular
when the
``right side'' of $S_1$ is close to the ``left side'' of $S_2$. In those
cases we cannot glue the cuts at their common terminal vertex $v$, since then the
resulting edge sequence $L$ would contain some edges more than once. Instead we modify
the procedure by glueing $S_1$ and $S_2$ at some appropriately chosen vertex $u$ in
$S_2$ or $S_1$ which has a neighbor in $S_1$ or $S_2$, respectively.
Figure~\ref{fig.CutFromSlits} (middle, right) illustrates
the result, showing an associated circle packing and the related maximal crosscuts.

\begin{lem}[]\label{lem.TwoSlitsToCut}
Let $S_1=(v_{1},\ldots,v_{t},v)$ and $S_2=(w_{1},\ldots,w_{s},v)$ be slits in
$K$ with $S_1\cap S_2=\{v\}$. Assume further that
$E^+_{S_1}(v_1)\cap E^-_{S_2}(w_1)=\emptyset$.
Then there exists a combinatoric crosscut $L$ and a vertex
$u\in(S_1\cup S_2)\cap U_L^+$
such that
\begin{equation} \label{eq.CombRedSLM}
S_L^-=\big(w_{1},w_{2},\ldots,w_{\sigma},u_1,\ldots,u_k,
v_\tau,v_{\tau-1},\ldots,v_1\big), \qquad
1\le \tau\le t,\ 1\le \sigma\le s,
\end{equation}
where
$\big(w_{\sigma},u_1,\ldots,u_k,v_\tau\big)$
is a (positively oriented) chain of neighbors of $u$.
\end{lem}

\begin{proof}
We set $v_{t+1}:=v$ and $w_{s+1}:=v$. Let $i$ be the smallest number in
$\{1,...,t+1\}$ for which $E_{S_1}^+(v_i)$ contains an edge $e(v_i,w)$ with
$w\in S_2$. Then let $j$ be the smallest number in $\{1,...,s+1\}$ for which
$E_{S_2}^-(w_j)$ contains an edge $e(w_j,v_i)$.
If $i\neq 1$ and $j\neq s+1$ we set $\tau:=i-1$, $\sigma:=j$ and $u:=v_i$.
If $i\neq 1$ but $j=s+1$, then $i=t$ must hold (otherwise $v$ would
have more then one neighbor in $S_1$), and we set $\tau:=t$, $\sigma:=s$ and
$u:=v$. If $i=1$ we set $\tau:=1$, $\sigma:=j-1$ and $u:=w_j$. In the last case
we have $j>1$, since otherwise $i=j=1$ would contradict the assumption
$E^+_{S_1}(v_1)\cap E^-_{S_2}(w_1)=\emptyset$.

In every case $1\leq\tau\leq t$ and $1\leq\sigma\leq s$ hold, and $u$ is well
defined. We now build $L$ as the concatenation of the edge sequences
\[
E_{S_2}^-(w_1),\ \ldots,\ E_{S_2}^-(w_{\sigma}),\quad
E^*(u), \quad
E_{S_1}^+(v_{\tau}),\ \ldots,\ E_{S_1}^+(v_1),
\]
where $E^*(u)=\big(e(u,w_{\sigma}),e(u,u_1),\ldots,e(u,u_k),e(u,v_{\tau})\big)$
is the negatively oriented chain of edges in the set
$\{e'\in E(v): e(u,w_{\sigma})\leq e'\leq e(u,v_\tau)\}$.

Because $S_1,S_2$ are slits, all edges in the ``$E_{S_1}^+$-part'' and in the
``$E_{S_2}^-$-part'' of $L$ are pairwise different.
Furthermore, it cannot happen that such an edge is contained in both parts
(according to the definition of $u$), or that it belongs to $E^*(u)$ (by definition
of $E^*(u)$). Hence, $L$ satisfies condition~(i) of the crosscut definition
(page~\pageref{def.Cuts.Cross}).

Condition (ii) can easily be verified and (iv) is trivial. In order to prove
(iii) we assume that three edges of $L$ form a face of $K$. By definition of $u$,
the sequence $(w_1,w_2,...,w_{\sigma},u,v_{\tau},...,v_2,v_1)$
divides $K$ into two parts $K_1,K_2$. All edges of the ``$E_{S_1}^+$-part'' and
of the ``$E_{S_2}^-$-part'' have exactly one vertex lying
in $S_1^0\cup S_2^0$ and one in $K_1$, so three of them can never form a face of
$K$. All edges of $E^*(u)\setminus\{e(u,v_{\tau}),e(u,w_{\sigma})\}$ have
exactly one vertex lying in $S_1^0\cup S_2^0$ and one in $K_2$, so again three of
them can never form a face of $K$. The only remaining edges are
$e(u,v_{\tau}),e(u,w_{\sigma})$, but two edges cannot form a face, and a combination
of edges from more than one of the three distinguished edge types can clearly never
form a face. Hence, $L$ is a crosscut with $u\in(S_1\cup S_2)\cap U_L^+$,
and $S_L^-$ has the form {\rm\eqref{eq.CombRedSLM}}.
\end{proof}

The operation described in the proof is well defined by the
slits $S_1$ and $S_2$, and will be referred to as \emph{reflected concatenation}
$S_1\circleddash S_2$ of $S_1$ with $S_2$. It delivers a crosscut $L$, a vertex
$u$, and the reduced slits $S_1^0,S_2^0$. Note that the reflected concatenation
is not commutative.\label{def.reflect}

\subsection{Subdivision by Disk Chains} \label{sec.Subdivision}

Let $v_\beta$ be an arbitrary accessible boundary vertex. In this section we
describe an approach which allows us to apply Lemma~\ref{lem.Upper.Goal}
recursively, until we find a slit $S$ with initial vertex $v_\beta$ such that
$D_v'=D_v$ for all $v\in S$, so especially $D'_{v_\beta}=D_{v_\beta}$. During
this procedure we construct a sequence of crosscuts $L_j$ such that $V_{L_j}^+$
contains $v_\beta$ and the number of elements in $V_{L_j}^+$ is strictly
decreasing for increasing $j$. This procedure will be crucial for proving the
following lemma, and finally Theorem~\ref{thm.CircRigid}.

\smallskip\par

\begin{lem}[]\label{lem.BndUni}
Let $(\mathcal{P},\mathcal{P}')$ be an admissible pair with complex $K$, interior
alpha vertex $v_\alpha$ and crosscut $L$. Then $D'_v=D_v$ for all accessible
boundary vertices $v\in\partial V^*$.
\end{lem}

\begin{proof}
To begin with, let $S_0=(v_1,\ldots,v_{s},v_\alpha)$ be a slit according
to Lemma~\ref{lem.Upper.Goal}. Let $v_\beta$ be an accessible boundary vertex.
If $v_1=v_\beta$ then $D_\beta'=D_\beta$ and we are done. So let us assume that
$v_\beta\notin S_0$.

By Lemma~\ref{lem.SlitToCut} there exists a crosscut $L_1$ such that
$S_{L_1}^-=(v_1,\ldots,v_{s-1},v_s,v_{s-1}\ldots,v_1)$ and $v_\alpha\in S_{L_1}^+$.
Applying Lemma~\ref{lem.Upper.Goal} again, but now with respect to the crosscut
$L_1$, we get another slit $S_1=(w_{1},\ldots,w_{t},v_\alpha)\subset V_{L_1}^+$,
such that $D_v'=D_v$ for all $v\in S_1$.
If $v_{1}=v_\beta$ then $D_\beta'=D_\beta$ and we are done. So suppose that
$v_\beta\notin S_1$.
\smallskip\par

\noindent
The three boundary vertices $v_1$, $w_1$ and $v_\beta$ are pairwise different,
and we assume, without loss of generality, that they are oriented such that
$w_1<v_\beta<v_1$. This ensures the condition
$E^+_{S_1}(v_1)\cap E^-_{S_2}(w_1)=\emptyset$ of Lemma~\ref{lem.TwoSlitsToCut},
because otherwise $v_\beta$ could be either accessible or a boundary
vertex, but not both. Since, except $v_\alpha$, all vertices of $S_0$ belong
to $V_{L_0}^-$, we have $S_0\cap S_1 = \{v_\alpha\}$.
Consequently, the reflected concatenation $S_0\circleddash S_1$ of $S_0$ with
$S_1$ is well defined. It delivers a crosscut $L_2$, a vertex $v_{\alpha_2}$,
and reduced slits $S_2^- \subset S_0, S_2^+ \subset S_1$ with common terminal
vertex $v_{\alpha_2}$.
By Lemma~\ref{lem.TwoSlitsToCut} the vertex $v_{\alpha_2}$ belongs to $S_1$
or $S_2$ and the
set $U_{L_2}^-$ of lower neighbors of $L_2$ consists solely of elements of
$S_0 \cup S_1$ and of (lower) neighbors of $v_{\alpha_2}$. Since $D_v'=D_v$ for all
$v\in S_0 \cup S_1$, this implies that $L_2$ is an admissible crosscut for
$(\mathcal{P},\mathcal{P}')$.
Moreover, the order of $S_0$ and $S_1$ in the reflected concatenation has been
chosen such that $v_\beta$ belongs to $V_{L_2}^+$.

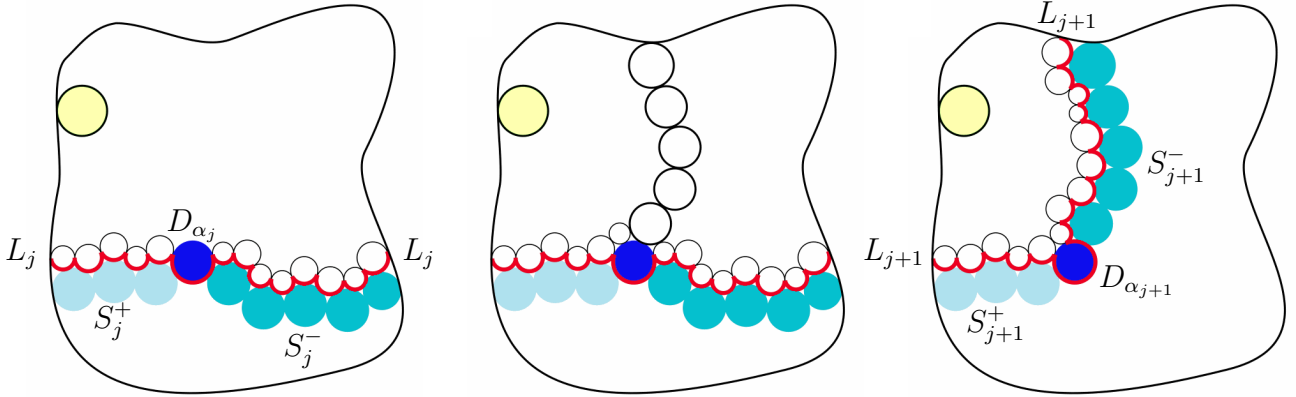
\begin{figure}[H]
\begin{center}
\begin{overpic}{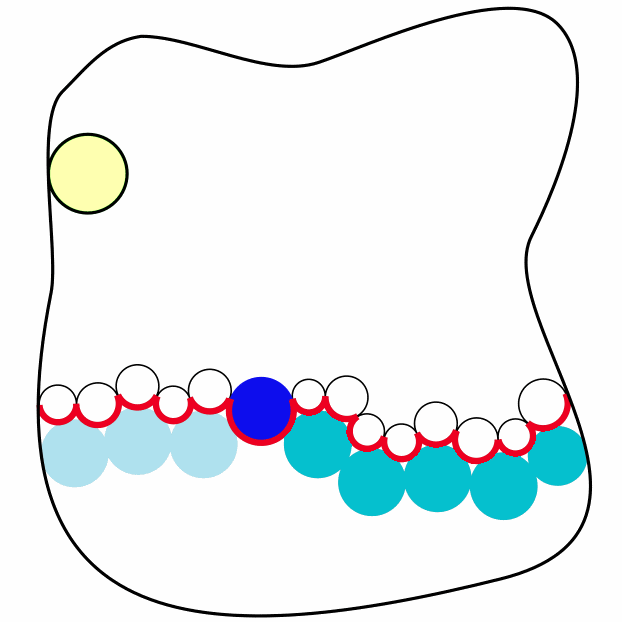}
\fontsize{12pt}{14pt}\selectfont
\put(22,19){\makebox(0,0)[cc]{$S_j^+$}}
\put(42,43){\makebox(0,0)[cc]{$D_{\alpha_j}$}}
\put(70,12){\makebox(0,0)[cc]{$S_j^-$}}
\put(95,36){\makebox(0,0)[lc]{$L_j$}}
\put(3,36){\makebox(0,0)[rc]{$L_j$}}
\end{overpic}
\hfill
\begin{overpic}{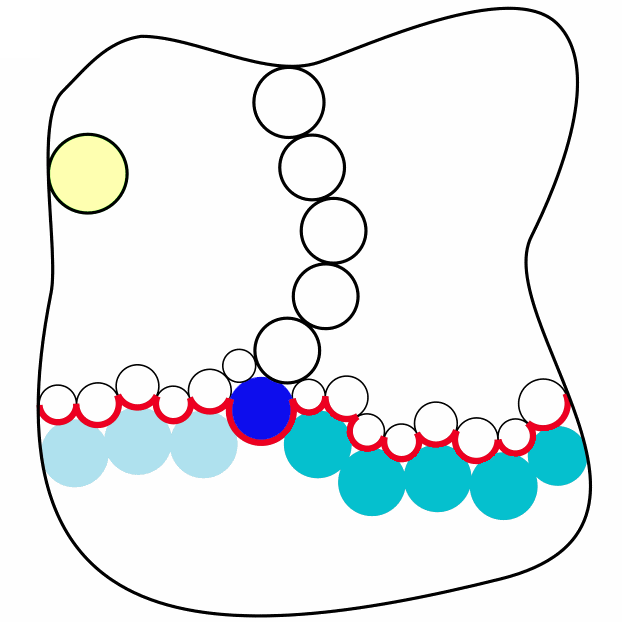}
\fontsize{12pt}{14pt}\selectfont
\end{overpic}
\hfill
\begin{overpic}{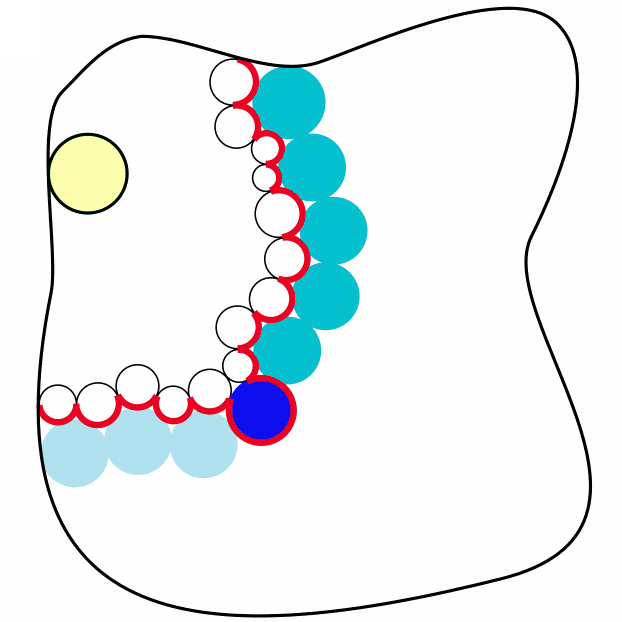}
\fontsize{12pt}{14pt}\selectfont
\put(22,18){\makebox(0,0)[cc]{$S_{j+1}^+$}}
\put(60,58){\makebox(0,0)[lc]{$S_{j+1}^-$}}
\put(48,29){\makebox(0,0)[lu]{$D_{\alpha_{j+1}}$}}
\put(4,36){\makebox(0,0)[rc]{$L_{j+1}$}}
\put(40,92){\makebox(0,0)[cb]{$L_{j+1}$}}
\end{overpic}
\caption{Construction of the crosscut $L_{j+1}$ from $L_j$}
\label{fig.ChainSubdiv2}
\end{center}
\end{figure}
\vspace{-4mm}

\noindent
The general step of the procedure is as follows. Assume that we already have
an admissible crosscut $L_j$, the alpha vertex $v_{\alpha_j}$, and the reduced
slits $S_j^-$ and $S_j^+$, such that $v_\beta\in V_{L_j}^+$. Denoting by
$v_j^-$ and $v_j^+$ the initial vertices of $S_j^-$ and $S_j^+$, respectively,
we may assume that $v_j^-<v_\beta<v_j^+$, which will again be essential to ensure
the special condition of Lemma~\ref{lem.TwoSlitsToCut}.

Applying Lemma~\ref{lem.Upper.Goal}, we get a new slit $S_j\subset V_{L_j}^+$,
such that $S_j^-$, $S_j$ and $S_j^+$ are pairwise disjoint, except at their
common terminal vertex $v_{\alpha_j}$, and $D_v'=D_v$ for all $v\in S_j$.

If $v_\beta\in S_j$ we are done. Otherwise we either have
$v_j^-<v_\beta<v_j$ or $v_j<v_\beta<v_j^+$.
In the first case we build the reflected concatenation $S_j^-\circleddash S_j$,
in the second case we form $S_j\circleddash S_j^+$.
The result is a new crosscut $L_{j+1}$, a corresponding alpha-vertex
$v_{\alpha_{j+1}}$, and reduced slits $S_{j+1}^-$, $S_{j+1}^+$.

If follows directly from the construction of the reflected concatenation that
$v_{\alpha_{j+1}},v_\beta\in V_{L_{j+1}}^+$. Moreover,
$v_{\alpha_{j+1}}\in S_j^-$, and hence $D_{\alpha_{j+1}}'=D_{\alpha_{j+1}}$.
To see that $L_{j+1}$ is admissible for the pair $(\mathcal{P},\mathcal{P}')$
it remains to prove that $D_v'\subset G_{L_{j+1}}^-$ for all
$v\in U_{L_{j+1}}^-$.

By Lemma~\ref{lem.TwoSlitsToCut} the set $U_{L_{j+1}}^-$ of lower neighbors
of $L_{j+1}$ consists solely of elements of $S_j^- \cup S_j^+$ and of (lower)
neighbors of $v_{\alpha_{j+1}}$. Since $D_v'=D_v$ for all
$v\in S_j^- \cup S_j^+ \cup \{v_{\alpha_{j+1}}\}$, and
$D_v\subset G_{L_{j+1}}^-$ for all $v\in U_{L_{j+1}}^-$, the assertion follows.

The number of elements in $V_{L_j}^+$ is strictly decreasing in every step, and
hence the procedure must come to end. This can only happen if
$v_\beta\in S_{j^*}$ for some $j^*\in\mathbb{N}$. Because $D_v'=D_v$ for all
$v\in S_j$ with $j\leq j^*$, we have shown $D'_{v_\beta}=D_{v_\beta}$.
\end{proof}

\bigskip

Now we are close to the end. By Lemma~\ref{lem.solid3} the kernel $K^*$ is a strongly
connected complex with vertex set $V^*$. Since we have shown that $D_v'=D_v$
for all boundary vertices $v\in \partial V^*$ of $K^*$, and every boundary vertex
of $K^*$ is also a boundary vertex of $K$ (that is $\partial V^*=V^*\cap\partial V$),
Theorem~11.6 in Stephenson~\cite{SteBook} (on the uniqueness of a locally univalent
packing with presribed combinatorics and given radii of boundary circles) tells us
that $D_v'=D_v$ for all $v\in V^*$, which is the assertion of Theorem~\ref{thm.CircRigid}.

\section{Concluding Remarks} \label{sec.Remarks}

All proofs in this paper work with (simple) geometric or combinatoric arguments,
alone in the very last step we had recourse to a theorem established in the
literature. For purists we mention that even this could have been avoided, at
the expense of adding a few pages to this rather longish text.

Theorem~\ref{thm.CircRigid} can be interpreted as uniqueness result for (the range
packing of) discrete conformal mappings. Here is a simple version:

\begin{thm}\label{thm.conf1}
Suppose that two univalent packings $\mathcal{P}$ and $\mathcal{P}'$ for $K$ fill $G$.
If $D'_\alpha$ and $D_\alpha$ have the same center, and if $D'_\beta\subset D_\beta$
for some boundary vertex $v_\beta$, then $D_v'=D_v$ for all vertices $v\in V^*$.
\end{thm}

The proof follows immediately from Theorem~\ref{thm.CircRigid} applied to the maximal
crosscut which separates the disk $D_\beta$ from the rest of the packing $\mathcal{P}$
(see the leftmost image of Figure~\ref{fig.Beta}). The condition $D_\beta'\subset D_\beta$
can even be relaxed, it suffices to require that $D'_\beta$ lies in the lower domain
$G_-$ with respect to this crosscut (see the second image of Figure~\ref{fig.Beta}).
Note that both figures show the packing $\mathcal{P}$ and a single disk $D'_\beta$
of $\mathcal{P}'$ in $G_-$.

We point out that the condition $D_\beta'\subset G_-$ is always satisfied
(possibly after exchanging the roles of $\mathcal{P}$ and $\mathcal{P}'$), if the
packings are normalized so that $D'_\beta$ and $D_\beta$ \emph{touch the boundary
$\partial G$ in a generalized sense} at the same \emph{regular point} (or,
more generally, at the same \emph{regular prime end}). Without explaining these
concepts here (see~\cite{KrgWgt}), we mention that a point which lies
on a smooth subarc of $\partial G$ is always regular, while a point at a re-entrant
corner fails to be regular. The two pictures on the right of Figure~\ref{fig.Beta}
illustrate that uniqueness of domain-filling circle packings may be violated in that
case. Both displayed packings $\mathcal{P}$ and $\mathcal{P}'$ fill
a Jordan domain $G$, $D_\alpha$ and $D_\alpha'$ have the same center, and $D_\beta$
and $D'_\beta$ touch $\partial G$ at the same point. While this type of normalization
implies uniqueness of classical conformal mappings, the corresponding circle packings
$\mathcal{P}$ and $\mathcal{P}'$ are completely different.

\begin{figure}[H]
\begin{center}
\begin{overpic}{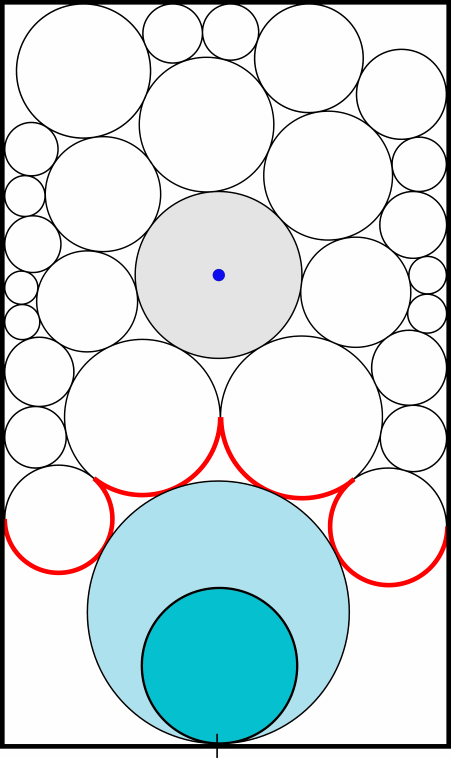}
\fontsize{12pt}{14pt}\selectfont
\put(30,68){\makebox(0,0)[cc]{$D_\alpha$}}
\put(28,12){\makebox(0,0)[cc]{$D'_\beta$}}
\put(28,28){\makebox(0,0)[cc]{$D_\beta$}}
\end{overpic}
\hfill
\begin{overpic}{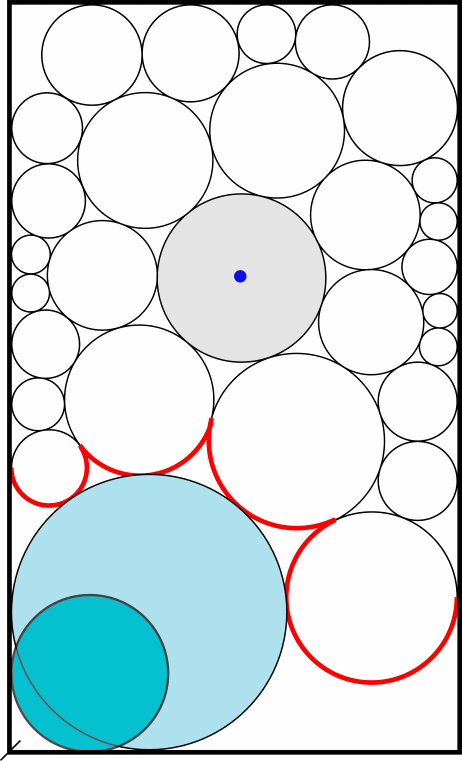}
\fontsize{12pt}{14pt}\selectfont
\put(33,68){\makebox(0,0)[cc]{$D_\alpha$}}
\put(25,25){\makebox(0,0)[cc]{$D_\beta$}}
\put(12,12){\makebox(0,0)[cc]{$D'_\beta$}}
\end{overpic}
\hfill
\begin{overpic}{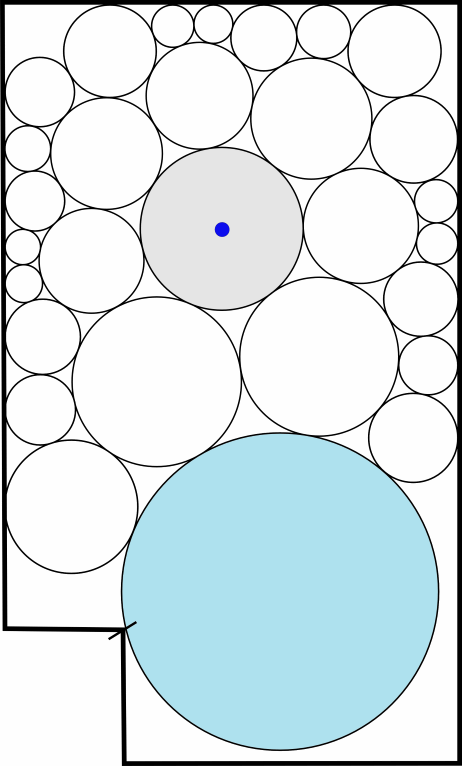}
\fontsize{12pt}{14pt}\selectfont
\put(28,75){\makebox(0,0)[cc]{$D_\alpha$}}
\put(37,23){\makebox(0,0)[cc]{$D_\beta$}}
\end{overpic}
\hfill
\begin{overpic}{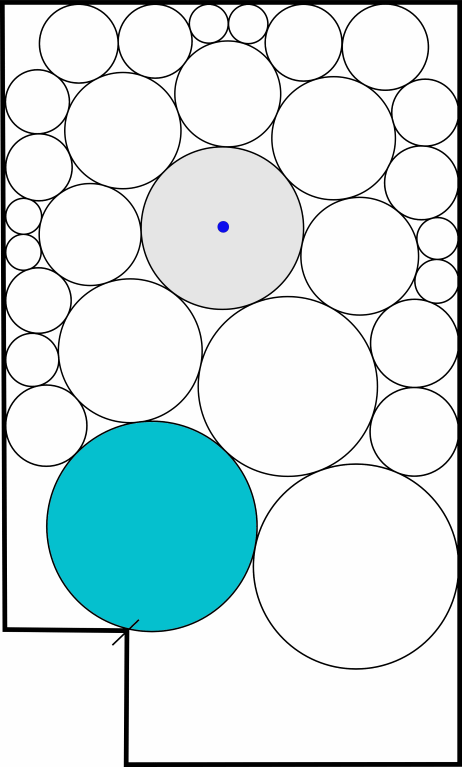}
\fontsize{12pt}{14pt}\selectfont
\put(28,75){\makebox(0,0)[cc]{$D'_\alpha$}}
\put(20,32){\makebox(0,0)[cc]{$D'_\beta$}}
\end{overpic}
\caption{Applications of Theorem~\ref{thm.CircRigid} to discrete conformal mapping}
\label{fig.Beta}
\end{center}
\end{figure}
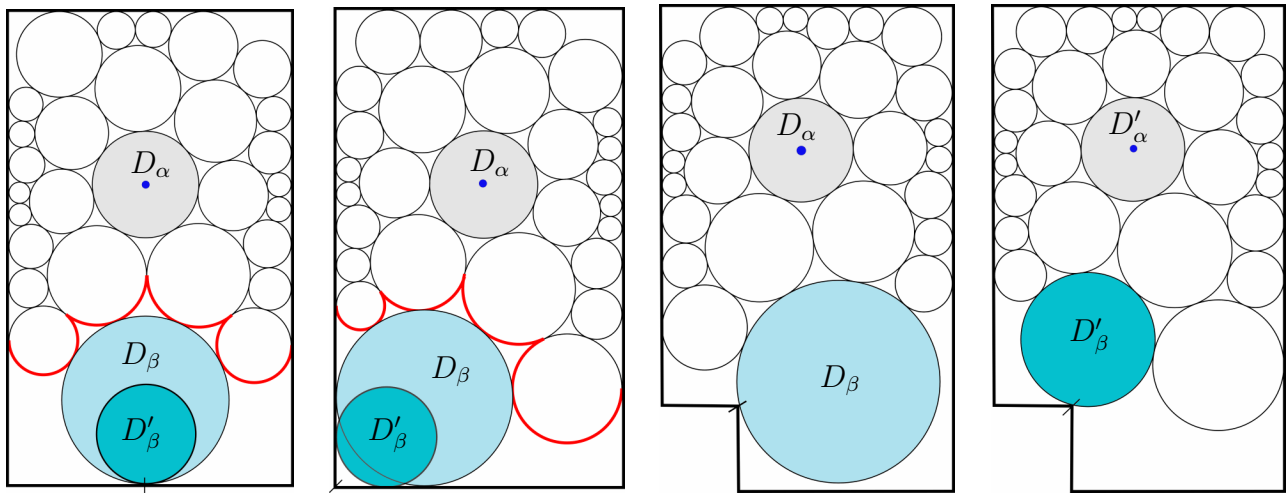
\vspace{-4mm}

\noindent
We further mention that for domain-filling circle packings $\mathcal{P}$ and $\mathcal{P}'$
the assertions of Theorem~\ref{thm.CircRigid} and Theorem~\ref{thm.conf1} can be
strengthened to $D'_v=D_v$ for all $v\in V$, using the results of our forthcoming
paper~\cite{KrgWgt}.

In the general setting of Theorem~\ref{thm.CircRigid}, a complete description
which disks are uniquely determined by a crosscut seems not to be known. The figures
below show some examples. The accessible disks are depicted in darker colors, the
alpha-disk is the darkest one. By Theorem~\ref{thm.CircRigid} these disks are uniquely
determined (rigid) by the crosscut, but the rigid part also comprises the non-accessible
disks shown in brighter color.

\begin{figure}[H]
\begin{center}
\includegraphics{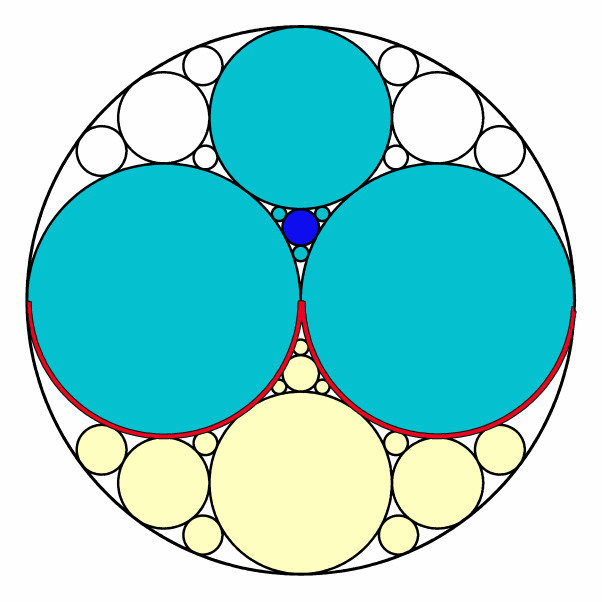}
\hfill
\includegraphics{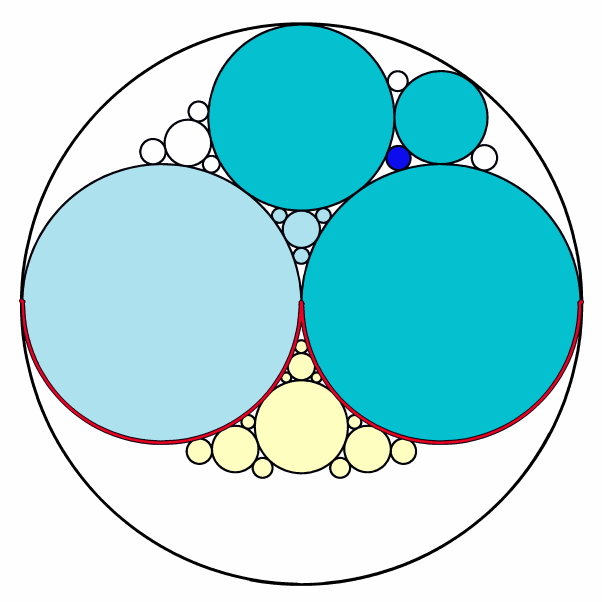}
\hfill
\includegraphics{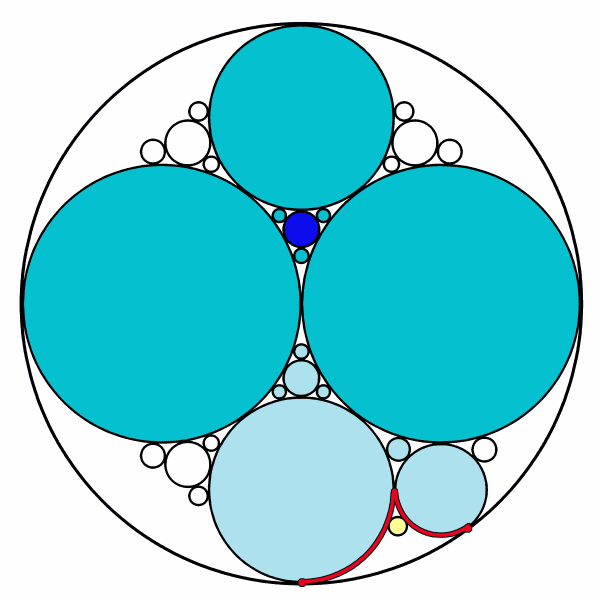}
\caption{Rigid configurations of disks in a packing with crosscut}
\label{fig.GenUnique}
\end{center}
\end{figure}

\noindent
The example on the right is of special interest: a short crosscut separates
only one non-accessible disk $D_\beta$ from the alpha-disk. Here the theorem yields
rigidity for the dark (blue) disks, so that $D_\beta$ seems to have some mysterious
"remote action". However, a little thought shows that there is a chain of rigid
disks (depicted in lighter color) which connects the cut with the alpha-disk
and acts as ``transmission line''.
\smallskip\par

\bigskip\par\noindent

Isn't it wonderful that simple circles can form such fascinating structures?
\bigskip

\section*{Glossary}

\begin{tabbing} \hspace*{20mm} \=   \kill
$\circleddash$,\ $S_1\circleddash S_2$ \> reflected concatenation of slit
$S_1$ with slit $S_2$; p.~\pageref{def.reflect} \\
$\langle u,v,w \rangle $ \> oriented face of $K$ with vertices $u$,$v$ and $w$;
p.~\pageref{def.VEF} \\
$\langle u,v \rangle$ \> oriented edge of $K$ from vertex $u$ to vertex $v$;
p.~\pageref{def.VEF} \\
$\alpha_i, \alpha$ \> special Jordan arcs connecting $y_i^-$ and $y_i^+$, and their
concatenation; p.~\pageref{lem.AlphaArc} \\
$B(v)$ \> the flower of the vertex $v$, a subcomplex of $K$; p.~\pageref{def.flower} \\
$c(u,v)$ \> contact point of the disks $D_u$ and $D_v$,
$c(u,v)=\overline{D}_u \cap\overline{D}_v$; p.~\pageref{def.contact}\\
$c_k^-, c_k^+$ \> contact points of boundary disk $D_k$ with $D_{k-1}$ and $D_{k+1}$,
respectively; p.~\pageref{def.arcs2}\\
$D$ \> union of all disks in $\mathcal{P}$; p.~\pageref{def.arcs} \\
$D^*$ \> carrier of $\mathcal{P}$; p.~\pageref{def.carrier} \\
$D_k, D_k'$ \> boundary disks in $\mathcal{P}$ and $\mathcal{P}'$, respectively;
p.~\pageref{def.delta} \\
$D_v, D_v'$ \> disks in $\mathcal{P}$ and $\mathcal{P}'$, respectively;
p.~\pageref{def.Pack.Pack} \\
$\partial $ \> boundary operator, applied to various objects \\
$\delta(p,q)$ \> positively oriented open circular arc from $p$ to $q$ on $\partial D$;
p.~\pageref{def.arcs} \\
$\delta[p,q]$ \> positively oriented closed circular arc from $p$ to $q$ on $\partial D$;
p.~\pageref{def.arcs} \\
$\delta(c^-_k,c^+_k)$ \>  exterior boundary arc of $D_k$; p.~\pageref{def.arcs2} \\
$\delta(c^+_k,c^-_k)$ \>  interior boundary arc of $D_k$; p.~\pageref{def.arcs2} \\
$\delta_k$ \>  smallest subarc of $\delta[c^-_k,c^+_k]$ which contains $G_k$;
p.~\pageref{def.delta} \\
$E_S$ \> the edge sequence of the slit $S$; p.~\pageref{def.Slit} \\
$E$ \> the set of edges of the complex $K$; p.~\pageref{def.VEF} \\
$\partial E$ \> boundary edges of the complex $K$; p.~\pageref{def.BndVE} \\
$E(v)$ \> the (cyclically ordered) sequence of edges adjacent to $v\in V$;
p.~\pageref{def.VEF} \\
$E_L^\pm(v)$ \> sequences of upper and lower accompanying edges of the crosscut L;
p.~\pageref{def.Accompany.Edge} \\
$E_S^\pm(v)$ \> sequences of edges adjacent to a vertex $v$ in a slit $S$;
p.~\pageref{def.ESv+-} \\
$E_S^\pm$ \> sequences of left and right neighbor edges of slit $S$, respectively;
p.~\pageref{def.ES+-} \\
$e(u,v)$ \> non-oriented edge between vertices $u$ and $v$; p.~\pageref{def.VEF} \\
$e_j$ \> edges in a crosscut, $L=(e_0,e_1,\ldots,e_l)$; p.~\pageref{def.Cuts.Cross} \\
$e_j^-$, $e_j^+$ \> lower and upper accompanying edges of the crosscut $L$, respectively;
p.~\pageref{def.Accompany.Edge} \\
$\eta_k, \eta$ \> segments connecting the centers of $D_k$ and $D_{k+1}$ and their
concatenation; p.~\pageref{def.eta} \\
$F$ \> set of faces of the complex $K$; p.~\pageref{def.VEF} \\
$f(u,v,w)$ \> non-oriented face with vertices $u,v$ and $w$; p.~\pageref{def.VEF} \\
$G$ \> Jordan domain to be filled with $\mathcal{P}$; p.~\pageref{def.G} \\
$G_L^-, G_L^+$ \> lower and upper domains of $G$ with maximal crosscut $J_L^+$,
$G_L^-=\Omega$; p.~\pageref{def.omega} \\
$G_k$ \> set of contact points of $D_k$ with $\partial G$; p.~\pageref{def.delta} ,
$G_k:=\overline{D}_k \cap \partial G$ \\
$g_k^-,g_k^+$ \> first and the last contact point of $D_k$ with $\partial G$;
p.~\pageref{def.delta} \\
$I_k$ \> boundary interstice between $D_k$ and $D_{k+1}$;
p.~\pageref{page.boundary.interstice} \\
$I(u,v,w)$ \> interstice between the disks $D_u,D_v$ and $D_w$;
p.~\pageref{def.interstice} \\
$J_L^0$ \>  polygonal (geometric) crosscut in $G$ for (combinatoric) crosscut $L$ in $K$;
p.~\pageref{def.PolyCrosscut} \\
$J_L^+$ \>  maximal `crosscut', the upper boundary of the lower domain $G_L^-$,
$J_L^+=\omega$; p.~\pageref{def.MaxCrosscut} \\
$K$ \> simplicial 2-complex, combinatorial disk, finite triangulation, $K=(V,E,F)$;
p.~\pageref{def.complex} \\
$K^*$ \> kernel of $K$, largest sub-complex of $K$ with vertex set $V^*$;
p.~\pageref{def.Access} \\
$L$ \> combinatorial crosscut, sequence of edges in $K$; p.~\pageref{def.Cuts.Cross} \\
$l(i)$ \> smallest label $k$ of prime end set $\omega_k^*$ associated with
$\nu_i$; p.~\pageref{def.number}\\
$M$, $M(\mu)$ \> loop of a multiple loner $v_\mu$, a sequence of edges;
p.~\pageref{def.Loop} \\
$\Omega$ \> lower subdomain of $G$ with respect to a maximal crosscut, $\Omega=G_L^-$;
p.~\pageref{def.omega} \\
$\omega$ \> upper boundary of lower domain $\Omega$, concatenation of the $\omega_i$,
maximal crosscut; p.~\pageref{def.omega} \\
$\omega^*$ \> prime ends of $\Omega$ associated with $\omega$\;
p.~\pageref{def.PrimeEnds} \\
$\omega_i$ \> circular subarcs of $\omega$ in between its turning points;
p.~\pageref{def.omega} \\
$\omega_i^*$ \> classes of prime ends associated with the arcs $\omega_i$;
p.~\pageref{def.PrimeEnds} \\
$\nu_i, \pi_i$ \> negatively and positively oriented arcs on $\partial D$ from
$y_i^-,y_i^+$ to $\omega$, respectively; p.~\pageref{def.NuPi}  \\
$\nu_i^+, \pi_i^+$ \> terminal points of the arcs $\nu_i, \pi_i$, respectively;
p.~\pageref{lem.Lone.NuPiToOmega} \\
$\nu_i^*, \pi_i^*$ \> prime ends of $\Omega$ associated with $\nu_i, \pi_i$, respectively;
p.~\pageref{def.Prime.NuPi} \\
$\mathcal{P}$ \> a univalent circle packing for $K$ filling $G$;
p.~\pageref{def.Pack.Pack}, p.~\pageref{def.Chap} \\
$\mathcal{P}'$ \> a univalent circle packing for $K$ in $G$; p.~\pageref{def.Chap} \\
$r(i)$ \> largest label $k$ of prime end set $\omega_k^*$ associated with
$\pi_i$; p.~\pageref{def.number}\\
$S$ \> combinatoric slit, a sequence of vertices; p.~\pageref{def.Slit} \\
$S_L^-$, $S_L^+$ \> sequences of lower and upper accompanying vertices of $L$,
respectively; p.~\pageref{def.Accompany.Vertex} \\
$t_i$ \> turning points of the upper boundary $\omega$, cusps of $\Omega$;
p.~\pageref{def.omega} \\
$U_L^-$, $U_L^+$ \> sets of lower and upper neighbors of $L$, respectively,
$U_L^-\subset V_L^-, U_L^+\subset V_L^+$; p.~\pageref{def.lower.upper} \\
$U_M$ \> sequence of the vertices in $V_M$ for a loop $M$; p.~\pageref{def.Loop} \\
$V$ \> vertex set of the complex $K$; p.~\pageref{def.VEF} \\
$V^*$ \> the set of all accessible vertices of $K$; p.~\pageref{def.Access} \\
$\partial V$ \> boundary vertices of the complex $K$; p.~\pageref{def.BndVE} \\
$V_L^-,V_L^+$ \> lower and upper vertices of $K$ with crosscut $L$, respectively,
subsets of $V$; p.~\pageref{def.lower.upper} \\
$V_M$ \> set of all vertices met by a loop $M$; p.~\pageref{def.Loop} \\
$v_\alpha$ \> alpha vertex of $K$, a distinguished interior vertex;
p.~\pageref{def.Access} \\
$v(i)$ \> vertex of the disk which contains the circular arc $\omega_i$,
$v(i)\in U_L^+$; p.~\pageref{def.Yi} \\
$x_k, X$ \> contact points of upper with lower disks in $\mathcal{P}$,
the set of all $x_k$; p.~\pageref{def.X} \\
$X_i$ \> sets of contact points $x_k$ on $\omega_i$, $X_i \subset X$;
p.~\pageref{def.XYi} \\
$y_-,y_+$ \> initial point and terminal point of $\alpha$, respectively;
p.~\pageref{def.yz-} \\
$y_k, Y$ \> contact points of upper with lower disks in $\mathcal{P}'$,
the set of all $y_k$; p.~\pageref{def.Y} \\
$y_i^-,y_i^+$ \> minimal and maximal element of $Y_i$, respectively;
p.~\pageref{def.Yi}\\
$Y_i$ \> sets of contact points $y_k$ with $x_k\in \omega_i$, $Y_i \subset Y$;
p.~\pageref{def.XYi} \\
$z_-,z_+$ \> terminal points of $\nu_1$ and $\pi_n$, respectively; p.~\pageref{def.yz-} \\
$z_k$ \> shifted contact points when $y_k$ is critical; p.~\pageref{def.Z} \\

\end{tabbing}


\begin{thebibliography}{99}
\bibitem{BauSteWgt}
Bauer, D., Stephenson, K., Wegert, E.:
Circle packings as differentiable manifolds.
Contrib. Algebra Geom. {\bf 53} (2012) 399-420.

\bibitem{BeaSte}
Beardon, A.F., Stephenson, K.:
The Uniformization Theorem for Circle Packings.
Indiana Univ.\ Math.\ Journal {\bf 39} (1990) 1383-1425.

\bibitem{GolBook}
Golusin, G.M.:
Geometrische Funktionentheorie.
Berlin, Dt. Verl. d. Wissenschaften 1957.

\bibitem{He}
He, Z.-X.; Schramm, O.:
On the convergence of circle packings to the Riemann map.
Inventiones mathematicae {\bf 125} (1996) 285-305.

\bibitem{Hen}
Henle, M.:
A Combinatorial Introduction to Toplogy.
Dover Publ. 1979.

\bibitem{Koebe}
Koebe, P.:
Kontaktprobleme der konformen Abbildung.
Ber.\ S\"achs.\ Akad.\ Wiss.\ Leipzig, Math.-Phys. Kl. 88, Leipzig 1936:
141-164.

\bibitem{KrgWgt}
Krieg, D.; Wegert, E.:
Domain-filling circle packings.
(in preparation)

\bibitem{PomBook}
Pommerenke, Ch.:
Boundary Behaviour of Conformal Maps.
Berlin, Springer 1992.

\bibitem{Rodin}
Rodin, B.:
Schwarz's lemma for circle packings.
Inventiones mathematicae {\bf 89} (1987) 271-289.

\bibitem{RodSul}
Rodin, B.; Sullivan, D.:
The convergence of circle packings to the Riemann Mapping.
J. Differential Geometry {\bf 89} (1987) 349-360.

\bibitem{Schramm1}
Schramm, O.:
Combinatorically prescribed packings and applications to conformal and
quasiconformal maps.
Ph. D. thesis., Princeton 1990.

\bibitem{Schramm2}
Schramm, O.:
Existence and uniqueness of packings with specified combinatorics.
Israel J.\ of Math.\ {\bf 73} (1991), 321-341.

\bibitem{SteBook}
Stephenson, K.:
Introduction to Circle Packing.
Cambridge Univ.\ Press, Cambridge 2005.

\bibitem{Thur}
Thurston, W.:
The finite Riemann mapping theorem.
Invited talk in the International Symposium at Purdue University on the
occasion of the proof of the Bieberbach conjecture 1985.

\bibitem{WgtKrg}
Wegert, E., Krieg, D.: Incircles of trilaterals.
Contributions to Algebra and Geometry {\bf 55} (2014) 277--287.

\bibitem{WgtRotKra}
Wegert, E., Roth, O., Kraus, D.:
On Beurling's boundary value problem in circle packing.
Complex Variables and Elliptic Equations, {\bf 57} (2012) 397-410.

\end{thebibliography}
\end{document}